\documentclass[aos,preprint]{imsart}
\usepackage{amsmath}
\usepackage{amsthm}
\usepackage{amssymb}
\usepackage{dsfont}
\usepackage[authoryear,round]{natbib}
\usepackage[colorlinks=true,linkcolor=blue,citecolor=blue]{hyperref}
\usepackage{hypernat}
\usepackage{verbatim}
\usepackage{textcomp}
\usepackage{enumerate}
\usepackage{graphicx}
\usepackage{bbm}
\usepackage{array}
\usepackage{multirow}

\newcolumntype{H}{>{\setbox0=\hbox\bgroup}c<{\egroup}@{}}

\newcommand{\R}{{\mathbb R}}
\renewcommand{\r}{{\mathbb R}}
\newcommand{\E}{{\mathbb E}}
\renewcommand{\P}{{\mathcal P}}
\renewcommand{\Pr}{{\mathbb P}}
\newcommand{\N}{{\mathbb N}}
\newcommand{\Z}{{\mathbb Z}}

\newcommand{\eps}{\varepsilon}
\newcommand{\B}{\mathcal B}
\newcommand{\X}{\mathcal X}

\newcommand{\dtv}{d_{\text{TV}}}

\arxiv{}

\DeclareMathOperator{\Var}{Var}

\DeclareMathOperator{\supp}{supp}

\DeclareMathOperator{\s}{span}

\newtheorem{theorem}{Theorem}[section]
\newtheorem{proposition}[theorem]{Proposition}
\newtheorem{lemma}[theorem]{Lemma}

\newtheorem{remark}[theorem]{Remark}
\newtheorem*{remark*}{Remark}

\newtheorem*{definition*}{Definition}

\numberwithin{equation}{section}

\newcounter{rcnt}[section]

\newcommand{\rem}[1]{}

\newcounter{desccount}

\newcommand{\descref}[1]{\hyperref[#1]{#1}}

\begin{document}

\begin{frontmatter}

\title{Interactive versus non-interactive locally differentially private estimation: Two elbows for the quadratic functional}
\runauthor{Butucea, C. and Rohde, A. and Steinberger, L.}

\runtitle{Private estimation of the quadratic functional}

\thankstext{T1}{Supported by the ANR Grant Labex Ecodec (ANR-11-LABEX-0047).}
\thankstext{T2}{Supported by the DFG Research Grant RO 3766/4-1.}
\thankstext{T3}{Supported by the Austrian Science Fund (FWF): I 5484-N.}
\thankstext{T4}{Supported by MFO Research in Pairs 2207q.}

\begin{aug}
\author{\fnms{Cristina} \snm{Butucea}\thanksref{T1,T4}}
\and
\author{\fnms{Angelika} \snm{Rohde}\thanksref{T2,T4}}
\and
\author{\fnms{Lukas} \snm{Steinberger}\thanksref{T2,T3,T4}}
\affiliation{CREST, ENSAE, IP Paris; University of Freiburg and University of Vienna}

\address{
	CREST, ENSAE, IP Paris\\
    5, avenue Henry Le Chatelier\\
    91120 Palaiseau \\
	E-mail: cristina.butucea@ensae.fr}

\address{
	Institute of Mathematics\\
	Albert-Ludwigs-Universit\"at Freiburg  \\
	Ernst-Zermelo-Stra{\ss}e 1\\
	79104 Freiburg im Breisgau\\
	E-Mail: angelika.rohde@stochastik.uni-freiburg.de}

\address{
	Department of Statistics and OR\\
	and Data Science @ Uni Vienna\\
	University of Vienna  \\
	Oskar-Morgenstern-Platz 1\\
	1090 Vienna\\
	E-mail: lukas.steinberger@univie.ac.at}
	
\end{aug}

\begin{abstract}
Local differential privacy has recently received increasing attention from the statistics community as a valuable tool to protect the privacy of individual data owners without the need of a trusted third party. Similar to the classical notion of randomized response, the idea is that data owners randomize their true information locally and only release the perturbed data. Many different protocols for such local perturbation procedures can be designed. In most estimation problems studied in the literature so far, however, no significant difference in terms of minimax risk between purely non-interactive protocols and protocols that allow for some amount of interaction between individual data providers could be observed. In this paper we show that for estimating the integrated square of a density, sequentially interactive procedures improve substantially over the best possible non-interactive procedure in terms of minimax rate of estimation. \\
In particular, in the non-interactive scenario we identify an elbow in the minimax rate at $s=\frac34$, whereas in the sequentially interactive scenario the elbow is at $s=\frac12$. This is markedly different from both, the case of direct observations, where the elbow is well known to be at $s=\frac14$, as well as from the case where Laplace noise is added to the original data, where an elbow at $s= \frac94$ is obtained.\\
We also provide adaptive estimators that achieve the optimal rate up to log-factors, we draw connections to non-parametric goodness-of-fit testing and estimation of more general integral functionals and conduct a series of numerical experiments.
The fact that a particular locally differentially private, but interactive, mechanism improves over the simple non-interactive one is also of great importance for practical implementations of local differential privacy.
\end{abstract}



\end{frontmatter}

\section{Introduction}

In the modern information-age an increasing amount of private and sensitive data about each and every one of us (such as medical information, smartphone user behavior, etc.) is perpetually being collected, electronically stored, processed and analyzed. This trend is opposed by an increasing desire for data privacy protection and stricter regulations as expressed, for instance, by the EU \emph{General Data Protection Regulation}\footnote{\url{https://gdpr-info.eu}} which is in effect since May 2018. On the technological side, a particularly fruitful approach to data privacy protection, that is considered insusceptible to privacy breaches, is `differential privacy', formally introduced by \citet{Dwork06}. However, the design and development of optimal statistical estimation procedures under differential privacy is still at its beginnings. A few first contributions in that direction are \citet{Duchi13a, Duchi13b, Duchi14, Wasserman10, Smith08, Smith11, Ye17, Rohde18, Butucea19, Cai19}.

In this paper we focus on the concept of \emph{$\alpha$-local differential privacy} (LDP) to protect the information of individual data providers. The general notion of $\alpha$-differential privacy, as introduced by \citet{Dwork06}, denotes a private data release mechanism that produces an output $Z$ based on original and confidential data $X_1,\dots, X_n$, such that the conditional distribution of $Z$ given $X=(X_1,\dots, X_n)$ satisfies
\begin{equation}\label{eq:DiffPriv}
\sup_{A} \sup_{x,x' : d_0(x,x')=1} \frac{Pr(Z\in A|X=x)}{Pr(Z\in A|X=x')} \quad\le\quad e^\alpha,
\end{equation}
where the first supremum runs over all measurable sets and $d_0(x,x'):=|\{i:x_i\ne x_i'\}|$ denotes the Hamming distance between $x$ and $x'$. Clearly, a smaller $\alpha$ implies a stronger privacy protection. Throughout this paper, we restrict to the case $\alpha\le1$, that is, the privacy protection is not allowed to deteriorate as the sample size increases. The `local' paradigm within differential privacy describes a situation where \emph{no} trusted third party is available that can do data collection and processing, but the original data $X_i$ have to be `sanitized' already on the data providers `local machine' \citep[cf.][]{Evfim03}. This is also closely related to the classical idea of randomized response \citep{Warner65}. In such local privacy protocols, even though the data providers trust nobody with their original data, some amount of interaction may be allowed between individuals. Here we consider two popular protocols for locally private estimation. First, we study the \emph{non-interactive protocol}, where individual $i$ generates a private view $Z_i$ of its original data $X_i$ independently of all the other individuals. Furthermore, we also consider the \emph{sequentially interactive protocol} where the $i$-th individual also has access to the previously sanitized data $Z_1,\dots, Z_{i-1}$ of other individuals in order to generate its own $Z_i$. Of course, sequentially interactive protocols are more flexible than non-interactive ones and have the potential to retain more information about the original unobserved data sample.

Our goal is to provide a complete picture of the minimax theory of estimating the integrated square $D(f) = \int f^2(x) dx$ of the density $f$ of the original i.i.d. data $X_1, \dots, X_n$, under local differential privacy. The quadratic functional plays an important role in statistics, for instance, in goodness-of-fit testing. Very recently, \citet{Lam20} have investigated goodness-of-fit testing based on the quadratic functional for the non-interactive protocol of differential privacy, where each individual uses the same channel to produce a sanitized observation $Z_i$. They succeeded in deriving the non-interactive minimax rate of testing over the particular Besov classes $B_s^{2\infty}$. Here, we study quadratic functional estimation over the scale of general Besov classes $B_s^{pq}$, $p\ge 2$, for both, non-interactive and sequentially interactive locally differentially private mechanisms.
Contrary to most existing results on locally differentially private estimation, we find that for estimating the quadratic functional, using a sequentially interactive protocol considerably improves over the non-interactive one, even in terms of minimax rate of convergence. This phenomenon, that sequentially interactive procedures improve substantially over non-interactive ones, can not be observed for many other private estimation problems such as density estimation \citep{Butucea19}, high-dimensional regression and mean estimation \citep{Duchi17} and estimation of general linear functionals of the true data generating distribution \citep{Rohde18}. However, \citet{Kasiviswanathan11} showed that so-called masked-parity functions can only be learned with interactive procedures but not with purely non-interactive ones. Furthermore, for certain pointer-chasing games that are usually not studied in the statistics literature, \citet{Joseph19, Joseph20} showed that going through multiple rounds of interaction in local differential privacy can reduce the sample complexity by a polynomial factor in the problem dimension.  Only very recently 
\citet{Acharya22} showed that 
sequential interaction can somewhat improve the rate of $s$-sparse mean estimation from $\frac{sd}{n\alpha^2}\log\frac{ed}{s}$ for non-interactive protocols to $\frac{sd}{n\alpha^2}$. Moreover, following up on the present paper, the separation between non-interactive and sequentially interactive privacy mechanisms was further examined in \citet{BerrettButucea} and \citet{Butucea21} for testing and non-linear functional estimation with discrete data. Sequentially interactive locally differentially private mechanisms for statistical estimation have also been proposed, e.g., in \citet[Section~3.2.2]{Duchi17} and \citet{Duchi18v3}, but without rigorously establishing the superiority over non-interactive procedures.


Our main contributions are the following:
\begin{itemize}

\item In the non-interactive case we construct an $\alpha$-differentially private data release mechanism and estimator for $\int f^2(x)dx$ based on U-statistics and sanitized empirical wavelet coefficients. Our procedure is related to the one of \citet{Butucea19} and is shown to achieve the minimax rate (up to log factors) within the class of $\alpha$-non-interactive differentially private procedures over Besov classes $B_s^{pq}$ with $p\ge2$, $q\ge1$. In this case, for $\alpha\in(0,1]$, the optimal convergence rate is given by
$$
(n \alpha^2)^{-\frac{4s}{4s+3}} \lor (n \alpha^2)^{-1/2}.
$$
Notice the elbow at $s=3/4$, where the nonparametric rate transitions into the rate of parametric $\alpha$-private estimation of $\sqrt{n\alpha^2}$. Also observe that the minimax rate of testing in \citet{Lam20} corresponds to the square-root of the nonparamteric part of our rate with respect to $n$, but is suboptimal with respect to $\alpha$ when it tends to zero.

\item The crucial point is that we improve the classical U-statistics approach by considering a two-step procedure that requires sequential (but still locally differentially private) interaction between data owners. The first part $X^{(1)} = (X_1,\dots, X_{n/2})$ of the sample is used to locally construct sanitized data $Z^{(1)} = (Z_1, \dots, Z_{n/2})$ and an estimate $\hat{f}^{(1)}$ of the density $f$, using the method of \citet{Butucea19}. Then, conditional on $Z^{(1)}$, we estimate the linear functional $f\mapsto \int \hat{f}^{(1)} f$ by the method of \citet{Rohde18} in a locally private way. Since $\hat{f}^{(1)}$ has to be provided to the owners of the second half of the data $X^{(2)} = (X_{n/2+1}, \dots, X_n)$ in order for them to generate sanitized data $Z^{(2)} = (Z_{n/2+1}, \dots, Z_n)$, the two-step procedure is inherently sequentially interactive. We establish its optimality within the class of all sequentially interactive procedures (up to log factors) by proving lower bounds on the corresponding minimax risk using a private version of the generalized Le~Cam method \citep[see also][Section~5]{Duchi18v1}. The achieved rate is given by
$$
(n\alpha^2)^{-\frac{4s}{4s+2}} \lor (n\alpha^2)^{-1/2}.
$$
Notice that the elbow is now at $s=1/2$. The fact that sequentially interactive methods may improve substantially over non-interactive ones is also an important lesson for implementations of local differential privacy.

\item We discuss two practically important applications for estimation of the quadratic functional: estimating more general integral functionals and goodness-of-fit testing.

\item We provide a non-interactive as well as a sequentially interactive $\alpha$-locally differentially private estimator of the quadratic functional, both of which do not depend on the smoothness $s$ of the density $f$, and we prove that they attain the respective minimax lower bounds, up to logarithmic factors in $n\alpha^2$.

\item Several numerical experiments are conducted which show that the sequential procedure can also be superior to the non-interactive one in smaller samples.
\end{itemize}

\subsection{Background on estimating quadratic functionals}

One particularly interesting non-linear functional is the quadratic functional. \cite{Bickel88} were the first to discover the so-called elbow phenomenon arising for estimating the integrated square of a density based on independent, identically distributed (i.i.d.)  observations: While a $\sqrt{n}$-efficient estimator exists for H\"older smoothness to the exponent $s>1/4$, the minimax rate of convergence over H\"older balls is $n^{-4s/(4s+1)}$ whenever $s\leq 1/4$ although the standard information bound is strictly positive and finite, see \cite{Ritov90}. Within the Gaussian sequence space model and minimax estimation of the squared $\ell_2$-norm  of the sequential parameter, \cite{Donoho90} found a corresponding phenomenon over $\ell_2$-ellipsoids. A fully data-driven procedure for quadratic functionals, based on model selection, with the functional class being some $\ell_p$ or Besov body for $0<p<2$, is developed in  \cite{Laurent00}. Estimation via quadratic rules of the quadratic functional over parameter spaces which are not quadratically convex is studied in \cite{Cai05}. It is shown that the near minimaxity of optimal quadratic rules typically
does not hold when the parameter space is not quadratically convex. The maximum
risk of quadratic procedures over any parameter space is established to be equal
to the maximum risk over the quadratic convex hull. It also follows from the results
that for Besov balls and $\ell_p$ balls with $0<p<2$, quadratic rules can be minimax rate optimal
only if the minimax quadratic risk is of order $n^{-1}$. The minimax quadratic risk also exhibits the well-known elbow phenomenon as mentioned
above for H\"older balls. More precisely, with $B_s^{pq}(M)$ denoting the centered ball of radius $M$ in the Besov class $B_s^{pq}$,
$$
\inf_{\hat D} \sup_{f\in B_s^{pq}(M)} \E_f\Big[ \big( \hat D - D(f)\big)^2\Big]
\asymp \left\{
\begin{array}{ll}
\frac{M^2}{n},& \frac 12 - \frac 1{2p}  \leq s'
\\
n^{2- \frac{p}{1+2ps'}}, & s' < \frac 12 - \frac 1{2p} ,
\end{array}
\right.
$$
where $0<p<2$ and
$$
s':=s- \frac 1p + \frac 12 >0 .
$$
In the same setting of sparse $\ell_p$ and Besov bodies, \cite{Cai06} construct an adaptive minimax-optimal estimator selecting among a
collection of penalized nonquadratic estimators. A detailed comparison to the results
of \cite{Laurent00} is given in their Section 3.3. \cite{Klemela06} studies
estimation of quadratic functionals for $\ell_p$ bodies with $2<p<\infty$. \cite{Butucea07} treats the problem of quadratic functional estimation on Sobolev classes in the convolution model, where the noise distribution is known and its characteristic function decays either polynomially or exponentially asymptotically. Particularly under polynomial decay at exponent $-\sigma$,  the elbow between parametric and nonparametric rate  is present again but shifted from Sobolev smoothness $1/4$ to $1/4+\sigma$. \cite{Collier17} realize minimax  estimation of linear and quadratic functionals over sparsity classes.

\subsection{Organization of the paper}

The paper is organized as follows. In Section~\ref{sec:Prelim} we discuss some preliminaries on differential privacy and Besov spaces and introduce the formal notation. Section~\ref{SEC:NONINTER} contains our main results on the non-interactive case, including minimax lower bounds and a minimax rate optimal non-interactive estimation procedure. In Section~\ref{SEC:SEQINTER} we present a sequentially-interactive estimation procedure that improves on the rate of the non-interactive method from Section~\ref{SEC:NONINTER}. We also establish its optimality within the larger class of sequentially interactive procedures by proving matching lower bounds (up to log factors). In Section~\ref{SEC:APPL} we discuss consequences and applications of our work for locally private estimation of more general integral functionals and for goodness-of-fit testing. In Section~\ref{sec:adaptation} we present adaptive estimators for both the non-interactive as well as the sequentially interactive case. Finally, in Section~\ref{sec:sim} we summarize the results of extensive numerical experiments to compare and evaluate the performance of our procedures in small samples. All the proofs are collected in the supplementary material \citep{Butucea22supp}.


\section{Preliminaries and notation}
\label{sec:Prelim}

We consider the situation where our $n$ data providers hold confidential data $X_1, \dots, X_n$ assumed to be i.i.d. on $[0,1]$ with common probability density function (pdf) $f:[0,1] \to \R_+$, $f\in L^2[0,1]$. We want to estimate the quadratic functional $D(f) = \int_0^1 f^2(x) dx$. However, we do not observe the original data $X_1,\dots, X_n$, but only the sanitized data $Z_1,\dots, Z_n$ on the measurable space $(\mathcal Z, \mathcal G):=(\prod_{i=1}^n\mathcal Z_i, \bigotimes_{i=1}^n\mathcal G_i)$. The conditional distribution of the observations $Z=(Z_1,\dots, Z_n)$ given the original sample $X=(X_1,\dots, X_n)$ is described by the \emph{channel distribution} $Q$. That is, $Q$ is a Markov probability kernel from $([0,1]^n, \mathcal B([0,1])^{\otimes n})$ to $(\mathcal Z, \mathcal G)$, where $\B([0,1])$ denotes the Borel sets of $[0,1]$ and $\otimes n$ denotes the $n$-fold product sigma field. For ease of notation we suppress the $n$-dependence of $Q$. Hence, the joint distribution of the observation vector $Z = (Z_1, \dots, Z_n)$ on $\prod_{i=1}^n\mathcal Z_i$ is given by $Q_f := Q \Pr_f^{n} $, i.e., the measure $A\mapsto \int_{[0,1]^n}Q(A|x_1,\dots, x_n)\prod_{i=1}^n f(x_i) dx$, $A\in\mathcal G$, where $\Pr_f(B) = \int_B f(x_1) dx_1$, $B\in\mathcal B([0,1])$. Finally, whenever $f$ and $Q$ are fixed and clear from the context, we write $(\Omega, \mathcal F, \Pr)$ for the underlying probability space on which random vectors like $X$ and $Z$ are defined, and we denote by $\E$ and $\Var$ the corresponding expectation and variance operators.

\subsection{Preliminaries on Besov spaces}
\label{sec:Besov}
For the necessary background on Besov spaces we mainly follow \cite{Hardle98} and  \citet[][Section~4.3]{Gine16}.
For any $h>0$, let $\Delta_h$ denote the $h$-shift difference operator, acting pointwise on any real-valued function $g$ on $[0,1]$ as
$$
\Delta_hg(t)=
\begin{cases}
g(t+h)-g(t)&\text{ if } 0\leq t\leq 1-h\\
0 & \text{ otherwise.}
\end{cases}
$$
For any $2\leq r\in\N$, $\Delta_h^{r}=\Delta_h\circ \Delta_h^{r-1}$ inductively defines its $r$-fold composition and if $\arrowvert g\arrowvert^p $ is Lebesgue integrable, $p\geq 1$,
$$
\omega_r(g,t,p)=\sup_{h\in(0,t]}\Arrowvert \Delta_h^rg\Arrowvert_{L_p}
$$
denotes the $r$th modulus of smoothness in the Lebesgue space $L_p$.
   For any $s>0$ and $1\le q\le \infty$, the Besov space $B_s^{pq}$ is given as
$$
B_s^{pq} =
\left\{ f\in L_p([0,1]) : \|f\|_{B_s^{pq}} <\infty\right\},\ \ \ \text{for } 1\le p<\infty,
$$
and with $\mathcal{C}[0,1]$ denoting the real-valued continuous functions on the unit interval,
$$
B_s^{\infty q} =\left\{ f\in \mathcal{C}([0,1]) : \|f\|_{B_s^{\infty q}} <\infty\right\}.
$$
Here,
\begin{align}\label{eq:BesovNorm}
\|f\|_{B_s^{pq}} = \begin{cases}
\|f\|_{L_p} + \left( \sum_{j=0}^\infty  \left[ 2^{js} \omega_r(f, 2^{-j},p)\right]^q\right)^{1/q}, \quad&\text{if } 1\le q<\infty,\\
\|f\|_{L_p} + \sup_{j\ge 0}\left[ 2^{js} \omega_r(f, 2^{-j},p)\right], &\text{if } q=\infty,
\end{cases}
\end{align}
defines the Besov norm, where $r=\lceil s\rceil_>$ is the smallest integer strictly larger than $s$. Note that by classical Besov space embeddings (cf. \citet[][Prop.~4.3.9]{Gine16}) for $p\leq 2$ and Jensen's inequality for $p>2$, the relation $s>(1/p-1/2)_+$ reveals that $B_s^{pq}\subset L_2$.

For the scaling function $\phi = \psi_{-10}=\mathds 1_{(0,1]}$ with  wavelet $\psi=\mathds 1_{(0,1/2]}-\mathds 1_{(1/2,1]}$, define $\psi_{jk}=2^{j/2}\psi(2^j\cdot -k)$ for $j\in\N\cup \{0\}, k\in\{0,1,\dots, 2^j-1\}$. The corresponding family
$$
\Big\{\psi_{-10},\psi_{jk}: j\in\N\cup \{0\}, k\in\{0,1,\dots, 2^j-1\}\Big\}
$$
defines the orthonormal Haar wavelet basis of the Hilbert space $L_2$. Throughout, we will describe the regularity of the Lebesgue density $f$ by its membership in an appropriate Besov ball. For $L>0$,
\begin{equation}
\P_s^{pq}(L) = \left\{f:[0,1]\rightarrow\R : f\ge0, \int_0^1 f(x)\,dx=1, \|f\|_{B_s^{pq}}\le L\right\}
\end{equation}
denotes the subset of Lebesgue probability densities on the unit interval within the centered Besov ball  of radius $L$.
For any $f\in\mathcal{P}_s^{pq}(L)$ with $s>(1/p-1/2)_+$, an application of Parseval's identity reveals the representation
\begin{align*}
\int_0^1f(x)^2dx=\sum_{j\geq -1}\sum_{k=0}^{(1\lor 2^j)-1}\langle f, \psi_{jk}\rangle^2=\sum_{j\geq -1}\sum_{k=0}^{(1\lor 2^j)-1}\beta_{jk}^2
\end{align*}
with the wavelet coefficients $\beta_{jk}=\beta_{jk}(f)=\langle f, \psi_{jk}\rangle_{L_2}$. Note that for general parameter constellations $p,q,s$, the Besov spaces cannot be defined equivalently in terms of Haar wavelet coefficient norms. Nevertheless, the sequences $(\beta_{jk})_{k=0,\dots, 2^{j}-1}$ of above introduced coefficients satisfy the following relation with respect to  the modulus of smoothness. For any $ 1\le p\le \infty$,
there exists some constant $C_p>0$, such that for any $f\in B_s^{pq}$ with $s<1$,
\begin{equation}\label{eq:HaarNormBound}
2^{j(1/2-1/p)}\Arrowvert \beta_{j\cdot}(f)\big\Arrowvert_{\ell_p} \le C_p\,  \omega_1(f, 2^{-j},p)
\end{equation}
for $ j\ge0$, see \cite{Devore92}.

\subsection{Interactive and non-interactive differential privacy}

Recall that for $\alpha\in(0,1]$, a channel distribution $Q$ is called \emph{$\alpha$-differentially private}, if
\begin{equation}\label{eq:alphaPriv}
\sup_{A\in\bigotimes_{i=1}^n\mathcal G_i}\sup_{\substack{x,x'\in[0,1]^n\\ d_0(x,x')=1}} \frac{Q(A|x)}{Q(A|x')} \quad\le \quad e^\alpha,
\end{equation}
where $d_0(x,x') := |\{i:x_i\ne x_i'\}|$ is the Hamming distance between $x$ and $x'$.
Note that for this definition to make sense, the probability measures $Q(\cdot|x)$, for different $x\in[0,1]^n$, have to be equivalent and we interpret $\frac{0}{0}$ as equal to $1$.

Next, we introduce two specific classes of locally differentially private channels. A channel distribution $Q: (\bigotimes_{i=1}^n\mathcal G_i)\times [0,1]^n \to [0,1]$ is said to be $\alpha$-\emph{sequentially interactive} (or provides $\alpha$-sequentially interactive differential privacy) if the following two conditions are satisfied. First, we have for all $A\in \bigotimes_{i=1}^n\mathcal G_i$ and $x_1,\dots, x_n\in[0,1]$,
\begin{align}\label{eq:Seq}
&Q\left( A\Big|x_1,\dots, x_n\right)\notag \\
&=
\int_{\mathcal Z_1}\dots \int_{\mathcal Z_n} Q_n(A_{z_{1:n-1}}|x_n,z_{1:n-1})
Q_{n-1}(dz_{n-1} |x_{n-1}, z_{1:n-2})\dots Q_1(dz_1|x_1),
\end{align}
where, for each $i=1,\dots, n$, $Q_i$ is a channel from $[0,1]\times \bigotimes_{j=1}^{i-1}\mathcal G_j$ to $\mathcal Z_i$. Here, $z_{1:n} = (z_1,\dots, z_n)^T$ and $A_{z_{1:n-1}} = \{z\in\mathcal Z_n : (z_1,\dots, z_{n-1},z)^T\in A\}$ is the $z_{1:n-1}$-section of $A$. Second, we require that the conditional distributions $Q_i$ satisfy
\begin{equation}\label{eq:alphaSeq}
\sup_{A\in\bigotimes_{i=1}^n\mathcal G_i}\sup_{x_i,x_i',z_1,\dots, z_{i-1}} \frac{Q_i(A|x_i, z_1,\dots, z_{i-1})}{Q_i(A|x_i', z_1,\dots, z_{i-1})} \quad\le \quad e^\alpha \quad\quad\forall i=1,\dots, n.
\end{equation}
By the usual approximation of integrands by simple functions, it is easy to see that \eqref{eq:Seq} and \eqref{eq:alphaSeq} imply \eqref{eq:alphaPriv}.
This notion coincides with the definition of sequentially interactive channels in \citet{Duchi17} and \citet{Rohde18}.
We note that \eqref{eq:alphaSeq} only makes sense if for all $x_i, x_i', z_1,\dots, z_{i-1}$, the probability measure $Q_i(\cdot|x_i,z_{1:i-1})$ is absolutely continuous with respect to $Q_i(\cdot|x_i',z_{1:i-1})$.
Here, the idea is that individual $i$ can only use $X_i$ and previous $Z_j$, $j<i$, in its local privacy mechanism, thus leading to the sequential structure in the above definition. In the rest of the paper we only consider $\alpha$-sequentially interactive channels, which we sometimes simply call $\alpha$-private channels.

An important subclass of sequentially interactive channels are the so called \emph{non-interactive} channels $Q$ that are of product form
\begin{equation}\label{eq:non-Inter}
Q\left( A_1\times\dots\times A_n\Big|x_1,\dots, x_n\right) = \prod_{i=1}^n Q_i(A_i|x_i), \quad\quad \forall A_i\in\mathcal G_i, x_i\in[0,1].
\end{equation}
Clearly, a non-interactive channel $Q$ satisfies \eqref{eq:alphaPriv} if, and only if, for all  $i=1,\dots, n$,
\begin{equation}\label{eq:MarginPriv}
\sup_{A\in\mathcal G_i}\sup_{x,x'\in[0,1]}\frac{Q_i(A|x)}{Q_i(A|x')}\quad \le\quad e^\alpha.
\end{equation}
In that case it is also called \emph{$\alpha$-non-interactive}. Both, $\alpha$-non-interactive and $\alpha$-sequentially interactive channels satisfy the $\alpha$-local differential privacy constraint as defined in the introduction. Of course, every $\alpha$-non-interactive channel is also $\alpha$-sequentially interactive.

\subsection{Locally, differentially private minimax risk}

For a fixed channel distribution $Q$ from $([0,1]^n, \mathcal B([0,1]^n))$ to $(\mathcal Z, \mathcal G)$, the minimax risk of the above estimation problem is given by
\begin{align}\label{eq:Qminimax}
\mathcal M_n(Q,\P_s^{pq}) \quad=\quad \inf_{\hat{D}_n} \sup_{f\in \P_s^{pq}} \E_{Q \Pr_f^{n}}\left[ (\hat{D}_n - D(f))^2 \right],
\end{align}
where the infimum runs over all estimators $\hat{D}_n : \mathcal Z\to\R$. Next, define the set of $\alpha$-non-interactive  channels
\begin{equation}\label{eq:NIset}
\mathcal Q_\alpha^{(NI)} := \bigcup_{(\mathcal Z, \mathcal G)}\left\{ Q : Q \text{ is $\alpha$-non-interactive from $[0,1]^n$ to $\mathcal Z$} \right\},
\end{equation}
where the union runs over all $n$-fold product spaces,
and the set of $\alpha$-sequentially interactive channels
\begin{equation}\label{eq:SIset}
\mathcal Q_\alpha^{(SI)} := \bigcup_{(\mathcal Z, \mathcal G)}\left\{ Q : Q \text{ is $\alpha$-sequentially interactive from $[0,1]^n$ to $\mathcal Z$} \right\}.
\end{equation}
Clearly, $\mathcal Q_\alpha^{(NI)}\subseteq \mathcal Q_\alpha^{(SI)}$.
Therefore, we distinguish the $\alpha$-private minimax risks
\begin{align}\label{eq:minimaxPrivNI}
\mathcal M_{n,\alpha}^{(NI)}(\P_s^{pq}) \quad=\quad \inf_{Q\in\mathcal Q_\alpha^{(NI)}} \mathcal M_n(Q,\P_s^{pq})
\end{align}
and
\begin{align}\label{eq:minimaxPrivSI}
\mathcal M_{n,\alpha}^{(SI)}(\P_s^{pq}) \quad=\quad \inf_{Q\in\mathcal Q_\alpha^{(SI)}} \mathcal M_n(Q,\P_s^{pq}).
\end{align}
Note that the above infima include all possible product spaces $(\mathcal Z, \mathcal G)$.

In the sequel we will derive upper and lower bounds on both of these minimax risks (for appropriate subsets of $\P_s^{pq}$). In each case, we will also present an explicit construction of a locally private estimation procedure that attains the lower bound (up to logarithmic factors).

\subsection{Further notation}

We write $a\lor b = \max(a,b)$ and $a\land b = \min(a,b)$. Throughout, $C$, $C_0$, $c$ are positive finite constants that do neither depend on sample size $n$ nor on an unknown parameter $f$, but might depend on $s, p, q, L$ or other constants used to describe the parameter space for $f$, and might change from one occurrence to another. We sometimes write $a \lesssim b$ to mean $a\le C\cdot b$, for a finite constant $C>0$ that does not depend on $n$, $f$ and $\alpha$. Finally, $a \asymp b$ means that $a\lesssim b$ and $b\lesssim a$.


\section{Non-interactive privacy protocols}
\label{SEC:NONINTER}

In this section we present an $\alpha$-non-interactive privacy mechanism and subsequent estimator for the quadratic functional $D(f)=\int_0^1 f^2(x)dx$ and establish its minimax optimality within the class $\mathcal Q_\alpha^{(NI)}$ of all $\alpha$-non-interactive procedures.

\subsection{Upper bounds}

We first propose a non-interactive privacy mechanism, related to the one of \citet{Butucea19}, that is based on adding Laplace noise to empirical wavelet coefficients. The subsequent estimator is a standard U-statistic of order 2.

Let us define the following privacy mechanism using the Haar basis generated by $(\phi, \psi)$, with $\phi(x) = \mathds 1_{(0,1]}(x)$ and $\psi(x) = \mathds 1_{(0,\frac 12]}(x) - \mathds 1_{(\frac 12, 1]}(x)$, $x\in\R$. Fix $\alpha>0$, $a>1$ and $J\in\N$. Given its original data $X_i$, individual $i$ generates a random array $Z_i$ with $(j,k)$-th component
\begin{equation}\label{eq:nonInterQ}
Z_{ijk} =
\psi_{jk}(X_i ) + \sigma_j \cdot \frac{\sigma}{\alpha}\cdot W_{ijk}, \quad j=-1,...,J-1, \, k=0,...,\lceil2^j-1\rceil,
\end{equation}
where $\sigma_{-1} = 1$,  $\sigma_j = (1\vee j)^{a} 2^{j/2}$ for $j\geq 0$, and $\sigma = 4 + 2\sum_{j=1}^{\infty} \frac 1{j^{a}}  $. Moreover, $W_{ijk}$ are i.i.d. Laplace distributed with density $f^W(x) = \frac 12 \exp(-|x|)$. Note that $W_{ijk}$ are all centered, with variance 2. We write $Q^{(NI)}$ for the conditional distribution (Markov kernel, channel distribution) of $(Z_1,\dots, Z_n)$ given $(X_1, \dots, X_n)$. In particular, the channel $Q^{(NI)}$ is non-interactive. The following result establishes that $Q^{(NI)}\in\mathcal Q_\alpha^{(NI)}$. Its proof is deferred to Section~A.1 in the supplement \citep{Butucea22supp}.

\begin{proposition}\label{prop:NIchannel}
For any $J\in\N$ and $a > 0$, $\alpha>0$, the privacy mechanism $Q^{(NI)}$ defined in \eqref{eq:nonInterQ} is $\alpha$-non-interactive.
\end{proposition}

We shall use the notation
\begin{equation}\label{estcoeff}
\hat \beta_{jk} = \frac 1n \sum_{i=1}^n Z_{ijk}, \quad j = - 1,...,J-1, \, k=0,...,(1\lor 2^j)-1,
\end{equation}
with $\hat \beta_{-1,0}$ also called $\hat \alpha_{00}$.
Since $W_{ijk}$ are i.i.d. centered, with variance 2, we get for $j=-1,...,J-1$, $k=0,...,(1\lor 2^j)-1$:
$$
\E (\hat \beta_{jk}) = \beta_{jk} \quad \text{and }
\Var(\hat \beta_{jk}) = \frac 1n \left( \Var(\psi_{jk}(X_1)) + 2 \sigma_j^2 \cdot  \frac{\sigma^2}{\alpha^2}\right) .
$$
Finally, let us define the private estimator $\hat{D}_n$ of $D=D(f)$, by
\begin{equation}\label{hatD}
\hat{D}_n =  \frac 1{n(n-1)} \underset{ i\not= h}{\sum^n} \sum_{j=-1}^{J-1} \sum_{k=0}^{(1\lor 2^j)-1} Z_{ijk} \cdot Z_{hjk} .
\end{equation}
We are now in the position to formulate our first main result on the risk of $\hat{D}_n$. Its proof is deferred to Section~A.2 of the supplementary material \citep{Butucea22supp}.

\begin{theorem} \label{noninteractiveUB}
For finite constants $L, M_2, M_3 >0$, $1\le p,q\le\infty$ and $s> (\frac 1p - \frac12)_+$, consider $\bar{\P}_s^{pq}(L,M_2,M_3) = \P_s^{pq}(L)\cap\{f\in L_3([0,1]):\|f\|_{L_2}\le M_2, \|f\|_{L_3}\le M_3\}$. Put $s' = s - (\frac 1p - \frac 12 )_+$. Then, for every $n\in\N$ and $\alpha\in(0, 1]$, with $n\alpha^2>1$, the estimator $\hat{D}_n$ with $J = J_n$ given by
$$
2^{J_n} = \left\{ \begin{array}{ll}
\left( \frac {n \alpha^2}{ (\log (n\alpha^2))^{4a+1} } \right)^{\frac 13},& s' > \frac 34,\\
(n\alpha^2)^{\frac{2}{4s'+3}}, & 0 < s' \leq \frac 34,
\end{array} \right.
$$
verifies
$$
\sup_{f \in \bar{\P}_s^{pq}(L, M_2,M_3)} \E_{Q_f^{(NI)}}\left[ \left|\hat{D}_n - D(f) \right|^2 \right]
\;\lesssim \; \mathfrak r_n^{(NI)}(\alpha, a, s'),
$$
where
\begin{align*}
\mathfrak r_n^{(NI)}(\alpha, a, s') =
\begin{cases}
\frac{1}{n \alpha^2}, & s'> \frac 34, \\
(\log (n\alpha^2))^{4a+1} (n\alpha^2)^{-\frac{8s'}{4s'+3}}, & 0 < s' \leq \frac 34.
\end{cases}
\end{align*}
\end{theorem}

\subsection{Lower bounds}

We now show that the rate of the non-interactive U-statistics approach introduced in the previous subsection is indeed optimal for estimating the quadratic functional within the class of all $\alpha$-non-interactive procedures. See Section~A.3 in the supplementary material \citep{Butucea22supp} for the proof of the following theorem.

\begin{theorem}\label{LowBoNI} Fix $n \in \mathbb{N}$, $\alpha \in (0, \infty)$, $s \in (0,1)$, $p\geq 2$, $q \geq 1$, $L>1$, $M\ge 2$ and consider the class $\bar\P_s^{pq}(L,M) := \{f\in\P_s^{pq}(L):\|f\|_\infty\le M\}$. Define $z_\alpha:= e^{2\alpha} - e^{-2\alpha}$. If $n z_\alpha^2\ge 2$, then there exists a constant $c>0$, not depending on $n$ and $\alpha$, such that
$$
\inf_{Q \in {\mathcal{Q}}_\alpha^{(NI)}} \inf_{\hat D_n} \sup_{f \in \bar{\mathcal{P}}_s^{pq}(L,M)} \E_{Q\Pr_f^n} \left[ \left|\hat D_n - D(f)\right|^2\right] \;\geq\; \frac{c}{[\log (n z^2_\alpha)]^2} \left(n z^2_{\alpha} \right)^{- \frac{8 s}{4s + 3}}.
$$
Here, the set ${\mathcal{Q}}_\alpha^{(NI)}$ contains all $\alpha$-non-interactive channels.
\end{theorem}

\begin{remark}\label{rem:alpha2}
Since $\frac{1}{2}(e^{2\alpha}-e^{-2\alpha}) \le e^{2\alpha}-1$, we immediately get the slightly smaller lower bound
$$
c'\left[ n(e^{2\alpha}-1)^2\right]^{-\frac{8s}{4s+3}}\left(\log\left[ n (e^{2\alpha}-1)^2\right]\right)^{-2},
$$
which, for bounded $\alpha$, reduces to an expression in terms of the more familiar quantity $n\alpha^2$, i.e.,
$$
c''\left(n\alpha^2\right)^{-\frac{8s}{4s+3}} [\log(n\alpha^2)]^{-2}.
$$
\end{remark}

Theorem~\ref{LowBoNI} shows that the rate obtained in Theorem~\ref{noninteractiveUB} is indeed optimal (up to logarithmic factors), at least in the case $p\ge 2$, that is, $s'=s$.

Finally, we note that one can easily deduce a lower bound of the form $c(n\alpha^2)^{-1}$, even for the larger class $\mathcal Q_\alpha^{(SI)} \supseteq  \mathcal Q_\alpha^{(NI)}$ and over general Besov classes $\bar{\P}_s^{pq}(L,M)$, $s>0$, $1\le p,q\le\infty$, using Corollary~3.1 of \citet{Rohde18}. To that end, we only need to lower bound the modulus of continuity of the quadratic functional w.r.t. the total variation distance, that is,
$$
\omega_{TV}(\eps) := \sup\left\{\left|D(f_0) - D(f_1)\right| : \frac 12 \int\left|f_0-f_1\right|\le \eps, f_0,f_1\in \bar{\P}_s^{pq}(L,M) \right\},
$$
by an expression of order $\eps$, because a minimax lower bound is of the form $c_0[\omega_{TV}(c_1(n\alpha^2)^{-1/2})]^2$.
But this can easily be done for $\eps\in(0,1]$, by choosing $f_0 \equiv 1$ and $f_1(x) = f_0(x) + \delta g(x/\eps)$, for some non-trivial $g\in B_s^{pq}$ with $\int_0^1 g(x)dx=0$ and $\|g\|_\infty<\infty$, and for $0<\delta\le [(L-1)\land1]/(\|g\|_\infty\lor\|g\|_{B_s^{pq}})$. This choice implies that $f_1(x)\ge0$, $\|f_1\|_{B_{s}^{pq}} \le 1 + \eps\delta\|g\|_{B_{s}^{pq}}\le L$, $\|f_1\|_\infty\le 1 + \delta\|g\|_\infty\le 2\le M$, $|D(f_0) - D(f_1)| = \eps \delta^2\|g\|_2^2$ and $\int|f_0-f_1| = \eps\delta\|g\|_1$. Thus, $\omega_{TV}(\eps\delta\|g\|_1/2) \ge \eps\delta^2\|g\|_2^2$.


\section{Sequentially interactive privacy protocols}
\label{SEC:SEQINTER}

In Section~\ref{SEC:NONINTER} we have presented an $\alpha$-non-interactive procedure for estimating the quadratic functional $D=\int_0^1f^2(x)dx$ and established its minimax optimality within the class of all $\alpha$-non-interactive procedures. If we leave this class, however, and also allow for sequential interaction between data owners, then we can improve substantially over the rate of the best non-interactive procedure. In the present section we pursue such improvements and prove their optimality.

\subsection{Upper bounds}
\label{sec:SeqIntUpper}

We first provide a concrete example of a locally private estimation procedure which relies on some sequential communication between individual data providers and which achieves a faster convergence rate than that of Section~\ref{SEC:NONINTER}.

For convenience, we assume that the sample size is $2n$ and we split the data providing individuals into two groups of size $n$, such that the first group holds data $X^{(1)}=(X_1^{(1)}, \dots, X_n^{(1)})$ and the second group holds the data $X^{(2)}=(X_1^{(2)}, \dots, X_n^{(2)})$. Now, the individuals owning the data $X^{(1)}$ use the non-interactive privacy mechanism \eqref{eq:nonInterQ}, which is based on the Haar wavelets, to generate arrays $Z_i = Z_i^{(1)}$ based on their private information $X_i^{(1)}$. We write $Z^{(1)} = (Z_1^{(1)}, \dots, Z_n^{(1)})$. These sanitized data are now used to estimate the unknown data generating density $f\in\P_s^{pq}(L)$ at a point $x\in[0,1]$, by \citep[cf.][]{Butucea19}
\begin{equation}\label{eq:densEst}
\hat{f}_J^{(1)}(x) := \sum_{j=-1}^{J-1} \sum_{k=0}^{(1\lor 2^j)-1} \hat \beta_{jk} \psi_{jk}(x),
\end{equation}
with $\hat{\beta}_{jk}$ as in \eqref{estcoeff}, i.e.,
$$
\hat \beta_{jk} = \frac 1n \sum_{i=1}^n Z_{ijk}^{(1)}, \quad j = - 1,...,J-1, \, k=0,...,(1\lor 2^j)-1.
$$
Now, in order to privately estimate the quadratic functional $D = D(f) = \int_0^1 f^2(x)\,dx$, we instead privately estimate the (random) linear functional
$$
f\mapsto \int_0^1 \hat f_J^{(1)}(x) f(x)\,dx.
$$
This second step is carried out using the rate optimal mechanism of \citet{Rohde18}, that is, for some tuning parameter $\tau>0$ and for given $Z^{(1)}$ (or $\hat{f}_J^{(1)}$), each individual from the second group independently generates $Z_i^{(2)}$ by
\begin{align}\label{eq:Z_i2}
Z_i^{(2)} = \begin{cases}
\tau\frac{e^\alpha+1}{e^\alpha-1}, &\text{with probability } \frac12\left(1+ \frac{\Pi_\tau[ \hat{f}_J^{(1)}(X_i^{(2)})]}{\tau\frac{e^\alpha+1}{e^\alpha-1}} \right),\\
-\tau\frac{e^\alpha+1}{e^\alpha-1}, &\text{with probability }  \frac12\left(1- \frac{\Pi_\tau[ \hat{f}_J^{(1)}(X_i^{(2)})]}{\tau\frac{e^\alpha+1}{e^\alpha-1}} \right),
\end{cases}
\end{align}
where $\Pi_\tau[y] = (\tau\land y)\lor(-\tau)$. Write $Z^{(2)} = (Z_1^{(2)}, \dots, Z_n^{(2)})$. Note that the projection of $\hat f_J^{(1)}(X_i^{(2)})$ onto $[-\tau, \tau]$ ensures that the probabilities belong to $[0,1]$. Moreover, notice that we have $\E[Z_i^{(2)}|Z^{(1)}, X_i^{(2)}] = \Pi_\tau[ \hat{f}_J^{(1)}(X_i^{(2)})]$ and $\E[Z_i^{(2)}| Z^{(1)}] = \int_0^1 \Pi_\tau[ \hat{f}_J^{(1)}(x)] f(x)\, dx \to \int_0^1 \hat f_J^{(1)}(x) f(x)\,dx$ as $\tau\to\infty$. Our final estimator is then given by
\begin{equation}\label{eq:DhatSeq}
\tilde{D}_n = \tilde{D}_{n,\tau} = \frac{1}{n}\sum_{i=1}^n Z_i^{(2)}.
\end{equation}

We denote the above mechanism that outputs $(Z^{(1)}, Z^{(2)})$, given original data $(X^{(1)}, X^{(2)})$, by $Q^{(SI)}$. It clearly has a sequential structure because each $Z_i^{(2)}$ in the second group depends on the sanitized data $Z^{(1)}$ from the first group through $\hat{f}_J^{(1)}$, but on none of the other $Z_j^{(2)}$, $j\ne i$. It is also easy to see that it satisfies \eqref{eq:alphaSeq} and hence, it is $\alpha$-sequentially interactive, i.e., $Q^{(SI)}\in\mathcal Q_\alpha^{(SI)}$. The following theorem presents an upper bound on the risk of the estimation method proposed in \eqref{eq:DhatSeq}. Its proof is deferred to Section~B.2 of the supplement.


\begin{theorem}\label{interactiveUB}
Fix $M,L>0$, $1\le p,q\le \infty$ and $s> \left( \frac1p - \frac12\right)_+$ and consider the Besov class $\bar\P_s^{pq}(L,M) := \{f\in\P_s^{pq}(L):\|f\|_\infty\le M\}$. Define $s' = s - \left( \frac1p - \frac12\right)_+$. For $n\in\N$, $\alpha\in(0,1]$, consider the estimator $\tilde D_n$ defined in \eqref{eq:DhatSeq} based on the private wavelet estimator $\hat f_J^{(1)}$ in \eqref{eq:densEst}, with cut-off
$$
\tau^2 = [K^2M^2 (1 \vee  J^{2a + 1} 2^{J(1-2(s'\land \frac12)) })]\lor1,
$$
for a sufficiently large constant $K\ge2$ (that can be chosen independently of $n$ and $\alpha$) and for $J=J_n$ such that $2^{J_n} = (n \alpha^2)^{\frac 1{2(s' \wedge 1 ) + 1}}$, where $a>1$ is the constant from the privacy mechanism \eqref{eq:nonInterQ}. Then,
\begin{align*}
\sup_{f\in\bar{\P}_s^{pq}(L,M)} \E_{Q_f^{(SI)}}\left[\left|\tilde D_n - D(f)\right|^2\right] \lesssim \mathfrak r_n^{(SI)}(\alpha, a, s')
\end{align*}
with $$
\mathfrak r_n^{(SI)}(\alpha, a, s') = 
\begin{cases}
\frac{1}{n\alpha^2}, & s' > \frac 12\\
(\log (n \alpha^2))^{2a+1} (n \alpha^2)^{-\frac{4s'}{2s'+1}}, & s' \leq \frac 12,
\end{cases}
$$
provided that $n\alpha^2 > c_0$, for a finite constant $c_0>0$ that does not depend on $n$ and $\alpha$.
\end{theorem}

Theorem~\ref{interactiveUB} shows that faster rates than those of Section~\ref{SEC:NONINTER} can be attained using a sequentially interactive privacy mechanism. Indeed, the elbow effect occurs at the value $s'=\frac 12$ instead of $s'=\frac 34$
in Theorem~\ref{noninteractiveUB}, and in case $s' \leq \frac 12$ we have that
$$
(n \alpha^2)^{-\frac{4s'}{2s'+1}} / (n \alpha^2)^{-\frac{8s'}{4s'+3}} \to 0, \quad \text{as } n \alpha^2 \to \infty.
$$
Intuitively, a sequentially interactive privacy mechanism increases the information that the sanitized sample contains about the unknown parameter of interest. However, that this additional information can be exploited to obtain faster rates than those of non-interactive procedures can not be observed for the problem of density estimation in $L_r$ or of estimating linear functionals of the density \citep[cf.][]{Rohde18, Butucea19}.

\subsection{Lower bounds}
\label{sec:lowerB:SI}

In this subsection we show that the rate of the sequentially interactive procedure introduced in Subsection~\ref{sec:SeqIntUpper} is indeed optimal. See Section~B.3 in the supplement \citep{Butucea22supp} for the proof of the following theorem.

\begin{theorem}\label{thm:SIlowerB}
Fix $n\in\N$, $\alpha\in(0,\infty)$, $s\in(0,1)$, $p,q\in[1,\infty]$, $L>1, M\ge2$ and let the class $\bar\P_s^{pq}(L,M)$ be defined as in Theorem~\ref{interactiveUB}. Define $z_\alpha := e^{2\alpha}-e^{-2\alpha}$. Then, if $n z_\alpha^2\ge1$, there exists a constant $c>0$ not depending on $n$ and $\alpha$, such that
$$
\inf_{Q\in\mathcal Q_\alpha^{(SI)}} \inf_{\hat{D}_n}\sup_{f\in\bar{\P}_s^{pq}(L,M)}\E_{Q\Pr_f^n}\left[ \left|\hat{D}_n - D(f)\right|^2\right]\;\ge\; c \left[n z_\alpha^2\right]^{-\frac{4s}{2s+1}}.
$$
\end{theorem}

In view of Remark~\ref{rem:alpha2}, the lower bound can further be bounded from below by $c'[n\alpha^2]^{-\frac{4s}{2s+1}}$, provided that $\alpha$ is bounded. Theorem~\ref{thm:SIlowerB} shows that the rate in Theorem~\ref{interactiveUB} is optimal, at least in the regime where $s=s'$, that is, $p\ge 2$, and up to log factors. Recall that in the argument following Theorem~\ref{LowBoNI} we have already established the parametric lower bound of order $(n\alpha^2)^{-1}$.


\section{Applications}\label{SEC:APPL}

Next, we discuss two common applications where estimation of the quadratic functional plays an important role: estimating more general integral functionals and goodness-of-fit testing.

\subsection{Integral functionals of the density}

Suppose we want to estimate other integral functionals $T(f) = \int \phi (f(x)) dx$ of the bounded density $f$, such as, for example, the entropy $\int f(x)\log(f(x))dx$. If $\phi : \R_+ \to \mathbb{R}$ is three times continuously differentiable, we can follow ideas of \citet{Birge95} \citep[see also][Section~5.3.1]{Gine16}, and perform a Taylor expansion of $\phi$ at a suitable preliminary estimator followed by successive estimation of the resulting linear and quadratic functionals. More specifically, let $\hat f_n$ be a preliminary estimator of $f$, based on a subset $X^{(1)}$ of the whole sample and corresponding sanitized data $Z^{(1)}$, and write
\begin{eqnarray}
\int_0^1 \phi(f) &=& \int_0^1 \left[ \phi(\hat f_n) + \phi'(\hat f_n) (f-\hat f_n) + \frac 12\phi''(\hat f_n) (f-\hat f_n)^2 \right] + G_n\notag\\
&=& \int_0^1 \left[ \phi(\hat f_n) - \phi'(\hat f_n) \hat f_n + \frac 12\phi''(\hat f_n) (\hat f_n)^2  \right] \notag\\
&&+ \int_0^1 f \cdot \left[\phi'(\hat f_n) - \phi''(\hat f_n) \hat f_n \right] + \frac 12 \int_0^1 f^2\cdot \phi''(\hat f_n) + G_n,\label{eq:PhiExpansion}
\end{eqnarray}
where $|G_n|\leq \frac 16 \|\phi'''\|_\infty \int |f - \hat f_n|^3$.
Now, it remains to plug in optimal estimators of the linear and quadratic integral functionals $f\mapsto\int_0^1 f\cdot \psi_1 $ and $f\mapsto \int_0^1 f^2\cdot \psi_2$, for known functions $\psi_1$ and $\psi_2$, constructed with the remaining data sample $X^{(2)}$.

First note that according to \citet{Rohde18}, the rate for $\alpha$-privately estimating the linear functional $f\mapsto \int_0^1 f \cdot \psi_1$ over a convex parameter space, is $(n\alpha^2)^{-1/2}$, provided that the function $x\mapsto \psi_1(x) := \phi'(\hat f_n(x)) - \phi''(\hat f_n(x)) \hat f_n(x)$ is bounded on $(0,1)$. Hence, estimating the linear term in the expansion \eqref{eq:PhiExpansion} will never dominate the rate.\footnote{In case of a functional like the entropy, $\phi(f) = f\log(f)$, where $\phi'$ is unbounded on $(0,1)$, one usually assumes that both $f$ and $\hat{f}_n$ are bounded from below by some positive constant. }

Next, for the preliminary estimator $\hat{f}_n$ based on sanitized data $Z^{(1)}$, let us consider the minimax adaptive estimator in \citet{Butucea19}, which has the property that, for privacy level $\alpha \in(0,1]$ and $r \geq 1$,
\begin{align*}
\sup_{f \in \P_s^{pq}} \E_{Q^*\mathbb P_f^n} \left[\int_0^1 |f - \hat f_n|^r \right] \lesssim
(\log n)^C\begin{cases}
(n \alpha^2)^{- \frac{rs}{2s +2}},  &s > \frac rp -1 ,\\
\left(\frac{n \alpha^2}{ \log (n \alpha^2)} \right)^{-\frac{r(s-1/p + 1/r)}{2(s-1/p) + 2}}, &\frac 1p < s \le \frac rp - 1,
\end{cases}
\end{align*}
where $Q^*$ is the optimal adaptive non-interactive channel of \citet{Butucea19} that generates $Z^{(1)}$ from original data $X^{(1)}$ and does not depend on knowledge of $s$. This non-interactive procedure is actually shown to be rate optimal even among all sequentially interactive privacy mechanisms. For simplicity, here we ignore logarithmic terms in all the rates and only consider the case $p\ge2$, which implies that only the first of the two regimes above occurs and that $s'=s$.

Had we done only a first order expansion instead of \eqref{eq:PhiExpansion}, then the remainder term would dominate and the resulting private estimator would converge at a rate of $(n\alpha^2)^{-\frac{1}{2}}\lor(n\alpha^2)^{-\frac{s}{s+1}}$.
In view of our results in Section~\ref{sec:SeqIntUpper}, however, the quadratic functional can be estimated at a rate of $(n\alpha^2)^{-\frac{1}{2}}\lor(n\alpha^2)^{-\frac{2s}{2s+1}}$ and the remainder term $G_n$ in \eqref{eq:PhiExpansion} converges at the rate $(n\alpha^2)^{-\frac{3}{2}\frac{s}{s+1}}$, both of which are always faster than the rate of the first order expansion. Thus, the expansion \eqref{eq:PhiExpansion} improves over the first order expansion. Furthermore, if $s\ge1/2$, then both, $G_n$ and the quadratic functional estimate converge at the parametric rate and a higher order expansion would not improve the overall rate any further. If, on the other hand, $s<1/2$, then we might be able to improve the rate further by considering a third order expansion.

However, if we restrict to non-interactive privacy mechanisms, then the quadratic functional can only be estimated at the rate $(n\alpha^2)^{-\frac{1}{2}}\lor(n\alpha^2)^{-\frac{4s}{4s+3}}$ (cf. Section~\ref{SEC:NONINTER}) and this is always worse than the rate of the remainder term $G_n$. Thus, further expansion of $\phi$ to fourth or higher order can not improve the rate in the non-interactive case.

Hence, in some cases, our rates for estimating the quadratic functional already determine the rates for the estimation of much more general integral functionals $T(f) = \int \phi(f(x))dx$ with three times continuously differentiable $\phi$. This is in contrast with the direct case when $X_1,...,X_n$ are observed, where both, the quadratic and the cubic functional can be estimated at the rate
$(n\alpha^2)^{-\frac{1}{2}}\lor(n\alpha^2)^{-\frac{4s}{4s+1}}$ and the remainder term $G_n$ converges at the rate $(n\alpha^2)^{-\frac{3s}{2s+1}}$. Thus, the second order expansion is always dominated by the remainder term (for $s<1/4$ the remainder term converges strictly slower) and a third order expansion will be more efficient in terms of rate.
Due to the inverse problem that local differential privacy introduces, the cubic term is not always necessary in the private setting.

\subsection{Goodness-of-fit tests}

The most frequent application of our results is goodness-of-fit testing for the underlying density $f$. Due to the regularizing properties of the $\mathbb{L}_2$ norm, testing rates are usually faster than estimation rates of $f$ (with pointwise or integrated risks). The nonparametric test problem writes $H_0: f \equiv f_0$ for fixed, given $f_0$ in $\bar\P_s^{pq}$, against the alternative
$$
H_1(f_0,\mathcal{C} \varphi_n): f \in \bar\P_s^{pq}, \quad \|f - f_0\|_2 \geq \mathcal{C} \varphi_n,
$$
for some constant $\mathcal{C}>0$ and sequence $\varphi_n$ of real numbers decreasing to 0.
In the context of local differential privacy, test procedures $\Delta_n$ will be defined as measurable functions of the sanitized sample $Z_1, \dots, Z_n$, which is generated from the privacy mechanism $Q\in\mathcal Q_\alpha\subseteq \mathcal Q_\alpha^{(SI)}$. The risk measure of a test procedure for a given privacy mechanism is defined by
$$
\mathcal{T}_n(Q,\Delta_n, \mathcal{C}\varphi_n) := Q \Pr_{f_0}^{n} (\Delta_n = 1)
+ \sup_{f \in H_1(f_0,\mathcal{C} \varphi_n)} Q \Pr_{f}^{n} (\Delta_n = 0).
$$
Let $\gamma$ belong to (0,1). We say that a test procedure $\Delta_n$ associated to a privacy mechanism $Q$ attains the testing rate $\varphi_n$ if, for a constant $\mathcal{C}>0$,
$$
\limsup_{n \to \infty} \mathcal{T}_n(Q,\Delta_n, \mathcal{C}\varphi_n) \leq \gamma.
$$
This rate is the minimax rate of testing among all $\alpha$-sequentially interactive procedures if, for some $0 < \mathcal{C}^* < \mathcal{C}$,
$$
\liminf_{n \to \infty} \inf_{Q \in \mathcal{Q}_\alpha} \inf_{\Delta_n} \mathcal{T}_n(Q,\Delta_n, \mathcal{C}^*\varphi_n) \geq \gamma >0.
$$
We distinguish the cases of non-interactive privacy mechanisms $\mathcal Q_\alpha = \mathcal Q_\alpha^{(NI)}$ and of sequentially interactive privacy mechanisms $\mathcal Q_\alpha=\mathcal{Q}_\alpha^{(SI)}$.

It is known in the direct case (when $X_1,...,X_n$ are observed) that the optimal test is based on the optimal estimator of the quadratic functional $\|f-f_0\|_2^2$. Instead of a plug-in procedure, the test statistic is based on optimal estimators of $D(f)=\|f\|_2^2$ and of the linear functional $L = \int_0^1 f_0 f$. We already mentioned that for bounded $f_0$ in $\bar\P_s^{pq}$, the linear functional $L$ can be estimated by $\hat L_n$ at rate $(n \alpha^2)^{-1/2}$ via an estimator based on a non-interactive privacy mechanism \citep[see][]{Rohde18}, therefore the rates will be driven by the estimator of the quadratic functional $D(f)$.\\
The test procedure $\Delta_n^{(NI)} = 1$, iff $\hat D_n^{(NI)} - 2 \hat L_n + \|f_0\|_2^2 > C t_n^{(NI)}$, where $\hat D_n^{(NI)}$ is the procedure in \eqref{hatD} with $2^J = (n \alpha^2)^{-2/(4s'+1)}$, attains the rate $\varphi_n^{(NI)}$, where
$$
t_n^{(NI)} = \varphi_n^{(NI)} = (n \alpha^2)^{- \frac{2s'}{4s' +3}} \cdot \log^{a +\frac 14}(n\alpha^2), \quad a>1.
$$
The test procedure $\Delta_n^{(SI)} = 1$, iff $\hat D_n^{(SI)} - 2 \hat L_n + \|f_0\|_2^2 > C t_n^{(SI)}$, where $\hat D_n^{(SI)}$ is the procedure in $(\ref{eq:DhatSeq})$, attains the rate $\varphi_n^{(SI)}$, where
$$
t_n^{(SI)} = \varphi_n^{(SI)} = (n \alpha^2)^{- \frac{2s'}{4s' + 2}} \cdot \log^{\frac{a}{2} +\frac 14}(n\alpha^2), \quad a>1.
$$
The upper bounds are simple consequences of the upper bounds for estimating the quadratic functional $D(f)$. It is also easy to deduce the corresponding lower bounds (without the $\log$ factors) from the proofs of the lower bounds on estimation. Indeed, in these proofs, the estimation risk is first reduced to the risk for testing and this is further bounded from below.

\citet{Lam20} have recently derived similar results for goodness-of-fit testing over spaces $B_s^{2\infty}$ in the special case of non-interactive privacy with identical privacy mechanisms on each sample $X_i$, $Q^{\times n}$. Their innovative method for establishing lower bounds is generalized here in order to take into account general non-interactive privacy mechanisms $\prod_{i=1}^{n} Q_i$, in order to achieve optimality over Besov $B_s^{p q}$, $s>0$, $p\geq 2$, $q\geq 1$, smoothness classes and optimality with respect to the privacy level $\alpha$ when it tends to 0.

The testing approach above has been followed by \cite{BerrettButucea} for goodness-of-fit testing of discrete distributions with separation defined by the $\mathbb{L}_2$ and $\mathbb{L}_1$ norms. In the case of discrete distributions, it has also been noticed that the analogous interactive privacy mechanism introduced here allows for faster rates of testing than any non-interactive privacy mechanism.


\section{Adaptation to the smoothness}
\label{sec:adaptation}

Notice that the estimators considered so far use the smoothness $s$ of the unknown underlying probability density in order to determine the optimal resolution level $J$, and this resolution level plays a role in both, the construction of the sanitized data and in the estimation procedure of the quadratic functional. In addition, our sequentially interactive procedure relied on an optimal truncation parameter $\tau$ that also depended on $s$. In this section we show how to aggregate procedures for different values of the resolution level $J$ and select the optimal $\hat J$ and the associated estimator $\hat D_{\hat J}$ in a data driven way by minimizing a penalized criterion. 

\subsection{Non-interactive setup}

Let us consider the sanitized samples $Z_{ijk}, j=-1,...,J_{max}-1, \, k=0,...,\lceil2^j-1\rceil$ in \eqref{eq:nonInterQ}. Recall that we use a fixed arbitrary constant $a>1$ and $\sigma = 4 + 2\sum_{j=1}^{\infty} \frac 1{j^{a}} $ in the construction of $Z'$s. We use the estimator $\hat{D}_n$ as defined in \eqref{hatD}, but to emphasize the dependence on $J$ in the definition, we now write $\hat D_J$ instead of $\hat D_n$, for some $J\le J_{max}$, and we suppress the dependence on $n$. 

For some constant $\mathcal{C}>0$, define
$$
pen^{(NI)}(J) = \mathcal{C} \frac{J^{2a} 2^{3J/2}}{n \alpha^2} \log (2^{4J+1}),
$$
where $J \in \mathcal{J} = \{1,2,\dots, J_{max}\}$ such that the largest value $J_{max}$ in $\mathcal{J}$ satisfies, for some $\kappa > 4(a +1) $, 
$$
\frac{2^{3J_{max}}}{n^2\alpha^4} \leq \log^{-\kappa}(n \alpha^2).
$$
If $\mathcal{C}>0$ is chosen sufficiently large, the penalty allows us to define the final purely data-driven estimator as follows:
$$
\hat D_n^{(NI)} := \max \left\{ \hat D_J - pen^{(NI)}(J) : J \in \mathcal{J}\right\}. 
$$
The proof of the following theorem is deferred to Section~C.1 of the supplement \citep{Butucea22supp}.

\begin{theorem}\label{thm:adaptivNI} Under the assumptions and notations of Theorem~\ref{noninteractiveUB}, the adaptive estimator $\hat D_n^{(NI)}$ associated to the penalty $pen^{(NI)}$ above is such that
$$
\sup_{f \in \bar{\P}_s^{pq}(L, M_2,M_3)} \E_{Q_f^{(NI)}}\left[ \left|\hat{D}^{(NI)}_n - D(f) \right|^2 \right]
\;\lesssim_{log} \; \mathfrak r_n^{(NI)}(\alpha, a, s'),
$$
for every $n\in\N$ and $\alpha\in(0, 1]$, with $n\alpha^2>1$. Here, $\lesssim_{log}$ indicates that polynomial factors in $\log(n\alpha^2)$ have been omitted. 
\end{theorem}

We proceed similarly to \cite{Laurent2005} in order to build the adaptive procedure in the non-interactive case. However, the $U$-statistic that we build is based on randomized versions $Z_{ijk}$ of $\psi_{jk}(X_i)$ and therefore we can decompose the estimator $\hat D_J$ of $D$ at each resolution level $J$ into terms that already appeared in \cite{Laurent2005} but also two additional terms due to the Laplace random variables $W_{ijk}$. Our proof is mainly dedicated to dealing with these additional terms. For example, standard concentration inequalities for $U$-statistics of order 2 in \cite{HouRey2003} do not apply to unbounded random variables and we tailor a proof using the coupling inequality in \cite{delaPena95}.

\subsection{Interactive setup}
Like the non-adaptive one, our interactive adaptive procedure also proceeds in two steps. First, half of the data providers generate sanitized samples in a non-interactive way as before and these are used to build preliminary estimators of the wavelet coefficients and of the underlying probability density. Then, the second part of the data providers use this information to generate sanitized samples that are subsequently used both for estimating the quadratic functional at an arbitrary resolution level $J$, as well as for estimating the appropriate penalty term. Indeed, the variance of our interactive non-adaptive procedure depends on the smoothness of the unknown density through its wavelet coefficients and we need to estimate this quantity in order to build a purely data-dependent penalty.

Let $J_{\max} = J_{\max}(n\alpha^2,B)$ be defined such that 
$$
\frac{2^{2 J_{\max}}}{n \alpha^2} \asymp \frac 1{\log^B(n \alpha^2)},
$$
for some large enough $B>0$. 
Our interactive adaptive procedure is defined as follows. First, the sample is divided into two equally sized parts, where we assume for simplicity that $n\in 2\N$. Based on the first sample, we then generate for each $j\in\{-1,0, 1,\dots ,J_{\max}-1\}$ and $k\in\{0,1,\dots, (1\lor2^{j})-1\}$ the random variables
$
Z_{ijk}^{(1)}
$
as given in \eqref{eq:nonInterQ} and then build $\hat{\beta}_{jk}$ as in \eqref{estcoeff}. Now, in the second step, the set $\{n/2+1,\dots, n\}$ is decomposed into $J_{\max}+1$ (approximately) equally sized parts $\mathcal{N}_j$, $j\in\{-1,0, 1,\dots ,J_{\max}-1\}$. For each $j$ and each individual $X_i$ with $i\in \mathcal{N}_j$, we generate 
$$
Z_i^{(2,j)}:=\pm \tau\frac{e^{\alpha}+1}{e^{\alpha}-1}\ \ \text{with probability} \ \frac{1}{2}\left(1\pm \frac{\Pi_{\tau}[\sum_{k=0}^{(1\lor2^{j})-1}\hat{\beta}_{jk}\psi_{jk}(X_i^{(2)})]}{\tau\frac{e^{\alpha}+1}{e^{\alpha}-1}}\right),
$$
where $\tau = \log^\kappa(n\alpha^2)$ for some large enough $\kappa=\kappa(a,B)>0$.
Define the estimator at resolution level $J$ as
$$
\hat{D}_J:=\sum_{j=-1}^{J-1}\frac{1}{\arrowvert \mathcal{N}_j\arrowvert}\sum_{i\in\mathcal{N}_j} Z_i^{(2,j)}.
$$
Next, define
$$
\widehat{pen}^2(J):={\frac{1}{n\alpha^2}\sum_{j=-1}^{J-1}\sigma_j^2\left(\frac{1}{\arrowvert \mathcal{N}_j\arrowvert}\sum_{i\in\mathcal{N}_j} Z_i^{(2,j)}\right)_+},
$$
where $\sigma_j$ is as in \eqref{eq:nonInterQ}. 
Then 
\begin{equation}\label{eq: Jhat}
\hat{J}:=\text{argmax}_{J\in\{1,\dots, J_{\max}\}}\left(\hat{D}_J-\widehat{pen}(J)\right).
\end{equation}
The final estimator is then $\hat{D}_n^{(SI)}=\hat{D}_{\hat{J}}$. The following theorem is proved in Section~C.2 of the supplement \citep{Butucea22supp}.

\begin{theorem}\label{thm:adaptSI}
Under the assumptions and notations of Theorem~\ref{interactiveUB}, the adaptive estimator $\hat D_n^{(SI)}$ associated to the data-driven $\hat{J}$ in \eqref{eq: Jhat} satisfies
\begin{align*}
\sup_{f\in\bar{\P}_s^{pq}(L,M)} \E_{Q_f^{(SI)}}\left[\left|\hat D_n^{(SI)} - D(f)\right|^2\right] \lesssim \mathfrak r_n^{(SI)}(\alpha, a, s'),
\end{align*}
up to logarithmic factors.
\end{theorem}


\section{Numerical results}
\label{sec:sim}
In this section we complement our theoretical findings about the private minimax convergence rates by an extensive simulation study to further investigate potential strengths and weaknesses of the non-interactive and sequentially interactive locally private estimation procedures suggested above. 

In our first round of numerical experiments we consider H\"older smooth data generating densities $f_s$. More specifically, let
$$
f_s(x) := (s+1)x^s,\quad x\in(0,1], s\in(0,1).
$$
Then $f_s$ belongs to the H\"older space
$$
C^s((0,1]) := \left\{ f\in \mathcal C((0,1]): \|f\|_\infty + \sup_{x\neq y, x,y\in(0,1]} \frac{|f(x)-f(y)|}{|x-y|^s} < \infty\right\}
$$
which is itself identical to $B_s^{\infty\infty}$ with equivalent norms \citep[cf.][Proposition~4.3.23]{Gine16}. Figure~\ref{fig:Holder} shows examples of $f_s$ for different smoothness parameters $s$. 
\begin{figure}[htbp]
\includegraphics[width=\textwidth]{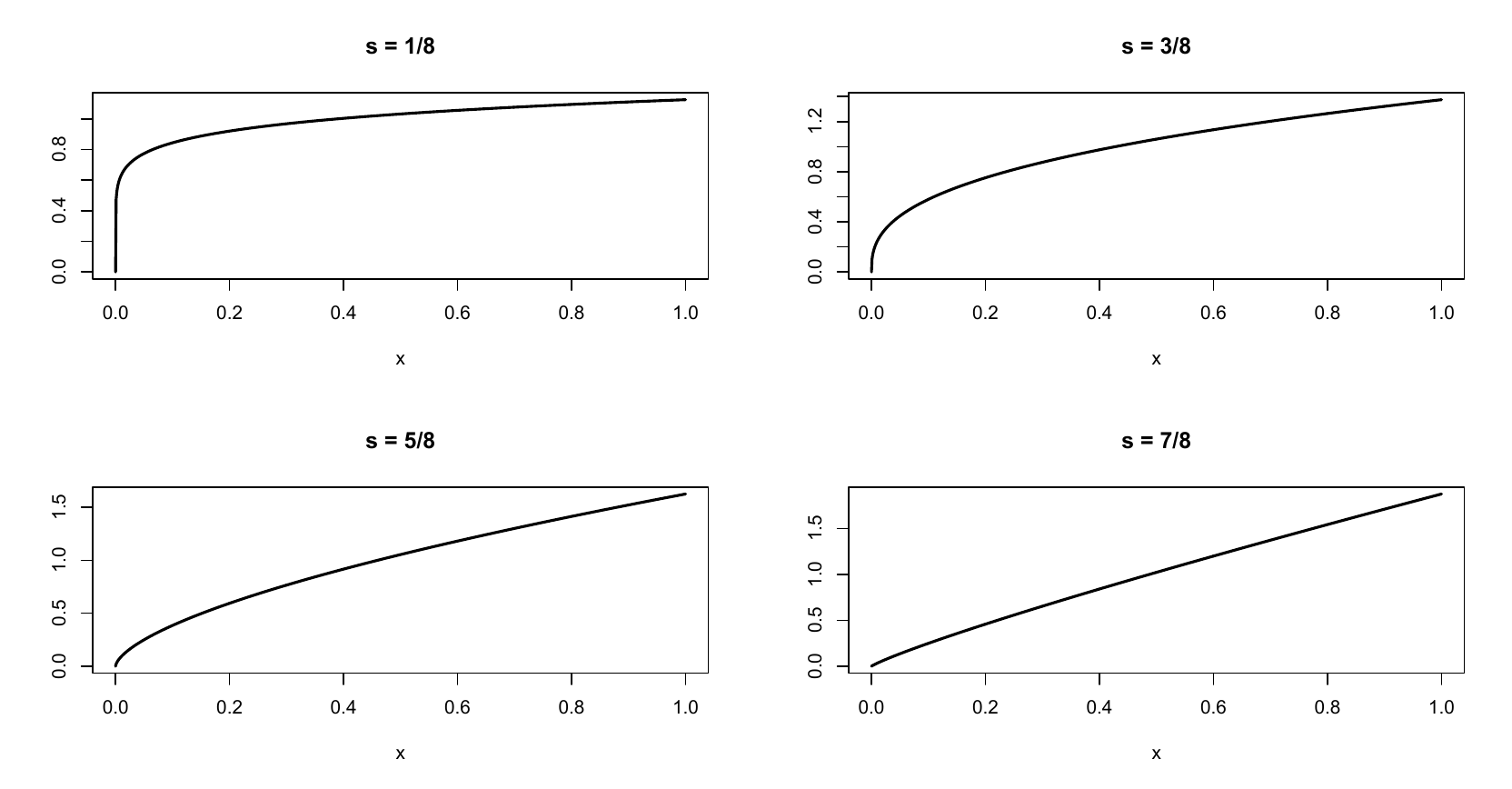} 
\caption{Data generating densities from $B_s^{\infty\infty}$ for different values of smoothness $s\in(0,1)$.}
\label{fig:Holder}
\end{figure}
We implement the non-interactive procedure in \eqref{hatD} and the private wavelet density estimator in \eqref{eq:densEst} both with $\sigma_{-1}=0$, $\sigma_j = (1\lor 2^{j/2})(2J+1)$, $j\ge0$, and $\sigma=1$. For the second step in the sequentially interactive procedure, we compute a more practical version that is, however, harder to analyze theoretically, by choosing $\tau = \|\hat{f}_J^{(1)}\|_\infty$. It is easy to see that with these modifications both procedures still satisfy the privacy constraint. In particular, we have $\psi_{-1,0}(x) = \phi(x) =1$, for all $x\in(0,1]$ and hence $Z_{i,-1,0} = 1$ and $\hat{\beta}_{i,-1,0}=1$, irrespective of the sensitive information $X_i$. In Figures~\ref{fig:HolderAlpha1} to \ref{fig:HolderAlpha100} we plot mean squared errors based on $100$ Montecarlo iterations of samples of size $n=1000$ each ($2n=1000$ in case of the two-step procedure), for different values of the tuning parameter $J$. The parameter $J$ determines the number of resolution levels to be included in the wavelet estimator, so a large value of $J$ results in an estimator that includes more details but also more Laplace noise, whereas a small $J$ leads to a larger bias. Notice that for $J=0$ both procedures disregard the original data $X_1,\dots, X_n$, the non-interactive procedure simply returns $\hat{D}_n=1$ and the sequentially interactive one essentially does the same but adds a little bit of random noise in the second step.

\begin{figure}[htbp]
\includegraphics[width=\textwidth]{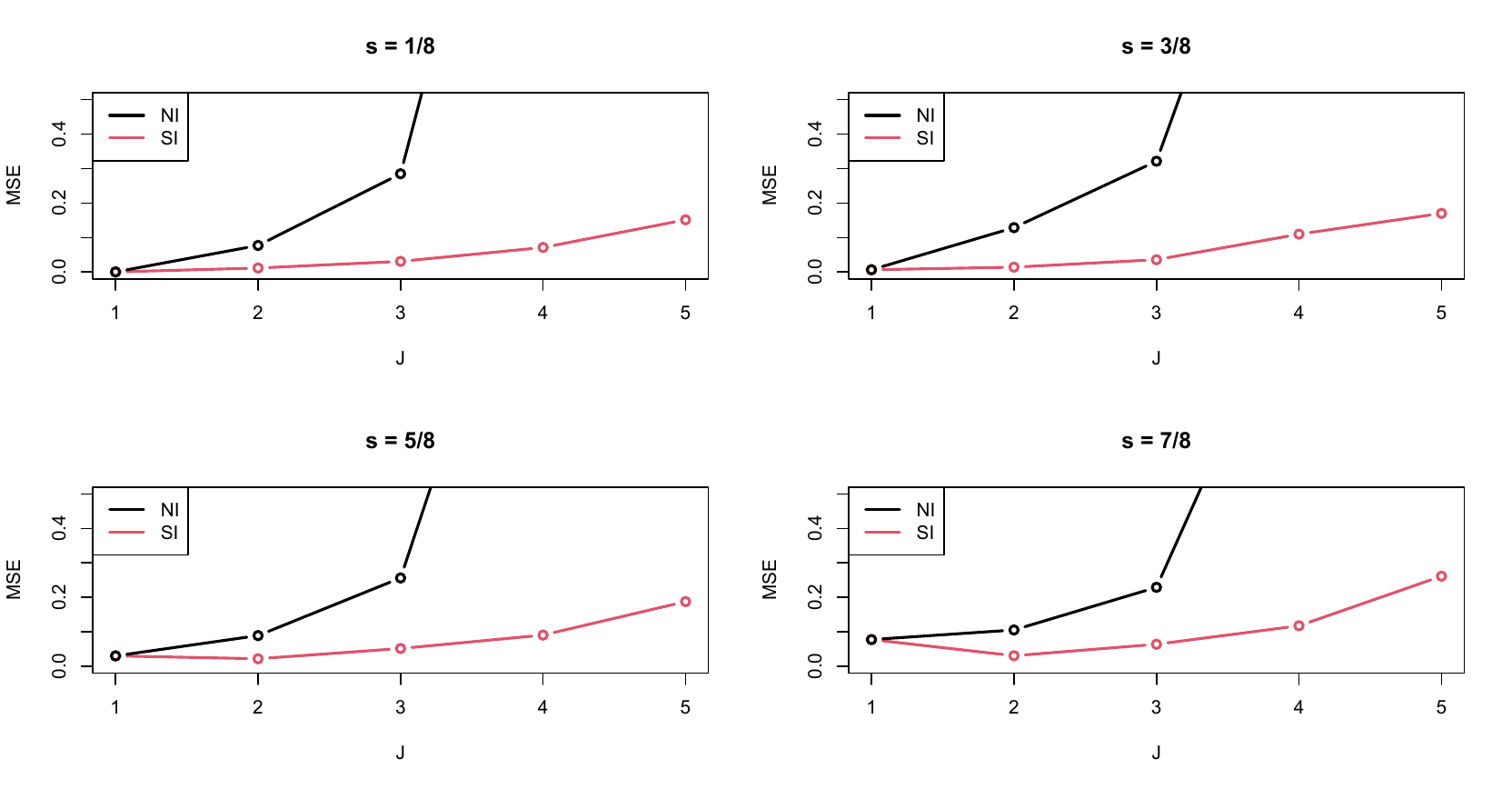} 
\caption{MSE of non-interactive (NI) and sequentially interactive (SI) procedure with privacy level $\alpha=1$, and true data generating distributions as in Figure~\ref{fig:Holder}.}
\label{fig:HolderAlpha1}
\end{figure}

\begin{figure}[htbp]
\includegraphics[width=\textwidth]{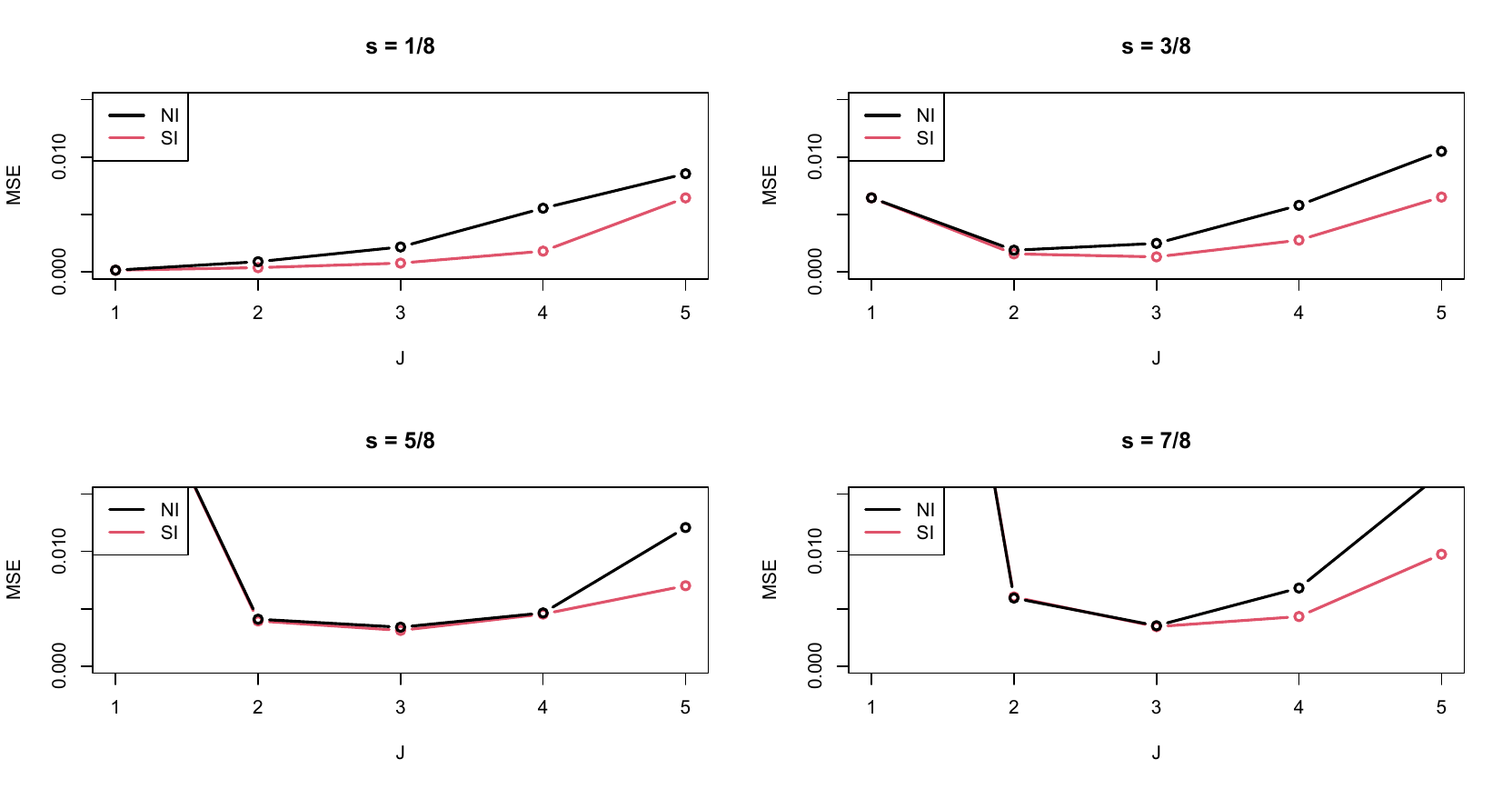} 
\caption{MSE of non-interactive (NI) and sequentially interactive (SI) procedure with privacy level $\alpha=10$, and true data generating distributions as in Figure~\ref{fig:Holder}.}
\label{fig:HolderAlpha10}
\end{figure}

\begin{figure}[htbp]
\includegraphics[width=\textwidth]{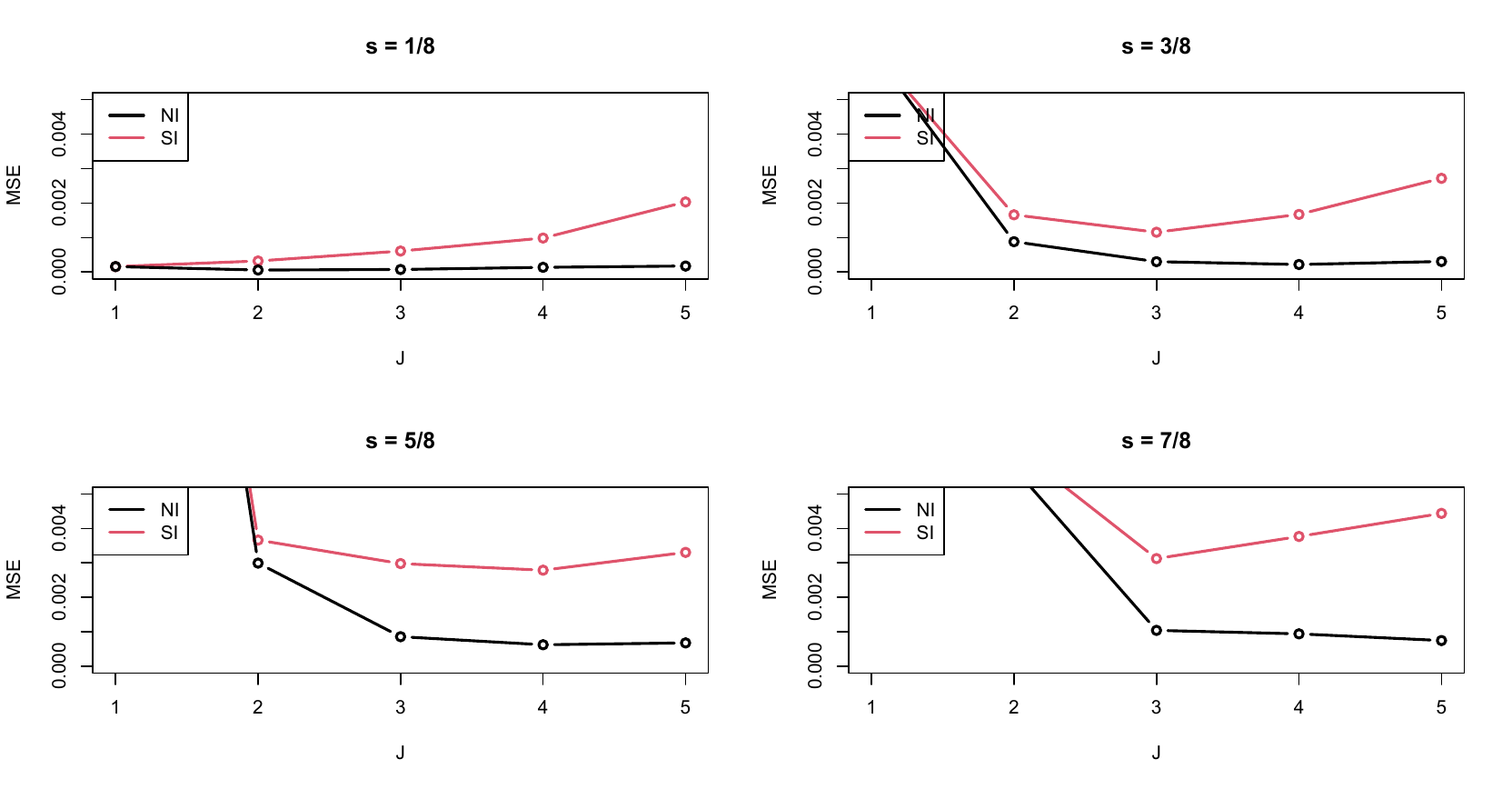} 
\caption{MSE of non-interactive (NI) and sequentially interactive (SI) procedure with privacy level $\alpha=100$, and true data generating distributions as in Figure~\ref{fig:Holder}.}
\label{fig:HolderAlpha100}
\end{figure}

We see that for a smaller privacy level of $\alpha=1$ (Figure~\ref{fig:HolderAlpha1}), the sequentially interactive procedure (SI) consistently outperforms the non-interactive (NI) one. However, for both methods, adding higher resolution levels to the estimator does not really pay off. Although we have a sample size of $n=1000$, still we are basically unable to pick up any signal under all the noise of the privacy mechanism, and simply assuming a uniform density (i.e., $J=0$ and $\hat{D}_n=1$) and suffering the resulting bias is acceptable compared to the higher variance that results from increasing $J$. There are two options for how to alleviate this problem. Either we further increase the sample size $n$, which gets challenging in terms of simulation runtime, or we can increase the privacy level $\alpha$. For $\alpha=10$ (Figure~\ref{fig:HolderAlpha10}), we observe a clear benefit of increasing the resolution $J$, except for the very non-smooth case $s=1/8$, which corresponds to a density that is close to uniform. However, at the same time the superiority of the SI mechanism is reduced and is lost altogether in case $\alpha=100$ (Figure~\ref{fig:HolderAlpha100}). Of course, this should not surprise us, because for large $\alpha$ the NI differentially private procedure is almost equivalent to the conventional U-statistics estimator based on direct observation of $X_i$ (the Laplace noise in \eqref{eq:nonInterQ} vanishes as $\alpha\to\infty$), in which case a sample splitting approach is clearly inferior. Moreover, notice that the noise added in the second step \eqref{eq:Z_i2} of the sequentially interactive mechanism does not disappear even as $\alpha\to\infty$.

From these observations we conclude that it is hard to see the superiority of the sequential mechanism in terms of convergence rate as $n\to\infty$ in a simulation, because in order to pick up enough local structure of the true density underneath all the added differentially private noise we need very large sample sizes. However, we also see that the sequential mechanism has an advantage when dealing with certain global features of the true density as in Figure~\ref{fig:Holder} ($s=7/8$). Therefore, we conduct a second round of simulations with very smooth but more structured beta-densities (Figure~\ref{fig:Beta}).

\begin{figure}[htbp]
\includegraphics[width=\textwidth]{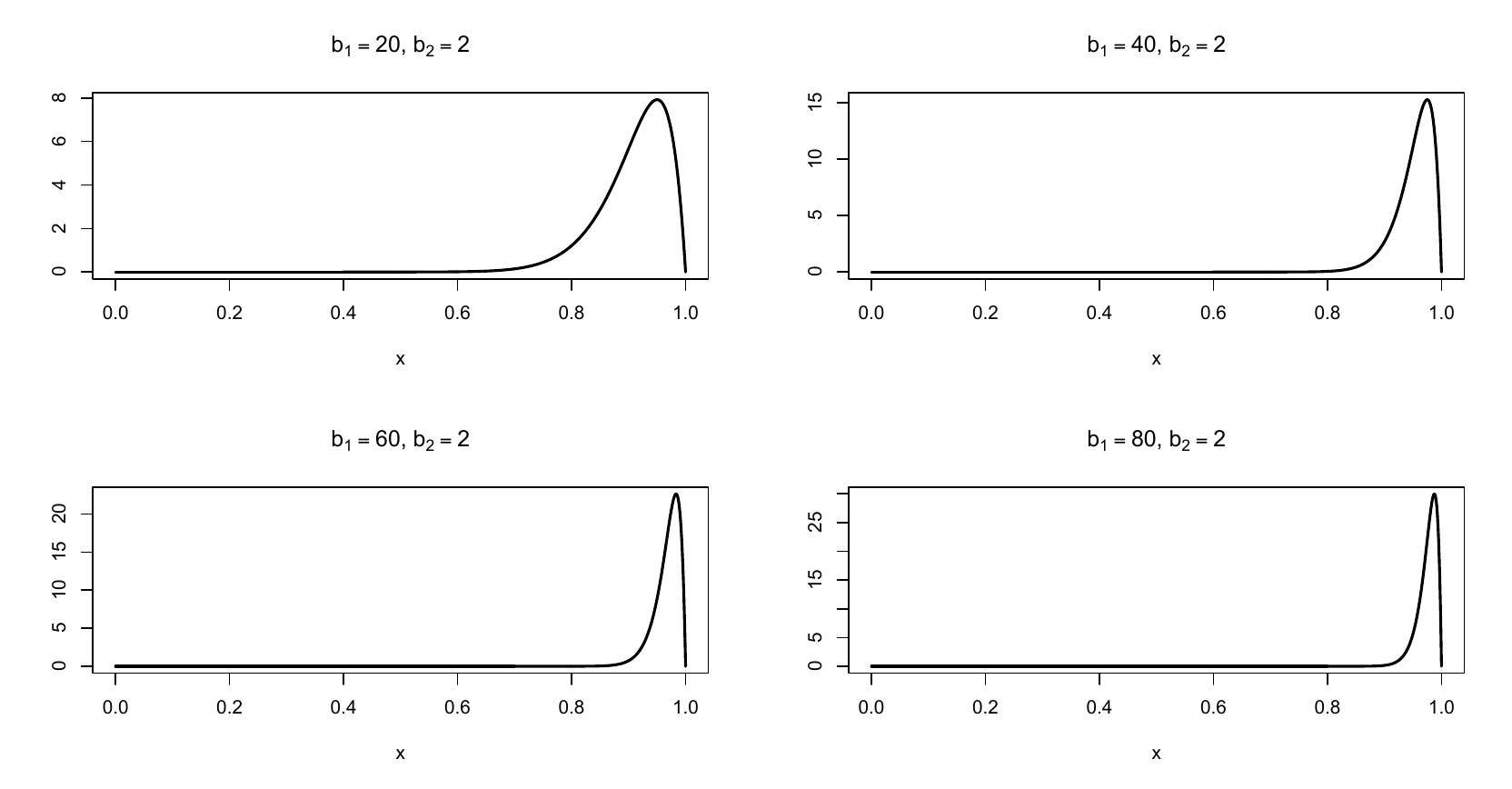} 
\caption{Beta($b_1,b_2$) densities.}
\label{fig:Beta}
\end{figure}

\begin{figure}[htbp]
\includegraphics[width=\textwidth]{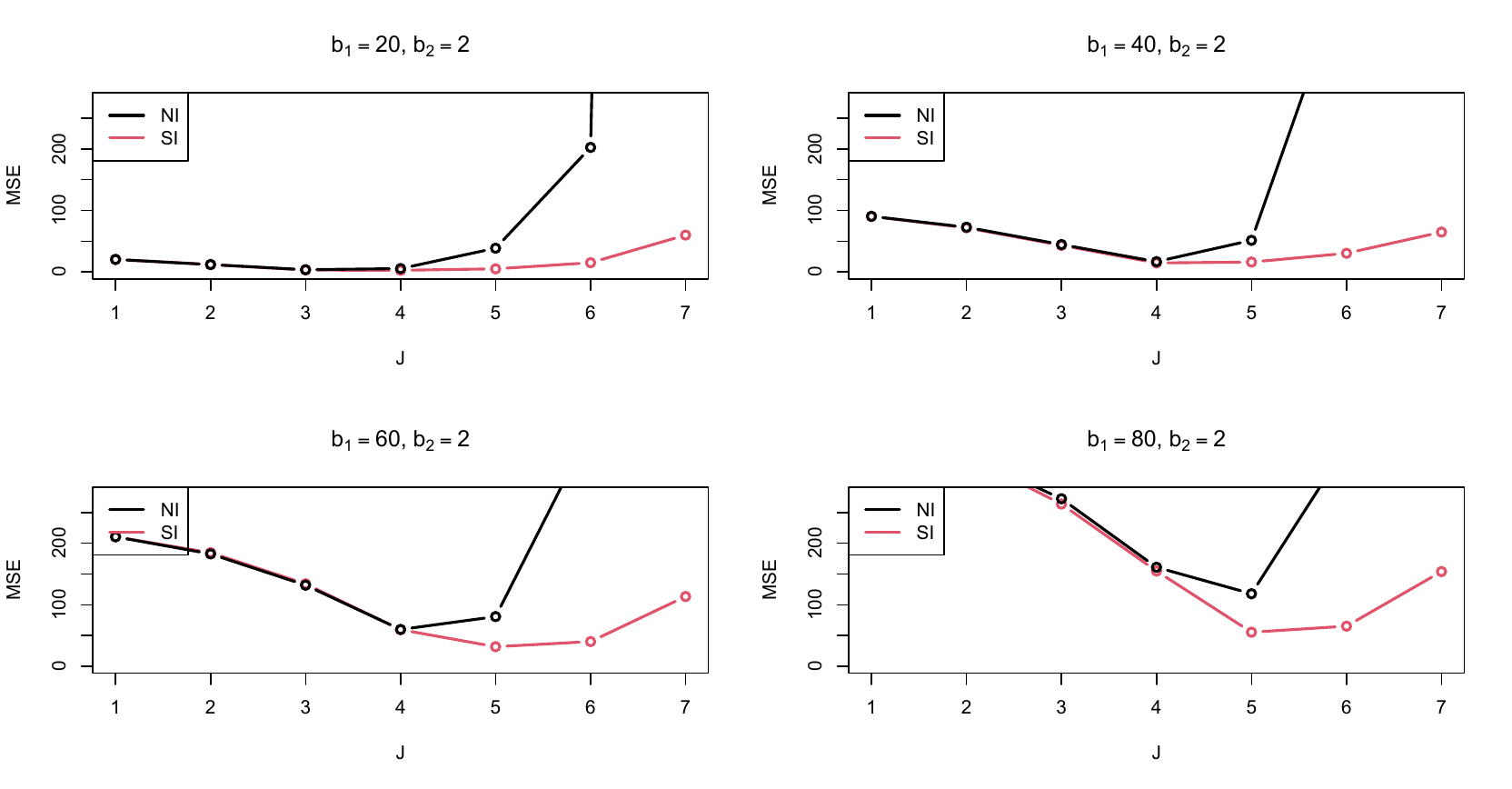} 
\caption{MSE of non-interactive (NI) and sequentially interactive (SI) procedure with privacy level $\alpha=1$, and true data generating distributions as in Figure~\ref{fig:Beta}.}
\label{fig:BetaAlpha1}
\end{figure}

For a privacy parameter of $\alpha=1$, in Figure~\ref{fig:BetaAlpha1} we see that both the NI as well as the SI mechanism can pick up some signal beneath the differentially private noise. Moreover, the SI procedure is never worse than the NI mechanism and achieves a smaller optimum at a value of $J$ that may also be different from the optimal $J$ of the $NI$ procedure.

Finally, we numerically investigate the role of the splitting ratio in the sequentially-interactive two-step procedure. By $n_1\le n$ we denote the number of observations that are used in the first step to compute the private wavelet density estimator \eqref{eq:densEst}. In Figure~\ref{fig:n1} we plot MSEs of the SI procedure as a function of $J$ and $\frac{n_1}{n}$. For reference, the black dashed line denotes the MSE of the NI procedure with optimally chosen $J$. Original data $X_i$ were generated from the Beta$(20,2)$-distribution. The nearly constant MSE curves for small values of $J$ can be explained by the fact that the bias of the SI procedure does not depend on the splitting ratio $\frac{n_1}{n}$ while at the same time the bias dominates the MSE for small $J$. However, the fraction of data points $\frac{n_1}{n}$ used in the first step of the SI procedure strongly influences the variance which dominates the MSE for larger values of $J$. We also see in Figure~\ref{fig:n1} that in order to beat the NI mechanism with the SI procedure, the choice of $J$ is not as critical ($J=2,3,4$ would do), provided that about two thirds of the data go into the first wavelet density estimation part of SI.

\begin{figure}[htbp]
\includegraphics[width=\textwidth]{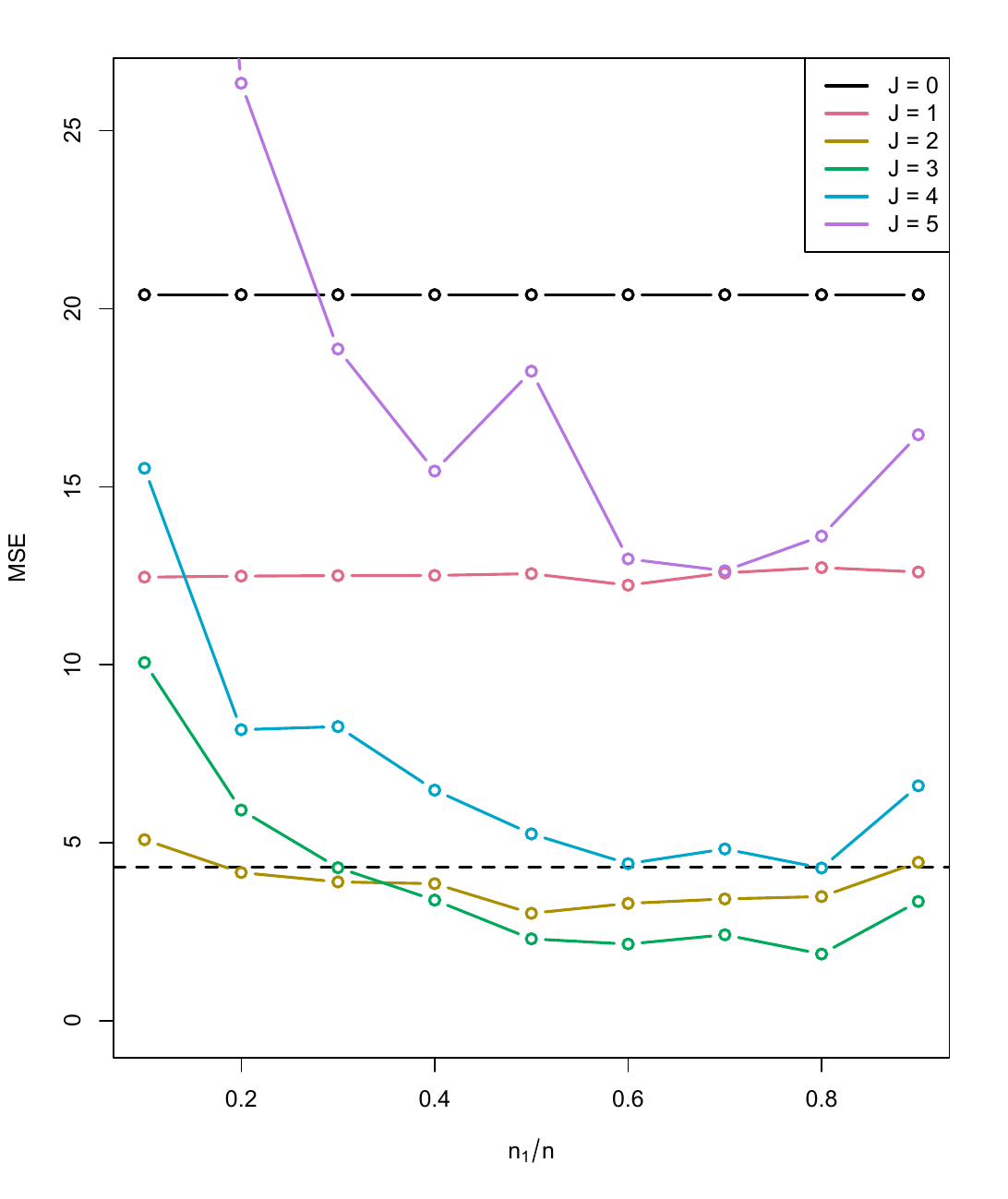} 
\caption{MSE of NI with optimal tuning $J=3$ (black dashed horizontal line) and of SI in dependence on $J$ and fraction $\frac{n_1}{n}$ of data points used in step one. Original private data were generated from the Beta$(20,2)$-distribution.}
\label{fig:n1}
\end{figure}

We conclude this numerical section by pointing out again that although the precise quantitive behavior of the convergence rates of the two local differentially private procedures (investigated in Sections~\ref{SEC:NONINTER} and \ref{SEC:SEQINTER}) can only be observed in very large samples, we still see clear benefits of the sequentially-interactive procedure over the simple non-interactive one also for smaller samples and when the privacy protection is strong and $\alpha$ is small.


\begin{supplement}[id=supp]
\stitle{Supplement to ``Interactive versus non-interactive locally differentially private estimation: Two elbows for the quadratic functional''}
\slink[url]{supplement.pdf}
\sdescription{The supplementary material contains all the proofs.}
\end{supplement}


\bibliographystyle{imsart-nameyear}
\bibliography{article}{}

\begin{thebibliography}{43}

\bibitem[\protect\citeauthoryear{Acharya et~al.}{2022}]{Acharya22}
\begin{barticle}[author]
\bauthor{\bsnm{Acharya},~\bfnm{Jayadev}\binits{J.}},
  \bauthor{\bsnm{Canonne},~\bfnm{Cl{\'e}ment~L}\binits{C.~L.}},
  \bauthor{\bsnm{Sun},~\bfnm{Ziteng}\binits{Z.}} \AND
  \bauthor{\bsnm{Tyagi},~\bfnm{Himanshu}\binits{H.}}
(\byear{2022}).
\btitle{The Role of Interactivity in Structured Estimation}.
\bjournal{arXiv preprint arXiv:2203.06870}.
\end{barticle}
\endbibitem

\bibitem[\protect\citeauthoryear{Berrett and Butucea}{2020}]{BerrettButucea}
\begin{barticle}[author]
\bauthor{\bsnm{Berrett},~\bfnm{Thomas}\binits{T.}} \AND
  \bauthor{\bsnm{Butucea},~\bfnm{Cristina}\binits{C.}}
(\byear{2020}).
\btitle{Locally private non-asymptotic testing of discrete distributions is
  faster using interactive mechanisms}.
\bjournal{NeurIPS}
\bvolume{34}.
\end{barticle}
\endbibitem

\bibitem[\protect\citeauthoryear{Bickel and Ritov}{1988}]{Bickel88}
\begin{barticle}[author]
\bauthor{\bsnm{Bickel},~\bfnm{Peter~J}\binits{P.~J.}} \AND
  \bauthor{\bsnm{Ritov},~\bfnm{Yaacov}\binits{Y.}}
(\byear{1988}).
\btitle{Estimating integrated squared density derivatives: sharp best order of
  convergence estimates}.
\bjournal{Sankhya A}
\bvolume{50}
\bpages{381--393}.
\end{barticle}
\endbibitem

\bibitem[\protect\citeauthoryear{Birg{\'e} and Massart}{1995}]{Birge95}
\begin{barticle}[author]
\bauthor{\bsnm{Birg{\'e}},~\bfnm{Lucien}\binits{L.}} \AND
  \bauthor{\bsnm{Massart},~\bfnm{Pascal}\binits{P.}}
(\byear{1995}).
\btitle{Estimation of integral functionals of a density}.
\bjournal{Ann. Statist.}
\bvolume{23}
\bpages{11--29}.
\end{barticle}
\endbibitem

\bibitem[\protect\citeauthoryear{Boucheron, Lugosi and
  Massart}{2013}]{BouLugMas13}
\begin{bbook}[author]
\bauthor{\bsnm{Boucheron},~\bfnm{S.}\binits{S.}},
  \bauthor{\bsnm{Lugosi},~\bfnm{G.}\binits{G.}} \AND
  \bauthor{\bsnm{Massart},~\bfnm{P.}\binits{P.}}
(\byear{2013}).
\btitle{Concentration inequalities. A nonasymptotic theory of independence,
  With a foreword by Michel Ledoux}.
\bpublisher{Oxford University Press}, \baddress{Oxford}.
\end{bbook}
\endbibitem

\bibitem[\protect\citeauthoryear{Butucea}{2007}]{Butucea07}
\begin{barticle}[author]
\bauthor{\bsnm{Butucea},~\bfnm{Cristina}\binits{C.}}
(\byear{2007}).
\btitle{Goodness-of-{F}it testing and quadratic functional estimation from
  indirect observations}.
\bjournal{Ann. Statist.}
\bvolume{35}
\bpages{1907--1930}.
\end{barticle}
\endbibitem

\bibitem[\protect\citeauthoryear{Butucea and Issartel}{2021}]{Butucea21}
\begin{barticle}[author]
\bauthor{\bsnm{Butucea},~\bfnm{C.}\binits{C.}} \AND
  \bauthor{\bsnm{Issartel},~\bfnm{Y.}\binits{Y.}}
(\byear{2021}).
\btitle{Locally differentially private estimation of nonlinear functionals of
  discrete distributions}.
\bjournal{NeurIPS}
\bvolume{34}.
\end{barticle}
\endbibitem

\bibitem[\protect\citeauthoryear{Butucea, Rohde and
  Steinberger}{2022}]{Butucea22supp}
\begin{barticle}[author]
\bauthor{\bsnm{Butucea},~\bfnm{Cristina}\binits{C.}},
  \bauthor{\bsnm{Rohde},~\bfnm{Angelika}\binits{A.}} \AND
  \bauthor{\bsnm{Steinberger},~\bfnm{Lukas}\binits{L.}}
(\byear{2022}).
\btitle{Supplement to ``Interactive versus non-interactive locally
  differentially private estimation: Two elbows for the quadratic
  functional''}.
\end{barticle}
\endbibitem

\bibitem[\protect\citeauthoryear{Butucea and Tsybakov}{2008}]{Butucea08}
\begin{barticle}[author]
\bauthor{\bsnm{Butucea},~\bfnm{Cristina}\binits{C.}} \AND
  \bauthor{\bsnm{Tsybakov},~\bfnm{Alexandre~B}\binits{A.~B.}}
(\byear{2008}).
\btitle{Sharp optimality in density deconvolution with dominating bias. I}.
\bjournal{Theory Probab. Appl.}
\bvolume{52}
\bpages{24--39}.
\end{barticle}
\endbibitem

\bibitem[\protect\citeauthoryear{Butucea et~al.}{2020}]{Butucea19}
\begin{barticle}[author]
\bauthor{\bsnm{Butucea},~\bfnm{Cristina}\binits{C.}},
  \bauthor{\bsnm{Dubois},~\bfnm{Amandine}\binits{A.}},
  \bauthor{\bsnm{Kroll},~\bfnm{Martin}\binits{M.}} \AND
  \bauthor{\bsnm{Saumard},~\bfnm{Adrien}\binits{A.}}
(\byear{2020}).
\btitle{Local differential privacy: Elbow effect in optimal density estimation
  and adaptation over Besov ellipsoids}.
\bjournal{Bernoulli}
\bvolume{26}
\bpages{1727--1764}.
\bdoi{10.3150/19-BEJ1165}
\end{barticle}
\endbibitem

\bibitem[\protect\citeauthoryear{Cai and Low}{2005}]{Cai05}
\begin{barticle}[author]
\bauthor{\bsnm{Cai},~\bfnm{T.~Tony}\binits{T.~T.}} \AND
  \bauthor{\bsnm{Low},~\bfnm{Mark~G.}\binits{M.~G.}}
(\byear{2005}).
\btitle{Nonquadratic estimators of a quadratic functional}.
\bjournal{Ann. Statist.}
\bvolume{33}
\bpages{2930--2956}.
\end{barticle}
\endbibitem

\bibitem[\protect\citeauthoryear{Cai and Low}{2006}]{Cai06}
\begin{barticle}[author]
\bauthor{\bsnm{Cai},~\bfnm{T.~Tony}\binits{T.~T.}} \AND
  \bauthor{\bsnm{Low},~\bfnm{Mark~G.}\binits{M.~G.}}
(\byear{2006}).
\btitle{Optimal adaptive estimation of a quadratic functional}.
\bjournal{Ann. Statist.}
\bvolume{34}
\bpages{2298--22325}.
\end{barticle}
\endbibitem

\bibitem[\protect\citeauthoryear{Cai, Wang and Zhang}{2019}]{Cai19}
\begin{barticle}[author]
\bauthor{\bsnm{Cai},~\bfnm{T~Tony}\binits{T.~T.}},
  \bauthor{\bsnm{Wang},~\bfnm{Yichen}\binits{Y.}} \AND
  \bauthor{\bsnm{Zhang},~\bfnm{Linjun}\binits{L.}}
(\byear{2019}).
\btitle{The cost of privacy: Optimal rates of convergence for parameter
  estimation with differential privacy}.
\end{barticle}
\endbibitem

\bibitem[\protect\citeauthoryear{Collier, Comminges and
  Tsybakov}{2017}]{Collier17}
\begin{barticle}[author]
\bauthor{\bsnm{Collier},~\bfnm{Olivier}\binits{O.}},
  \bauthor{\bsnm{Comminges},~\bfnm{La\"etitia}\binits{L.}} \AND
  \bauthor{\bsnm{Tsybakov},~\bfnm{Alexandre}\binits{A.}}
(\byear{2017}).
\btitle{Minimax estimation of linear and quadratic functionals on sparsity
  classes}.
\bjournal{Ann. Statist.}
\bvolume{45}
\bpages{923--958}.
\end{barticle}
\endbibitem

\bibitem[\protect\citeauthoryear{de~la Pe\~{n}a and
  Montgomery-Smith}{1995}]{delaPena95}
\begin{barticle}[author]
\bauthor{\bparticle{de~la} \bsnm{Pe\~{n}a},~\bfnm{V.~H.}\binits{V.~H.}} \AND
  \bauthor{\bsnm{Montgomery-Smith},~\bfnm{S.~J.}\binits{S.~J.}}
(\byear{1995}).
\btitle{Decoupling inequalities for the tail probabilities of multivariate
  $U-$statistics}.
\bjournal{Ann. Probab.}
\bvolume{23}.
\end{barticle}
\endbibitem

\bibitem[\protect\citeauthoryear{Devore, Jawerth and Popov}{1992}]{Devore92}
\begin{barticle}[author]
\bauthor{\bsnm{Devore},~\bfnm{R.~A.}\binits{R.~A.}},
  \bauthor{\bsnm{Jawerth},~\bfnm{B.}\binits{B.}} \AND
  \bauthor{\bsnm{Popov},~\bfnm{V.}\binits{V.}}
(\byear{1992}).
\btitle{Compression of wavelet coefficients}.
\bjournal{Amer. J. Math.}
\bvolume{114}
\bpages{737--785}.
\end{barticle}
\endbibitem

\bibitem[\protect\citeauthoryear{Donoho and Nussbaum}{1990}]{Donoho90}
\begin{barticle}[author]
\bauthor{\bsnm{Donoho},~\bfnm{David~L.}\binits{D.~L.}} \AND
  \bauthor{\bsnm{Nussbaum},~\bfnm{Michael}\binits{M.}}
(\byear{1990}).
\btitle{Minimax quadratic estimation of a quadratic functional}.
\bjournal{J. Complexity}
\bvolume{6}
\bpages{290--323}.
\end{barticle}
\endbibitem

\bibitem[\protect\citeauthoryear{Duchi, Jordan and
  Wainwright}{2013a}]{Duchi13a}
\begin{binproceedings}[author]
\bauthor{\bsnm{Duchi},~\bfnm{J.~C.}\binits{J.~C.}},
  \bauthor{\bsnm{Jordan},~\bfnm{M.~I.}\binits{M.~I.}} \AND
  \bauthor{\bsnm{Wainwright},~\bfnm{M.~J.}\binits{M.~J.}}
(\byear{2013}a).
\btitle{Local privacy and statistical minimax rates}.
In \bbooktitle{2013 IEEE 54th Annual Symposium on Foundations of Computer
  Science}
\bpages{429-438}.
\bdoi{10.1109/FOCS.2013.53}
\end{binproceedings}
\endbibitem

\bibitem[\protect\citeauthoryear{Duchi, Jordan and
  Wainwright}{2013b}]{Duchi13b}
\begin{binproceedings}[author]
\bauthor{\bsnm{Duchi},~\bfnm{John}\binits{J.}},
  \bauthor{\bsnm{Jordan},~\bfnm{Michael~I}\binits{M.~I.}} \AND
  \bauthor{\bsnm{Wainwright},~\bfnm{Martin~J}\binits{M.~J.}}
(\byear{2013}b).
\btitle{Local privacy and minimax bounds: Sharp rates for probability
  estimation}.
In \bbooktitle{Adv. Neural Inf. Process. Syst.}
\bpages{1529--1537}.
\end{binproceedings}
\endbibitem

\bibitem[\protect\citeauthoryear{Duchi, Jordan and Wainwright}{2014}]{Duchi14}
\begin{barticle}[author]
\bauthor{\bsnm{Duchi},~\bfnm{John~C.}\binits{J.~C.}},
  \bauthor{\bsnm{Jordan},~\bfnm{Michael~I.}\binits{M.~I.}} \AND
  \bauthor{\bsnm{Wainwright},~\bfnm{Martin~J.}\binits{M.~J.}}
(\byear{2014}).
\btitle{Local privacy, data processing inequalities, and statistical minimax
  rates}.
\bjournal{arXiv preprint arXiv:1302.3203}.
\end{barticle}
\endbibitem

\bibitem[\protect\citeauthoryear{Duchi, Jordan and Wainwright}{2018}]{Duchi17}
\begin{barticle}[author]
\bauthor{\bsnm{Duchi},~\bfnm{John~C}\binits{J.~C.}},
  \bauthor{\bsnm{Jordan},~\bfnm{Michael~I}\binits{M.~I.}} \AND
  \bauthor{\bsnm{Wainwright},~\bfnm{Martin~J}\binits{M.~J.}}
(\byear{2018}).
\btitle{Minimax optimal procedures for locally private estimation}.
\bjournal{J. Amer. Statist. Assoc.}
\bvolume{113}
\bpages{182--201}.
\bdoi{10.1080/01621459.2017.1389735}
\end{barticle}
\endbibitem

\bibitem[\protect\citeauthoryear{Duchi and Ruan}{2018}]{Duchi18v1}
\begin{barticle}[author]
\bauthor{\bsnm{Duchi},~\bfnm{John~C}\binits{J.~C.}} \AND
  \bauthor{\bsnm{Ruan},~\bfnm{Feng}\binits{F.}}
(\byear{2018}).
\btitle{The right complexity measure in locally private estimation: It is not
  the Fisher information}.
\bjournal{https://arxiv.org/abs/1806.05756v1}.
\end{barticle}
\endbibitem

\bibitem[\protect\citeauthoryear{Duchi and Ruan}{2020}]{Duchi18v3}
\begin{barticle}[author]
\bauthor{\bsnm{Duchi},~\bfnm{John~C}\binits{J.~C.}} \AND
  \bauthor{\bsnm{Ruan},~\bfnm{Feng}\binits{F.}}
(\byear{2020}).
\btitle{The right complexity measure in locally private estimation: It is not
  the Fisher information}.
\bjournal{https://arxiv.org/abs/1806.05756v3}.
\end{barticle}
\endbibitem

\bibitem[\protect\citeauthoryear{Dwork et~al.}{2006}]{Dwork06}
\begin{binproceedings}[author]
\bauthor{\bsnm{Dwork},~\bfnm{Cynthia}\binits{C.}},
  \bauthor{\bsnm{McSherry},~\bfnm{Frank}\binits{F.}},
  \bauthor{\bsnm{Nissim},~\bfnm{Kobbi}\binits{K.}} \AND
  \bauthor{\bsnm{Smith},~\bfnm{Adam}\binits{A.}}
(\byear{2006}).
\btitle{Calibrating noise to sensitivity in private data analysis}.
In \bbooktitle{Theory of Cryptography}
(\beditor{\bfnm{Shai}\binits{S.}~\bsnm{Halevi}} \AND
  \beditor{\bfnm{Tal}\binits{T.}~\bsnm{Rabin}}, eds.).
\bseries{Lecture Notes in Computer Science}
\bpages{265--284}.
\bpublisher{Springer}.
\bdoi{10.1007/11681878}
\end{binproceedings}
\endbibitem

\bibitem[\protect\citeauthoryear{Evfimievski, Gehrke and
  Srikant}{2003}]{Evfim03}
\begin{binproceedings}[author]
\bauthor{\bsnm{Evfimievski},~\bfnm{Alexandre}\binits{A.}},
  \bauthor{\bsnm{Gehrke},~\bfnm{Johannes}\binits{J.}} \AND
  \bauthor{\bsnm{Srikant},~\bfnm{Ramakrishnan}\binits{R.}}
(\byear{2003}).
\btitle{Limiting privacy breaches in privacy preserving data mining}.
In \bbooktitle{Proceedings of the Twenty-Second ACM SIGMOD-SIGACT-SIGART
  Symposium on Principles of Database Systems}
\bpages{211--222}.
\bpublisher{ACM}.
\bdoi{10.1145/773153.773174}
\end{binproceedings}
\endbibitem

\bibitem[\protect\citeauthoryear{Gin{\'e} and Nickl}{2016}]{Gine16}
\begin{bbook}[author]
\bauthor{\bsnm{Gin{\'e}},~\bfnm{Evarist}\binits{E.}} \AND
  \bauthor{\bsnm{Nickl},~\bfnm{Richard}\binits{R.}}
(\byear{2016}).
\btitle{Mathematical foundations of infinite-dimensional statistical models}.
\bseries{Cambridge Series in Statistical and Probabilistic Mathematics}
\bvolume{40}.
\bpublisher{Cambridge University Press}.
\bdoi{10.1017/CBO9781107337862}
\end{bbook}
\endbibitem

\bibitem[\protect\citeauthoryear{H{\"a}rdle et~al.}{1998}]{Hardle98}
\begin{bbook}[author]
\bauthor{\bsnm{H{\"a}rdle},~\bfnm{Wolfgang}\binits{W.}},
  \bauthor{\bsnm{Kerkyacharian},~\bfnm{Gerard}\binits{G.}},
  \bauthor{\bsnm{Picard},~\bfnm{Dominique}\binits{D.}} \AND
  \bauthor{\bsnm{Tsybakov},~\bfnm{Alexander}\binits{A.}}
(\byear{1998}).
\btitle{Wavelets, Approximation, and Statistical Applications}.
\bseries{Lecture Notes in Statistics}
\bvolume{129}.
\bpublisher{Springer}.
\end{bbook}
\endbibitem

\bibitem[\protect\citeauthoryear{Houdr\'{e} and
  Reynaud-Bouret}{2003}]{HouRey2003}
\begin{bincollection}[author]
\bauthor{\bsnm{Houdr\'{e}},~\bfnm{Christian}\binits{C.}} \AND
  \bauthor{\bsnm{Reynaud-Bouret},~\bfnm{Patricia}\binits{P.}}
(\byear{2003}).
\btitle{Exponential inequalities, with constants, for {U}-statistics of order
  two}.
In \bbooktitle{Stochastic inequalities and applications}.
\bseries{Progr. Probab.}
\bvolume{56}
\bpages{55--69}.
\bpublisher{Birkh\"{a}user, Basel}.
\bmrnumber{2073426}
\end{bincollection}
\endbibitem

\bibitem[\protect\citeauthoryear{Joseph, Mao and Roth}{2020}]{Joseph20}
\begin{binproceedings}[author]
\bauthor{\bsnm{Joseph},~\bfnm{Matthew}\binits{M.}},
  \bauthor{\bsnm{Mao},~\bfnm{Jieming}\binits{J.}} \AND
  \bauthor{\bsnm{Roth},~\bfnm{Aaron}\binits{A.}}
(\byear{2020}).
\btitle{Exponential separations in local differential privacy}.
In \bbooktitle{Proceedings of the Fourteenth Annual ACM-SIAM Symposium on
  Discrete Algorithms}
\bpages{515--527}.
\bpublisher{SIAM}.
\end{binproceedings}
\endbibitem

\bibitem[\protect\citeauthoryear{Joseph et~al.}{2019}]{Joseph19}
\begin{binproceedings}[author]
\bauthor{\bsnm{Joseph},~\bfnm{Matthew}\binits{M.}},
  \bauthor{\bsnm{Mao},~\bfnm{Jieming}\binits{J.}},
  \bauthor{\bsnm{Neel},~\bfnm{Seth}\binits{S.}} \AND
  \bauthor{\bsnm{Roth},~\bfnm{Aaron}\binits{A.}}
(\byear{2019}).
\btitle{The role of interactivity in local differential privacy}.
In \bbooktitle{2019 IEEE 60th Annual Symposium on Foundations of Computer
  Science (FOCS)}
\bpages{94--105}.
\bpublisher{IEEE}.
\end{binproceedings}
\endbibitem

\bibitem[\protect\citeauthoryear{Kasiviswanathan
  et~al.}{2011}]{Kasiviswanathan11}
\begin{barticle}[author]
\bauthor{\bsnm{Kasiviswanathan},~\bfnm{Shiva~Prasad}\binits{S.~P.}},
  \bauthor{\bsnm{Lee},~\bfnm{Homin~K}\binits{H.~K.}},
  \bauthor{\bsnm{Nissim},~\bfnm{Kobbi}\binits{K.}},
  \bauthor{\bsnm{Raskhodnikova},~\bfnm{Sofya}\binits{S.}} \AND
  \bauthor{\bsnm{Smith},~\bfnm{Adam}\binits{A.}}
(\byear{2011}).
\btitle{What can we learn privately?}
\bjournal{SIAM Journal on Computing}
\bvolume{40}
\bpages{793--826}.
\end{barticle}
\endbibitem

\bibitem[\protect\citeauthoryear{Klemel{\"a}}{2006}]{Klemela06}
\begin{barticle}[author]
\bauthor{\bsnm{Klemel{\"a}},~\bfnm{Jussi}\binits{J.}}
(\byear{2006}).
\btitle{Sharp adaptive estimation of quadratic functionals}.
\bjournal{Probab. Theory Relat. Fields}
\bvolume{134}
\bpages{539--564}.
\end{barticle}
\endbibitem

\bibitem[\protect\citeauthoryear{Lam-Weil, Laurent and Loubes}{2020}]{Lam20}
\begin{barticle}[author]
\bauthor{\bsnm{Lam-Weil},~\bfnm{Joseph}\binits{J.}},
  \bauthor{\bsnm{Laurent},~\bfnm{B{\'e}atrice}\binits{B.}} \AND
  \bauthor{\bsnm{Loubes},~\bfnm{Jean-Michel}\binits{J.-M.}}
(\byear{2020}).
\btitle{Minimax optimal goodness-of-fit testing for densities under a local
  differential privacy constraint}.
\bjournal{arXiv preprint arXiv:2002.04254}.
\end{barticle}
\endbibitem

\bibitem[\protect\citeauthoryear{Laurent}{2005}]{Laurent2005}
\begin{barticle}[author]
\bauthor{\bsnm{Laurent},~\bfnm{B\'{e}atrice}\binits{B.}}
(\byear{2005}).
\btitle{Adaptive estimation of a quadratic functional of a density by model
  selection}.
\bjournal{ESAIM Probab. Stat.}
\bvolume{9}
\bpages{1--18}.
\bdoi{10.1051/ps:2005001}
\bmrnumber{2148958}
\end{barticle}
\endbibitem

\bibitem[\protect\citeauthoryear{Laurent and Massart}{2000}]{Laurent00}
\begin{barticle}[author]
\bauthor{\bsnm{Laurent},~\bfnm{Beatrice}\binits{B.}} \AND
  \bauthor{\bsnm{Massart},~\bfnm{Pascal}\binits{P.}}
(\byear{2000}).
\btitle{Adaptive estimation of a quadratic functional by model selection}.
\bjournal{Ann. Statist.}
\bpages{1302--1338}.
\end{barticle}
\endbibitem

\bibitem[\protect\citeauthoryear{Ritov and Bickel}{1990}]{Ritov90}
\begin{barticle}[author]
\bauthor{\bsnm{Ritov},~\bfnm{Ya'acov}\binits{Y.}} \AND
  \bauthor{\bsnm{Bickel},~\bfnm{Peter}\binits{P.}}
(\byear{1990}).
\btitle{Achieving information bounds in non and semiparametric models}.
\bjournal{Ann. Statist.}
\bvolume{18}
\bpages{925--938}.
\end{barticle}
\endbibitem

\bibitem[\protect\citeauthoryear{Rohde and Steinberger}{2020}]{Rohde18}
\begin{barticle}[author]
\bauthor{\bsnm{Rohde},~\bfnm{Angelika}\binits{A.}} \AND
  \bauthor{\bsnm{Steinberger},~\bfnm{Lukas}\binits{L.}}
(\byear{2020}).
\btitle{Geometrizing rates of convergence under local differential privacy
  constraints}.
\bjournal{Ann. Statist.}
\bvolume{48}
\bpages{2646--2670}.
\end{barticle}
\endbibitem

\bibitem[\protect\citeauthoryear{Smith}{2008}]{Smith08}
\begin{barticle}[author]
\bauthor{\bsnm{Smith},~\bfnm{Adam}\binits{A.}}
(\byear{2008}).
\btitle{Efficient, differentially private point estimators}.
\bjournal{arXiv preprint arXiv:0809.4794}.
\end{barticle}
\endbibitem

\bibitem[\protect\citeauthoryear{Smith}{2011}]{Smith11}
\begin{binproceedings}[author]
\bauthor{\bsnm{Smith},~\bfnm{Adam}\binits{A.}}
(\byear{2011}).
\btitle{Privacy-preserving statistical estimation with optimal convergence
  rates}.
In \bbooktitle{Proceedings of the Forty-Third Annual ACM Symposium on Theory of
  Computing}
\bpages{813--822}.
\bpublisher{ACM}.
\bdoi{10.1145/1993636.1993743}
\end{binproceedings}
\endbibitem

\bibitem[\protect\citeauthoryear{{T}sybakov}{2009}]{Tsybakov09}
\begin{bbook}[author]
\bauthor{\bsnm{{T}sybakov},~\bfnm{Alexandre~B.}\binits{A.~B.}}
(\byear{2009}).
\btitle{{I}ntroduction to Nonparametric Estimation}.
\bseries{Springer Series in Statistics}.
\bpublisher{Springer}, \baddress{New York}.
\bdoi{10.1007/b13794}
\end{bbook}
\endbibitem

\bibitem[\protect\citeauthoryear{Warner}{1965}]{Warner65}
\begin{barticle}[author]
\bauthor{\bsnm{Warner},~\bfnm{Stanley~L}\binits{S.~L.}}
(\byear{1965}).
\btitle{Randomized response: A survey technique for eliminating evasive answer
  bias}.
\bjournal{J. Amer. Statist. Assoc.}
\bvolume{60}
\bpages{63--69}.
\bdoi{10.1080/01621459.1965.10480775}
\end{barticle}
\endbibitem

\bibitem[\protect\citeauthoryear{Wasserman and Zhou}{2010}]{Wasserman10}
\begin{barticle}[author]
\bauthor{\bsnm{Wasserman},~\bfnm{Larry}\binits{L.}} \AND
  \bauthor{\bsnm{Zhou},~\bfnm{Shuheng}\binits{S.}}
(\byear{2010}).
\btitle{A statistical framework for differential privacy}.
\bjournal{J. Amer. Statist. Assoc.}
\bvolume{105}
\bpages{375--389}.
\bdoi{10.1198/jasa.2009.tm08651}
\end{barticle}
\endbibitem

\bibitem[\protect\citeauthoryear{Ye and Barg}{2017}]{Ye17}
\begin{barticle}[author]
\bauthor{\bsnm{Ye},~\bfnm{Min}\binits{M.}} \AND
  \bauthor{\bsnm{Barg},~\bfnm{Alexander}\binits{A.}}
(\byear{2017}).
\btitle{Asymptotically optimal private estimation under mean square loss}.
\bjournal{arXiv preprint arXiv:1708.00059}.
\end{barticle}
\endbibitem

\end{thebibliography}

\appendix

\section{Proofs of Section~3 (non-interactive protocols)}
\label{sec:App:NONINTER}
\subsection{Proof of Proposition~3.1}

It suffices to show that the one dimensional marginal channel satisfies (2.8). Note that the conditional density of $Z_1|X_1=x$ is given by
$$
q(z|x) = \prod_{j=-1}^{J-1} \prod_{k=0}^{(1\lor 2^j)-1} \frac{\alpha}{2 \sigma\cdot\sigma_j} \exp\left( - \frac{\alpha}{\sigma \cdot \sigma_j } |z_{jk} - \psi_{jk}( x )| \right), \quad z \in \r^{M},
$$
where $M = \sum_{j=-1}^{J-1} (1\lor 2^j)$. Therefore,
\begin{align*}
&\frac{q(z|x)}{q(z|x')} =
\prod_{j=-1}^{J-1} \prod_{k=0}^{(1\lor 2^j)-1}
\exp\left( - \frac{\alpha}{\sigma \cdot \sigma_j } ( |z_{jk} - \psi_{jk}( x )|
-  |z_{jk} - \psi_{jk}( x' )|) \right)\\
&\quad \leq  \exp \left( \frac \alpha{ \sigma}  | \phi(x) - \phi(x')|
 + \frac \alpha{ \sigma}\sum_{j=0}^{J-1} \frac { 2^{j/2}}{\sigma_j} \sum_{k=0}^{(1\lor 2^j)-1} |\psi(2^jx-k)- \psi(2^jx'-k)| \right)
\end{align*}
For any fixed $x \not = x'\in[0,1]$, $| \phi(x) - \phi(x')| \le 1$ and, for $j \geq 0$,
$$
\sum_{k=0}^{(1\lor 2^j)-1}  |\psi(2^j x - k) - \psi(2^j x' - k)|  \leq 2.
$$
Thus,
\begin{eqnarray*}
\frac{q(z|x)}{q(z|x')} & \leq & \exp \left( \frac{\alpha}{\sigma} 2\left[ 2 + \sum_{j=1}^{J-1} \frac {1}{j^a}\right]\right)
\leq \exp(\alpha),
\end{eqnarray*}
for $\sigma = 4 + 2\sum_{j=1}^{\infty} \frac {1}{j^a}$. \hfill\qed


\subsection{Proof of Theorem~3.2}

Fix $f\in\bar{\P}_s^{pq}(L, M_2, M_3)$. It first follows from Lemma~\ref{lemma:bias} that the bias of the estimator $\hat{D}_n$ is bounded as follows:
$$
D - \E_{Q_f^{(NI)}}[\hat{D}_n] = \sum_{j \ge J} \|\beta_{j\cdot}\|_2^2 \leq
\left\{
\begin{array}{ll}
C 2^{-2 J s'}, & 0 \leq s' < 1\\
C 2^{-J \frac {5}3}, & s' \geq 1
\end{array}
\right. , \quad \text{ with } C>0.
$$
We remark that, in order to use only the Haar basis in our construction, a rough control of the bias is performed for $s' \geq 1$. We have to embed our Besov body into a larger one with smoothness parameter smaller than 1 and we chose a somehow arbitrary value $5/6$ that is larger than $3/4$. This is enough to get the parametric rate in the regime where $s' > \frac 34$.

Next, we study the variance of the private estimator $\hat{D}_n$ with $J\ge2$. Note that
\begin{eqnarray*}
\hat{D}_n - D &=& \frac 1{n(n-1)} \sum_{i \not = h}^n \sum_{j=-1}^{J-1} \sum_{k=0}^{(1\lor 2^j)-1} (Z_{ijk}-\beta_{jk})(Z_{hjk}-\beta_{jk})\\
&& + \frac 2n \sum_{i=1}^n \sum_{j=-1}^{J-1} \sum_{k=0}^{(1\lor 2^j)-1} (Z_{ijk}-\beta_{jk}) \beta_{jk} = : T_1 + T_2.
\end{eqnarray*}
Now,
\begin{eqnarray*}
\Var(T_1) &=& \frac{4}{n^2(n-1)^2} \Var\left( \sum_{i<h}^n \sum_j \sum_k (Z_{ijk}-\beta_{jk}) (Z_{hjk}-\beta_{jk}) \right)\\
&=&  \frac{2}{n (n-1)} \Var\left( \sum_j \sum_k (Z_{1jk}-\beta_{jk}) (Z_{2jk}-\beta_{jk}) \right).
\end{eqnarray*}
We can decompose the previous sum into uncorrelated terms as follows:
\begin{align*}
\sum_j \sum_k &(Z_{1jk}-\beta_{jk}) (Z_{2jk}-\beta_{jk}) = \sum_j \sum_k (\psi_{jk}(X_1)-\beta_{jk}) (\psi_{jk}(X_2)-\beta_{jk})\\
 &+\sum_j \sum_k\sigma_j \frac{\sigma}{\alpha} \left[ (\psi_{jk}(X_1)-\beta_{jk}) W_{2jk}+  W_{1jk}(\psi_{jk}(X_2)-\beta_{jk}) \right]\\
&+ \sum_j \sum_k \sigma_j^2 \frac{\sigma^2}{\alpha^2} W_{1jk}W_{2jk}.
 \end{align*}
Therefore, using independence, the inequality $(a-b)^2\le 2a^2+2b^2$ and $\E(W_{ijk}^2)=2$, we get
\begin{align*}
\Var(T_1) &=  \frac{2}{n (n-1)} \left\{\E\left[\left( \sum_j \sum_k (\psi_{jk}(X_1)-\beta_{jk}) (\psi_{jk}(X_2)-\beta_{jk}) \right)^2\right]\right. \\
&\quad+ 2 \cdot \frac{\sigma^2}{\alpha^2} \E\left[ \left(  \sum_j  \sum_k \sigma_j(\psi_{jk}(X_1)-\beta_{jk}) W_{1jk}\right)^2\right] \\
&\quad+ \left.\frac{\sigma^4}{\alpha^4} \left[\E(W_{1jk}^2)\right]^2(2 + \sum_{j=1}^{J-1} \sigma_j^4 \cdot 2^{j} ) \right\}\\
&\leq  \frac{2}{n(n-1)}\left\{
	2\E\left[\left(\sum_j \sum_k \psi_{jk}(X_1) \psi_{jk}(X_2) \right)^2 \right]
		+ 2 \left(\sum_j \sum_k \beta_{jk}^2\right)^2 \right.\\
&\quad+ 8 \frac{\sigma^2}{\alpha^2} \E \left[ \left(\sum_j  \sum_k \sigma_j\psi_{jk} (X_1)  \right)^2\right]
		+ 8 \frac{\sigma^2}{\alpha^2} \sum_j  \sum_k \sigma_j^2\beta_{jk}^2\\
&\quad + 4 J^{4a+1} 2^{3J} \frac{\sigma^4}{\alpha^4} \Bigg\}.
\end{align*}
Now, we easily see that
$$
\left(\sum_j \sum_k \beta_{jk}^2\right)^2 = D^2 \le M_2^4,
$$
and that
$$
\sum_j  \sum_k \sigma_j^2\beta_{jk}^2 = \sum_{j=-1}^{J-1} j^{2a} 2^j \|\beta_{j\cdot}\|_2^2 \le J^{2a+1} 2^J.
$$
Furthermore, by Jensen's inequality and the fact that for $k_1\ne k_2$, the basis functions $\psi_{jk_1}$ and $\psi_{jk_2}$ have disjoint support, we get
\begin{align*}
&\E\left[\left(\sum_j \sum_k \psi_{jk}(X_1) \psi_{jk}(X_2) \right)^2\right]
\leq (J+1) \cdot \E \left[ \sum_j \left(\sum_k \psi_{jk}(X_1) \psi_{jk}(X_2) \right)^2 \right]\\
&\quad\leq (J+1) \sum_j \sum_k \left(\E \left[\psi_{jk}^2(X_1) \right]\right)^2\\
&\quad= (J+1) \sum_j \sum_k \left(\E \left[(1\lor 2^j)\mathds 1_{[0,1]}((1\lor 2^j)X_1-k) \right]\right)^2\\
&\quad= (J+1) \sum_j \left((1\lor 2^j) \E \left[\mathds 1_{[0,1]}((1\lor 2^j)X_1) \right]\right)^2 \\
&\quad\le (J+1) \sum_{j=-1}^{J-1}  (1\lor 2^j)^2 \leq (J+1) \cdot 2^{2J}.
\end{align*}
Moreover, similar but simpler considerations yield
\begin{align*}
&\E \left[ \left(\sum_j \sigma_j \sum_k\psi_{jk} (X_1)  \right)^2\right]
\leq  (J+1) \cdot\sum_j \sigma_j^2 \E\left[ \left( \sum_k\psi_{jk} (X_1)  \right)^2\right]\\
&\quad\le (J+1) \cdot \sum_{j=-1}^{J-1} \sigma_j^2 (1\vee 2^j)
\leq (J+1) \left(2 + \sum_{j=1}^{J-1} j^{2a} 2^{2j}\right)\\
&\quad\leq (J+1)^{2a+2} 2^{2J}.
\end{align*}
In conclusion, there exists some constant $C>0$, not depending on $f$, $n$ or $\alpha\le 1$, such that
$$
\Var(T_1) \leq C\left( \frac{J \cdot 2^{2J}}{n^2} + \frac{J^{2a+2}2^{2J}}{n^2\alpha^2} + \frac{J^{4a+1}2^{3J}}{n^2\alpha^4}\right) \le 3C \frac{J^{4a+1}2^{3J}}{n^2\alpha^4}.
$$
Next,
\begin{eqnarray*}
\Var(T_2) &=& \frac 4n \Var\left(\sum_j \sum_k (Z_{1jk}-\beta_{jk}) \beta_{jk} \right)\\
&=& \frac 4n \Var\left(\sum_j \sum_k (\psi_{jk}(X_1)-\beta_{jk}) \beta_{jk} \right)+\frac 4n \Var\left(\sum_j \sum_k \sigma_j \frac{\sigma}{\alpha} W_{1jk} \beta_{jk} \right)\\
&=& \frac 4n \Var\left(\sum_j \sum_k \psi_{jk}(X_1) \beta_{jk} \right)+\frac 4n  \frac{\sigma^2}{\alpha^2} \sum_j \sigma_j^2 \sum_k 2 \beta_{jk}^2\\
&=& \frac 4n \left(\int_0^1 (P_Jf)^2\cdot f - \left(\int_0^1 (P_Jf) \cdot f\right)^2 \right) + \frac{8 \sigma^2 }{n \alpha^2} \sum_{j=-1}^{J-1} \sigma_j^2 \|\beta_{j \cdot}\|_2^2,
\end{eqnarray*}
where $P_Jf = \sum_{j=-1}^{J-1}\sum_k \beta_{jk}\psi_{jk}= \sum_{m=0}^{2^J-1}\alpha_{Jm}\phi_{Jm}$ is the projection of $f$ onto the linear space $V_J = \s\{\phi_{Jm} : m=0,\dots, 2^J-1\}$, $\phi_{Jm} = 2^{J/2}\phi(2^Jx-m)$ and $\alpha_{Jm} = \langle f,\phi_{Jm}\rangle$. Because of the special structure of the Haar basis functions $\phi_{Jm}$ we have $(P_Jf)^2\in V_J$ and therefore $(P_Jf)^2\perp (f-P_Jf)\in V_J^\perp$, so that we obtain
\begin{align*}
\int_0^1 (P_Jf)^2\cdot f - \left(\int_0^1 (P_Jf)\cdot f\right)^2 &\leq\int_0^1 (P_Jf)^3
 =
\int_0^1 \left(\sum_k \alpha_{Jk} \phi_{Jk}(x)\right)^3 dx\\
& = \int_0^1 \sum_k \alpha_{Jk}^3 \phi_{Jk}^3(x) dx = \sum_k 2^{J/2} \alpha_{Jk}^3 \\
& = \sum_k 2^{J/2} \left( \int_0^1 f\cdot \phi_{Jk}\right)^3.
\end{align*}
But Jensen's inequality yields
\begin{align*}
\sum_k 2^{J/2}&\left|2^{-J/2}\int_0^1 f\cdot2^{J/2} \phi_{Jk}\right|^3
\le
\sum_k 2^{J/2}2^{-3J/2}\int_0^1 f^3\cdot 2^{J/2}\phi_{Jk}\\
&=
\int_0^1 f^3(x)\cdot \sum_{k=0}^{2^J-1}\phi(2^Jx-k)dx = \int_0^1f^3 \le M_3^3.
\end{align*}

On the other hand, recall that for the bias part we already showed that $\|\beta_{j \cdot }\|_2 \lesssim 2^{-js'} \cdot \mathds 1_{\{0 < s'< 1\}} + 2^{- \frac 56 j} \cdot \mathds 1_{\{s' \geq 1\}}$ where $s' = s - \left( \frac 1p - \frac 12\right)_+$. Thus
\begin{eqnarray*}
\sum_{j=-1}^{J-1} \sigma_j^2 \|\beta_{j \cdot}\|_2^2 &\lesssim & 2 +
\sum_{j=1}^{J-1} j^{2a} \left[2^{j(1-2s')} \cdot \mathds 1_{\{0 < s'< 1\}} + 2^{j (1- \frac 53)} \cdot \mathds 1_{\{s' \geq 1\}}\right]\\
&\lesssim & 1 \vee (J^{2a+1} \cdot 2^{J(1-2s')}).
\end{eqnarray*}

Thus, for some constant $C>0$,
$$
\Var(T_2) \leq \frac{4 M_3}{n} + \frac{C}{n \alpha^2} \cdot \left\{ 1 \vee (J^{2a+1} \cdot 2^{J(1-2s')}) \right\}.
$$
Summing up the previous bounds, we get
\begin{align*}
\E_{Q_f^{(NI)}}[(\hat{D}_n - D)^2] &\lesssim  2^{-4Js'} \mathds 1_{\{0 < s'< 1\}}  + 2^{-J \frac{10}3} \mathds 1_{\{s' \geq 1\}} \\
&\quad+ \frac{J^{4a +1}2^{3J}}{n^2\alpha^4} + \frac{1}{n} + \frac{1 \vee (J^{2a+1} \cdot 2^{J(1-2s')})}{n \alpha^2}.
\end{align*}
With our choice of $J$ the result of Theorem~3.2 follows, because
\begin{align*}
&\frac{1 \vee (J^{2a+1} \cdot 2^{J(1-2s')})}{n \alpha^2}
=
\frac{1 \vee \left([\frac{2}{4s'+3}\log_2(n\alpha^2)]^{2a+1} \cdot (n\alpha^2)^{\frac{2(1-2s')}{4s'+3}}\right)}{n \alpha^2}\\
&\quad\le
\frac{1}{n\alpha^2} \vee \left(\left[\frac{2\log 2}{4s'+3}\log(n\alpha^2)\right]^{4a+1} \cdot (n\alpha^2)^{\frac{-8s'}{4s'+3}}\right).
\end{align*}\hfill\qed

\subsection{Proof of Theorem~3.3}

Fix a channel $Q\in\mathcal Q_\alpha^{(NI)}$ with marginal conditional densities $q_i(z_i|x_i)$, $i=1,\dots, n$ with respect to some reference probability measure $\mu_i$ on $\mathcal{Z}_i$, as in Lemma~\ref{lemma:Qdensities}, that is, $e^{-\alpha}\le q_i(z_i|x) \le e^\alpha$ and, in particular, $q_i(z_i|x) \le e^{2\alpha} q_i(z_i|x')$, for all $z_i\in\mathcal Z_i$ and all $x,x'\in[0,1]$.
The lines of proof are similar to those of \cite{Lam20} that we generalize in order to: a) take into account possibly different mechanisms $q_i$ for each $i$, b) consider Besov  smooth densities belonging to $B_s^{pq}$ with $s>0$, $p\geq 2$, $q\geq 1$, instead of $B_s^{2\infty}$ and c) get the (nearly) optimal dependence with respect to $\alpha$ when it tends to 0.

Let $f_0 = \mathds 1_{[0,1]}$ and denote by $g_{0,i} (z_i) = \int_0^1 q_i(z_i|x) dx\geq e^{-\alpha}$.
For any $i=1,\dots, n$, define the bounded linear operator $K_i: L_2([0,1])\rightarrow L_2(\mathcal{Z}_i,d\mu_i)$ by
$$
K_i f = \int_0^1 q_i(\cdot |x) f(x) \frac{dx}{\sqrt{g_{0,i}(\cdot)}}, \quad f \in \mathbb{L}_2([0,1]).
$$
Then with $K_i^\star$ denoting its adjoint, the operator $K_i^\star K_i$ is a symmetric integral operator with kernel $F_i(x,y)=\int q_i(z_i|x) q_i(z_i|y) / g_{0,i}(z_i) d\mu_i(z_i)$:
$$
K_i^\star K_i f(\cdot) = \int_{\mathcal{Z}_i}q_i(z_i|\cdot) \int_0^1  q_i(z_i|y) f(y) dy\frac{d\mu_i(z_i)}{g_{0,i}(z_i)} = \int_0^1 F_i(\cdot ,y) f(y) dy
$$
by Fubini's theorem. Next, let us note that $f_0$ is an eigenfunction of $K_i^\star K_i$, associated to the eigenvalue $\lambda_{0,i}=1$, for all $i$ from 1 to $n$:
$$
K_i^\star K_i f_0(x) = \int_{\mathcal{Z}_i}q_i(z_i\arrowvert x)\int_0^1q_i(z_i\arrowvert y)dy \frac{d \mu_i(z_i)}{g_{0,i}(z_i)}=\int_{\mathcal{Z}_i} q_i(z_i|x) d\mu_i(z_i) = 1
$$
for all $x \in [0,1]$. Now, define the operator
$$
K:=\frac{1}{n}\sum_{i=1}^nK_i^\star K_i.
$$
It is again symmetric and positive semidefinite and has the eigenfunction $w_0=f_0$ associated to the eigenvalue $\lambda_0=1$. It is also an integral operator with kernel $F(x,y)=n^{-1}\sum_{i=1}^nF_i(x,y)$. Recall the Haar wavelet functions $(\psi_{jk})$ of Section~2.1 and define
$$
W_m=\textrm{span}\big\{ \psi_{mk}: k=0,...,2^m-1\big\}
$$
as the linear subspace spanned by the orthonormal family consisting of $(\psi_{mk})_{k=0,...,2^m-1}$. Denote by $w_{1},...,w_{2^m}$ the  eigenfunctions of $K$ as an operator on the linear $L_2([0,1])$-subspace $W_m$, satisfying $\|w_{k}\|_{L_2}=1$ and $\int_0^1 w_{k}(x)dx=0$ since they are orthogonal to $f_0=\mathds 1_{[0,1]}$. Moreover, we write $\lambda_{1}^2,...,\lambda_{2^m}^2$ for the corresponding eigenvalues, respectively. Note that they are non-negative.

From now on, we denote by $z_\alpha  = e^{2\alpha} - e^{-2\alpha} \leq e^2$ for $\alpha$ in (0,1] and by
$$
\lambda_{k,\alpha,m} = ( \frac{ \lambda_k}{z_\alpha}) \vee 2^{-m/2} \geq 2^{-m/2}.
$$
Define the functions
$$
f_{\nu}(x)  = f_0(x) + 2^{-m(s + 1)} \delta \sum_{j=1}^{2^m} \nu_j \frac 1{\lambda_{j,\alpha,m} } \cdot w_{j}(x),\ \ x\in[0,1],
$$
where $\nu_j \in\{-1,1\}$, for $j=1,...,2^m$ and $\delta = \delta_m>0$ is to be specified later. By a slight abuse of notation, we identify $f_0$ with $f_\nu$ with $\nu  = (0,...,0)$. Lemma~\ref{lemma:alternatives} shows that for the overwhelming part of possible vectors $\nu$, $f_{\nu}$ is a density, belongs to the right Besov space and the corresponding quadratic functional $D(f_{\nu})$ is sufficiently far away from $D(f_0)$.


We choose the integer number $m$ such that:
$$
n z^2_{ \alpha}  \asymp 2^{2ms +3m/2}.
$$

Let us denote by $g_{\nu,i}$ the function $g_{\nu,i}(z_i) = \int_0^1 q_i(z_i| x)f_\nu(x)dx$, $z_i\in\mathcal Z_i$, and see that
$$
g_{\nu,i}(z_i) = g_{0,i}(z_i) + 2^{-m(s+\frac 12)} \delta \sum_{k=1}^{2^m} \nu_k \frac 1{\lambda_{k,\alpha,m}} \cdot \int_0^1 q_i(z_i|x) w_{k}(x) dx.
$$
Classical results allow us to reduce the lower bounds for estimating $f$ to testing between the probability measures
\begin{equation}\label{eq:Q0n}
d Q_{0,n}(z_1,...,z_n) := \prod_{i=1}^n g_{0,i}(z_i)d\mu_i(z_i)
\end{equation} 
(where, for $\mu:=\bigotimes_{i=1}^n \mu_i$, $\prod_{i=1}^n g_{0,i}(z_i)$ is also the $\mu$-density of the product measure $Q\Pr_{f_0}^n = \bigotimes_{i=1}^n (Q_i\Pr_{f_0})$) and the averaged alternative
\begin{equation}\label{eq:Qn}
d Q_n(z_1,...,z_n) := \E_\nu \left[ \prod_{i=1}^n {g_{\nu,i}}(z_i) d\mu_i(z_i)\right] ,
\end{equation}
where $\E_\nu$ stands for expectation over i.i.d. Rademacher random variables $\nu_k$.
Indeed, using Lemma~\ref{lemma:alternatives}, we first reduce the maximal risk over $\bar{\P}_s^{pq}$ to the maximal risk over the subfamily of pdf's $\{ f_\nu: \nu = 0 \text{ or }\nu \in A_\gamma\}\subseteq \bar{\P}_s^{pq}$ and then use the Markov inequality with $\Delta = \delta^2\cdot 2^{-2ms}$ to get
\begin{align}
&\inf_{Q \in {\mathcal{Q}}_\alpha^{(NI)}} \inf_{\hat D_n} \sup_{f \in \bar{\mathcal{P}}_s^{pq}(L,M)} \E_{Q\Pr_{f}^n} \left[ |\hat D_n - D(f)|^2\right] \nonumber \\
&\quad \geq \inf_{Q \in {\mathcal{Q}}_\alpha^{(NI)}} \inf_{\hat D_n} \sup_{ \nu\in A_\gamma\cup\{0\}} \left(\frac \Delta2 \right)^2 Q\Pr_{f_\nu}^n \left(|\hat D_n - D(f_\nu)| \geq \frac \Delta2 \right)\nonumber  \\
&\quad \geq  \left(\frac \Delta2 \right)^2 \inf_{Q \in {\mathcal{Q}}_\alpha^{(NI)}} \inf_{\hat D_n}
\max \Bigg\{ Q\Pr_{f_0}^n\left(|\hat D_n - D(f_0)|\geq \frac \Delta2\right) ,  \nonumber \\
&\hspace{2cm} \E_\nu \left[ \mathds 1_{A_\gamma}(\nu)  Q\Pr_{f_\nu}^n \left(|\hat D_n - D(f_\nu)|\geq \frac \Delta2\right) \right] \Bigg\} \nonumber  \\
&\quad\geq  \left(\frac \Delta2 \right)^2 \inf_{Q \in {\mathcal{Q}}_\alpha^{(NI)}} \inf_{\hat D_n}
\max \Bigg\{ Q\Pr_{f_0}^n(B) , \notag \\
&\hspace{2cm}  \E_\nu \left[ \mathds 1_{A_\gamma}(\nu) \E_{Q \Pr_{f_0}^n} \left(\frac{d(Q\Pr_{f_\nu}^n)}{d(Q\Pr_{f_0}^n)} \cdot \mathds 1_{|\hat D_n - D(f_\nu)|\geq \frac \Delta2} \right) \right] \Bigg\} , \label{term:ave}
\end{align}
where we denote by $B$ the event $\{ |\hat D_n - D(f_0)|\geq \frac \Delta2 \}$. Note that, because of $|D(f_\nu) - D(f_0)| \geq \Delta$ (cf. Lemma~\ref{lemma:alternatives}.(iii)), the complementary event of $B$, $\bar B$, implies that $|\hat D_n - D(f_\nu)|\geq \frac \Delta2$. Thus, for any $\tau\in(0,1)$ we can further bound from below the term in (\ref{term:ave}) by
\begin{align*}
&\E_{Q \Pr_{f_0}^n} \left(\E_\nu \left[ \mathds 1_{A_\gamma}(\nu) \frac{d(Q\Pr_{f_\nu}^n)}{d(Q\Pr_{f_0}^n)} \cdot \mathds 1_{\bar B} \right] \right)
=\E_{Q_{0,n}} \left(\E_\nu \left[ \mathds 1_{A_\gamma}(\nu) \frac{d(Q\Pr_{f_\nu}^n)}{dQ_{0,n}} \right] \cdot\mathds 1_{\bar B} \right) \\
&\quad =  \E_{Q_{0,n}} \left( \left[ \frac{dQ_n}{dQ_{0,n}} - \E_\nu \left( \mathds 1_{\bar{A}_\gamma}(\nu) \frac{d(Q\Pr_{f_\nu}^n)}{dQ_{0,n}}\right)\right] \cdot \mathds 1_{\bar B} \right)  \\
&\quad \geq  \E_{Q_{0,n}} \left( \frac{dQ_n}{dQ_{0,n}}  \cdot \mathds 1_{\bar B} \right)
- \E_\nu \left[ \mathds 1_{\bar{A}_\gamma}(\nu) \E_{Q_{0,n}} \frac {d(Q\Pr_{f_\nu}^n)}{dQ_{0,n}} \right] \\
&\quad \geq  \E_{Q_{0,n}} \left(  \tau \cdot \mathds 1_{\frac{dQ_n}{dQ_{0,n}}\geq \tau } \cdot \mathds 1_{\bar B} \right)-\gamma \\
&\quad \geq  \tau \cdot \left( Q_{0,n} \left( \frac{dQ_n}{dQ_{0,n}}\geq \tau \right) - Q_{0,n} (B) \right) -\gamma,
\end{align*}
where we have used that $\E_\nu[\mathds 1_{\bar{A}_\gamma}(\nu)] \le \gamma$, by Lemma~\ref{lemma:alternatives}.
If there exist $\epsilon,\tau \in (0,1)$ and $n_0\in\N$, all three not depending on $n$ and $\alpha$, such that, whenever $n z_\alpha^2\ge n_0$, we have
\begin{equation}\label{eq:AN}
Q_{0,n} \left( \frac{dQ_n}{dQ_{0,n}}\geq \tau \right) \geq 1-\epsilon, 
\end{equation}
then we conclude the proof by the following lower bound for the minimax risk
\begin{eqnarray*}
\mathcal{M}_{n,\alpha}^{(NI)}(\bar{\mathcal{P}}_s^{pq}(L,M) )&\geq &
\left(\frac \Delta2 \right)^2 \Bigg(\inf_{Q \in {\mathcal{Q}}_\alpha^{(NI)}} \inf_{\hat D_n}
\max \Bigg\{ Q_{0,n}(B) , \\\
&&  +\tau \cdot \left( Q_{0,n} \left( \frac{dQ_n}{dQ_{0,n}}\geq \tau \right) - Q_{0,n} (B) \right) \Bigg\} - \gamma \Bigg) ,\\
&\geq &
\left(\frac \Delta2 \right)^2 \left(\inf_{p \in (0,1)} \max\{p, \tau (1-\epsilon - p)\} -\gamma\right)\\
&=&  \left(\frac \Delta2 \right)^2 \cdot \left(\frac \tau{1+\tau} (1-\epsilon) - \gamma\right),
\end{eqnarray*}
provided that $n z_\alpha^2\ge n_0$.
Thus, for an appropriate choice of $\gamma\in(0,1)$ and with $\delta=\delta_m$ as in Lemma~\ref{lemma:alternatives}, $\Delta^2 = \delta^4 2^{-4ms} \asymp (n z^2_{\alpha} )^{-8s/(4s+3)} [\log(n z_\alpha^2)]^{-2}$ is the desired rate.

A sufficient condition for \eqref{eq:AN} is that for $n z_\alpha^2\ge n_0$,
$$
\chi^2 (Q_n, Q_{0,n}):=\int \left(\frac{dQ_n}{dQ_{0,n}}-1\right)^2 dQ_{0,n} \leq (1-\tau)^4.
$$
Indeed, $Q_{0,n}(\frac{dQ_n}{dQ_{0,n}} \geq \tau) \geq 1 - \frac 1{(1-\tau)^2} \int (\frac{dQ_n}{dQ_{0,n}}-1)^2 dQ_{0,n}\geq 1 - (1-\tau)^2$ and this checks (\ref{eq:AN}) with $\epsilon = (1-\tau)^2$.

We have
\begin{align*}
\chi^2 &(Q_n, Q_{0,n}) =  -1 + \E_{Q_{0,n}}\left[ \Big(\frac{dQ_n}{dQ_{0,n}}\Big)^2 \right]\\
&=  -1 + \E_{Q_{0,n}}\left( \left[ \E_{\nu}\prod_{i=1}^n \left( 1 + 2^{-m(s+1)} \delta \sum_{k=1}^{2^m} \frac {\nu_k }{\lambda_{k,\alpha,m}} \cdot \frac{\langle q_i(Z_i|\cdot ) ,w_{k}\rangle }{g_{0,i}(Z_i)} \right) \right]^2\right)\\
& = -1 + \E_{Q_{0,n}} \left[ \E_{\nu,\nu'}\prod_{i=1}^n \left( 1 + 2^{-m(s+1)} \delta \sum_{k=1}^{2^m} \frac {\nu_k }{\lambda_{k,\alpha,m}} \cdot \frac{\langle q_i(Z_i|\cdot ), w_{k}\rangle }{g_{0,i}(Z_i)} \right) \cdot\right. \\
&\quad \left.  \left( 1 + 2^{-m(s+1)} \delta \sum_{k=1}^{2^m} \frac {\nu'_k }{\lambda_{k,\alpha,m}} \cdot \frac{\langle q_i(Z_i|\cdot ), w_{k}\rangle }{g_{0,i}(Z_i)} \right) \right]\\
&= -1 + \E_{\nu,\nu'}\prod_{i=1}^n \left( 1+ \E_{Q_{0,n}}
2^{-2m(s+1)}\delta^2 \sum_{k_1,k_2=1}^{2^m}\frac {\nu_{k_1} }{\lambda_{k_1,\alpha,m}} \cdot \frac{\langle q_i(Z_i|\cdot ), w_{k_1}\rangle }{g_{0,i}(Z_i)}\cdot\right. \\
&\quad \left.  \frac {\nu'_{k_2} }{\lambda_{k_2,\alpha,m}} \cdot \frac{\langle q_i(Z_i|\cdot ), w_{k_2}\rangle }{g_{0,i}(Z_i)}  \right),
\end{align*}
where $\nu,\nu'$ are independent copies of vectors with i.i.d. Rademacher entries and we used that
$$
\E_{Q_{0,n}}\left( \frac{\langle q_i(Z_i|\cdot ) ,w_{k}\rangle }{g_{0,i}(Z_i)} \right)
= \int_{\mathcal{Z}_i} \langle q_i(z_i|\cdot ) ,w_{k}\rangle d\mu_i(z_i)
= \int_0^1 w_{k}(x) dx = 0.
$$
Note that
\begin{align*}
& \sum_{k_1,k_2=1}^{2^m}\frac {\nu_{k_1} }{\lambda_{k_1,\alpha,m}} \cdot\frac {\nu'_{k_2} }{\lambda_{k_2,\alpha,m}} \E_{Q_{0,n}} \left[ \frac{\langle q_i(Z_i|\cdot ), w_{k_1}\rangle }{g_{0,i}(Z_i)}\cdot \frac{\langle q_i(Z_i|\cdot ), w_{k_2}\rangle }{g_{0,i}(Z_i)}\right]\\
& = \sum_{k_1,k_2=1}^{2^m} \frac {\nu_{k_1} \nu'_{k_2}}{\lambda_{k_1,\alpha,m} \lambda_{k_2,\alpha,m}} \int_0^1 \int_0^1 F_i(x,y) w_{k_1}(x) w_{k_2}(y) dx dy.
\end{align*}
Now, we use that $1+u \leq \exp(u)$ for all real numbers $u$ and since $w_{k}$ are orthonormal eigenfunctions of $K = \frac1n  \sum_{i=1}^n K_i^\star K_i$ we also have that
\begin{align*}
\chi^2 (Q_n, Q_{0,n}) &\leq -1 + \E_{\nu,\nu'} \exp\left( 2^{-2m(s+1)}\delta^2 \sum_{k_1,k_2=1}^{2^m}
 \frac {\nu_{k_1} \nu'_{k_2}}{\lambda_{k_1,\alpha,m} \lambda_{k_2,\alpha,m}}n \cdot\right.\\
&\hspace{3cm}\left. \int_0^1 \int_0^1  \frac 1n  \sum_{i=1}^n F_i(x,y) w_{k_1}(x) w_{k_2}(y) dx dy \right).
\end{align*}
Remember that
$$
\int_0^1 \int_0^1 \frac 1n \sum_{i=1}^n  F_i(x,y) w_{k_1}(x) w_{k_2}(y) dx dy 
=
\langle Kw_{k_1}, w_{k_2}\rangle = \lambda_{k_1}^2 \cdot\langle w_{k_1}, w_{k_2}\rangle,
$$
where $\lambda_k^2$ are eigenvalues of $K$.


We use that
\begin{align*}
& \frac{\lambda_{k}^2}{\lambda_{k,\alpha,m}^2 }= \frac{\lambda_{k}^2}{( z^{-2}_\alpha  \cdot \lambda_{k}^2 ) \vee 2^{-m} } \leq z^2_\alpha ,
\end{align*}
and that $\nu_k\nu'_k$, $k=1,\dots, 2^m$, are Rademacher distributed and independent, to further obtain
\begin{align*}
\chi^2 (Q_n, Q_{0,n}) &\leq - 1 + \E_\nu \exp \left(\sum_{k=1}^{2^m} 2^{-2m(s+1)}\delta^2\frac {\nu_k }{\lambda_{k,\alpha,m}^2} n  \lambda_{k}^2 \right)\\
&= -1 + \E_\nu\prod_{k=1}^{2^m}\exp \left( \nu_k 2^{-2m(s+1)} \delta^2 n \frac {\lambda^2_k }{\lambda_{k,\alpha,m}^2}\right)\\
&\leq - 1 + \prod_{k=1}^{2^m} \cosh (2^{-2m(s+1)}\delta^2 n z^2_\alpha  ).
\end{align*}
Let us further see that $\cosh(u)\leq \exp(u^2/2)$ for all real numbers $u$ and therefore
\begin{align*}
&\chi^2 (Q_n, Q_{0,n}) \leq - 1 + \exp\left( \frac 12 \sum_{k=1}^{2^m} (2^{-2m(s+1)} \delta^2 n z^2_\alpha )^2 \right)\\
& \leq -1+\exp( 2^{-m(4s+3)} \delta^4 \cdot n^2 z^4_\alpha ),
\end{align*}
which, for our choice of $m$ and $\delta = \delta_m$ (cf. Lemma~\ref{lemma:alternatives}), tends to $0$ as $n z_\alpha^2$ becomes large. This concludes the proof.\hfill\qed

\subsection{Auxiliary lemmas}

\begin{lemma}\label{lemma:alternatives}
Let $\mathbb{P}_{\nu}$ denote the uniform distribution on $\{-1,1\}^{2^m}$. For any $\gamma\in(0,1)$, there exists $\delta=\delta_m = c/ \sqrt{2 \log(2^{m+1}/\gamma)}$ in the definition of $f_{\nu}$ for some constant $c>0$ independent of $m$ and a subset $A_{\gamma}\subseteq \{-1,1\}^{2^m}$ with $\mathbb{P}_{\nu}(A_{\gamma})\geq 1-\gamma$, such that
\begin{itemize}
\item[(i)] $f_{\nu}\geq 0$ and $\|f_\nu\|_\infty\le M$, for all $\nu\in A_{\gamma}$,
\item[(ii)] $f_{\nu} \in \bar{\P}_s^{pq}(L)$ for all $\nu\in A_{\gamma}$, and
\item[(iii)] $|D( f_{\nu}) - D( f_0)|\geq \delta^2 \cdot 2^{-2ms}$ for all $\nu$.
\end{itemize}
\end{lemma}
\begin{proof}
Representing the orthonormal eigenvectors $w_1,\dots, w_{2^m}$ as linear combination
$$
w_j=\sum_{k=0}^{2^m-1}a_{kj}\psi_{mk}
$$
of the basis functions $\psi_{mk}$, the $2^m\times 2^m$ matrix $(a_{kj})_{kj}$ of corresponding coefficients is orthogonal and
\begin{align*}
f_{\nu}(x)=f_0(x)+ 2^{-m(s + 1)} \delta\sum_{j=1}^{2^m}\sum_{k=0}^{2^m-1} \nu_j \frac {a_{kj}}{\lambda_{j,\alpha,m}} \psi_{mk}(x).
\end{align*}
For some $\gamma$ in (0,1), define
$$
A_\gamma = \left\{ \nu: \bigg\arrowvert \sum_{j=1}^{2^m}\nu_j \frac {a_{kj} }{\lambda_{j,\alpha,m}}\bigg\arrowvert \leq {2^{m/2}}{\sqrt{2 \log\left(\frac{2^{m+1}}{\gamma}\right)}} \ \text{for all } 0\leq k\leq 2^m-1
\right\}.
$$
By the union bound and  Hoeffding's inequality it is easy to see that $\mathbb{P}_{\nu}(A_\gamma) \geq 1-\gamma$.

$(i)$ Since the basis functions $\psi_{mk}$, $k=0, \dots, 2^m-1$, have disjoint support and are bounded
in absolute value by $2^{m/2}$, it has to be shown that
\begin{equation*}
\bigg\arrowvert 2^{-m(s + 1)} \delta\sum_{j=1}^{2^m }\nu_j \frac {a_{kj} }{\lambda_{j,\alpha,m}}\bigg\arrowvert\leq 2^{-m/2}
\end{equation*}
for all $\nu\in A_\gamma$. But on this event we have
$$
\bigg\arrowvert 2^{-m(s + 1)} \delta\sum_{j=1}^{2^m} \nu_j \frac {a_{kj} }{\lambda_{j,\alpha,m}}\bigg\arrowvert 2^{m/2} \leq \delta \cdot 2^{-ms} \sqrt{2 \log(2^{m+1}/\gamma)}.
$$
But by our choice of $\delta$ and provided $c\le1$, it follows that
$$
\delta\cdot {2^{-ms}}{\sqrt{2\log\big(2^{m+1}/\gamma\big)}} \leq 1.
$$
Note that this also proves that $\|f_\nu\|_\infty\le 2\le M$.

$(ii)$ We have already seen that for $\nu\in A_\gamma$, $f_{\nu}$ is a density. Note that $\|f_\nu\|_{B_s^{pq}}\le 1 + \|f_\nu-f_0\|_{B_s^{pq}}$. Since $s<1$, the Haar wavelets can be used to characterize the Besov space and in view of Proposition~4.3.2 of \citet{Gine16} and if $c>0$ is sufficiently small, it remains to show that the wavelet coefficient norm of $f_\nu-f_0$ is bounded by $L-1>0$. That is, we need to show that
\begin{align*}
\sum_{k=0}^{2^m-1}\big\arrowvert\big\langle f_{\nu}-f_0,\psi_{mk}\big\rangle\big\arrowvert^p \leq (L-1)^p \cdot 2^{-mp(s+1/2-1/p)}.
\end{align*}
Because of
$$
\big\langle f_{\nu}-f_0,\psi_{mk}\big\rangle=2^{-m(s+1)}\delta\sum_{j=1}^{2^m} \nu_j\frac{a_{kj} }{\lambda_{j,\alpha,m}},
$$
for any $0\leq k\leq 2^m-1$, this is the case if
\begin{align*}
\sum_{k=0}^{2^m-1}\bigg\arrowvert 2^{-m(s+1)}\delta\sum_{j=1}^{2^m} \nu_j\frac{a_{kj} }{\lambda_{j,\alpha,m}}\bigg\arrowvert^p\leq (L-1)^p\cdot 2^{-mp(s+1/2-1/p)}.
\end{align*}
But if $\nu\in A_\gamma$, we have
\begin{align*}
 &\sum_{k=0}^{2^m-1}\bigg\arrowvert 2^{-m(s+1)}\delta\sum_{j=1}^{2^m}\nu_j\frac{a_{kj} }{\lambda_{j,\alpha,m}}\bigg\arrowvert^p \\
 & \leq 2^m \left( 2^{-m(s+1)}\delta 2^{m/2}  \sqrt{2 \log\left(\frac{2^{m+1}}{\gamma} \right)} \right)^p \\
 & \leq 2^{-m p (s -1/p  +1/2 ) }c^p.
\end{align*}
Thus, if $c \leq L-1$ and $\nu\in A_\gamma$, $f_\nu$ belongs to $\P_s^{pq}(L)$.

$(iii)$ By orthonormality of $f_0,w_1,...w_{2^m}$, we have that
$$
\int_0^1 f_\nu^2(x)dx= \int_0^1 f_0^2(x)dx + 2^{-2m(s+1)}\delta^2 \sum_{k=1}^{2^m} \frac 1{\lambda_{k,\alpha,m}^2} .
$$
We get
\begin{align*}
|D(f_\nu) - D(f_0)| & =  2^{-2m(s+1)}\delta^2 \sum_{k=1}^{2^m}\left( 2^{m} \cdot \mathds 1_{z_\alpha^{-1} \lambda_k  < 2^{-m/2}} + \frac{z^2_\alpha}{\lambda_k^2} \cdot \mathds 1_{z_{\alpha}^{-1} \lambda_k  \geq 2^{-m/2}} \right).
\end{align*}
Denote by $\kappa$ the number of values in the set $\mathcal{K} = \{ k : z_{\alpha}^{-1} \lambda_k  \geq 2^{-m/2}\}$.
We have
\begin{align*}
|D(f_\nu) - D(f_0)| & =  2^{-2ms}\delta^2 \left( 2^{-m}(2^{m}-\kappa) + 2^{-2m} \sum_{k \in \mathcal{K}}\frac {z^2_\alpha}{ \lambda_{k}^2}  \right)\\
& =  2^{-2ms}\delta^2 \left( 1 - 2^{-m}\kappa + 2^{-2m} z^2_\alpha \sum_{k \in \mathcal{K}}\frac {1}{ \lambda_{k}^2} \right)\\
& \geq  2^{-2ms}\delta^2 \left( 1 - 2^{-m}\kappa + 2^{-2m} z^2_\alpha \kappa^2 (\sum_{k \in \mathcal{K}}{ \lambda_{k}^2} )^{-1}\right),
\end{align*}
where we used the inequality between harmonic and arithmetic mean.
If we can prove
\begin{equation}\label{specbound}
\sum_{k =1}^{2^m}{ \lambda_{k}^2} \leq z^2_\alpha,
\end{equation}
we conclude that
\begin{equation*}\label{dist}
|D(f_\nu) - D(f_0)| \geq 2^{-2ms}\delta^2 (1 - 2^{-m} \kappa + 2^{-2m}\kappa^2) \geq \frac 34 \delta^2 2^{-2ms}.
\end{equation*}
Thus, let us finish by the proof of (\ref{specbound}).
It is easy to see that,
\begin{align*}
\sum_{k =1}^{2^m}{ \lambda_{k}^2} & = \sum_{k =1}^{2^m} \int_0^1 \int_0^1 w_k(x) w_k(y)\frac 1n \sum_{i=1}^n F_i(x,y) dx dy\\
&= \frac 1n \sum_{i=1}^n \int_{\mathcal Z_i} \sum_{k =1}^{2^m} \left(\int_0^1 \frac{q_i(z|x)}{g_{0,i}(z)} w_k(x) dx\right)^2 g_{0,i}(z) d\mu_i(z) \\
&= \frac 1n \sum_{i=1}^n \int_{\mathcal Z_i} \sum_{k =1}^{2^m} \left(\int_0^1 \left( \frac{q_i(z|x)}{g_{0,i}(z)} - e^{-2\alpha} \right) w_k(x) dx\right)^2 g_{0,i}(z) d\mu_i(z)
\end{align*}
as $\int_0^1 w_k(x) dx = 0$. By our choice of densities $q_i$, we have $0 \leq f_{z,i}(x):= \frac{q_i(z|x)}{g_{0,i}(z)} - e^{-2\alpha} \leq z_\alpha$. Since the $w_1, \dots, w_{2^m}$ are orthonormal and $\|f_{z,i}\|_{L_2} \le z_\alpha$, it follows that
\begin{align*}
\sum_{k =1}^{2^m} \left(\int_0^1 \left( \frac{q_i(z|x)}{g_{0,i}(z)} - e^{-2\alpha} \right) w_k(x) dx\right)^2 
&=
\sum_{k =1}^{2^m} \langle f_{z,i}, w_k\rangle^2 
=
\left\|\sum_{k =1}^{2^m} \langle f_{z,i}, w_k\rangle w_k\right\|_{L_2}^2\\
&\le
\|f_{z,i}\|_{L_2}^2 \le z_\alpha^2.
\end{align*}
Moreover, $\int_{\mathcal Z_i} g_{0,i}(z) d\mu_i(z) = \int_0^1\int_{\mathcal Z_i} q_i(z|x) d\mu_i(z) dx = 1$, and we arrive at $\sum_{k =1}^{2^m}{ \lambda_{k}^2} \le z_\alpha^2$, as desired.
\end{proof}

\begin{lemma}\label{bias}
If $f$ belongs to the Besov ball $B_s^{pq}(R) := \{g\in B_s^{pq}([0,1]) : \|g\|_{s}^{pq}\le R\}$ and $s' := s - \left(\frac 1p - \frac 12 \right)_+<1$, then the Haar coefficients of $f$ satisfy
$$
\|\beta_{j \cdot}\|_2 \leq  2^{-j s'}\varepsilon_j, \quad j\ge1,
$$
and $\|\eps\|_q\le C_0R$, for a constant $C_0$ that does not depend on $f$.
\end{lemma}
\begin{proof}[Proof of Lemma~\ref{bias}]
We shall consider first the case $p\geq 2$. In that case, by the H\"older inequality with
$\frac 12 = \frac 1p + \left( \frac 12 - \frac 1p\right)$, we get
$$
\|\beta_{j \cdot}\|_2 \leq 2^{j \left(\frac 12 - \frac 1p \right)} \|\beta_{j \cdot}\|_p.
$$
Since $s=s'<1$, by (2.3), we have that $\|\beta_{j \cdot}\|_p = 2^{-j (s + \frac 12 - \frac 1p)} \varepsilon_j$, for a sequence $\{\varepsilon_j\} \in \ell_q$ with $\|\eps\|_q \le \|f\|_{B_s^{pq}}\le C_0R$, giving
$$
\|\beta_{j \cdot}\|_2 \leq 2^{-js} \varepsilon_j.
$$
In the case $1\le p<2$, we use the continuous embedding
$$B_s^{pq} \subseteq B_{s- \frac 1p + \frac 12}^{2q} = B_{s'}^{2q},$$
which follows from the characterization of the Besov space in terms of wavelet coefficients. Again, by (2.3), we get
$$
\|\beta_{j \cdot}\|_2 = 2^{-js'} \varepsilon_j,
$$
for a sequence $\{\varepsilon_j\} \in \ell_q$ as desired.
\end{proof}

\begin{lemma}\label{lemma:bias}
Fix $J\ge1$ and, for $j\ge J$ and $k=0,\dots, 2^j-1$, let $\beta_{jk}$ be the Haar coefficients of $f\in \P_s^{pq}(L)$. Then, for $s' = s - \left( \frac{1}{p} - \frac12\right)_+$, we have
$$
\sum_{j \ge J} \|\beta_{j\cdot}\|_2^2\; \lesssim\; 2^{-2Js'} \mathds 1_{[0,1)}(s') + 2^{-J\frac53}  \mathds 1_{[1,\infty)}(s').
$$
\end{lemma}
\begin{proof}
We consider successively the cases where $1\le p<2$ and where $p \geq 2$.
If $1\le p<2$, the continuous embedding $$B_s^{pq} \subseteq B_{s - \frac 1p + \frac 12}^{2q}$$ holds in view of the definition of the wavelet Besov norm.
Now, in the case $s' = s - \frac 1p + \frac 12 < 1$ we get, by Lemma~\ref{bias}, that
$$
\sum_{j \ge J} \|\beta_{j\cdot}\|_2^2 \le \sum_{j \ge J} 2^{-2js'} \varepsilon_j^2 \leq C\cdot 2^{-2J s'},
$$
for some $C>0$ that does not depend on $f$, by using that $\|\eps\|_\infty\le \|\eps\|_q$.
In case $s' \geq 1 > 5/6$, we use the further embedding $B_s^{pq} \subseteq B_{s'}^{2q}\subseteq B_{5/6}^{2q}$. Thus, from Lemma~\ref{bias}, we get that $\|\beta_{j\cdot}\|_2\le 2^{-j5/6} C$, which implies
$$
\sum_{j \ge J} \|\beta_{j\cdot}\|_2^2 \le \sum_{j \ge J} 2^{-j \frac 53} C^2 \leq C^2\cdot 2^{-J \frac{5}3 }.
$$
If $p \geq 2$, and if $s' = s < 1$, then we apply directly Lemma~\ref{bias} to get
$$
\sum_{j \ge J} \|\beta_{j\cdot}\|_2^2 \leq C 2^{-2Js}, \quad \text{for some constant }C>0.
$$
If $s' = s \geq 1$, we use the embedding $B_s^{pq} \subseteq B_{5/6}^{pq}$ and conclude, again from Lemma~\ref{bias}, that
$$
\sum_{j \ge J} \|\beta_{j\cdot}\|_2^2 = \sum_{j\ge J} 2^{-j \frac 53} \varepsilon_j^2 \leq C\cdot 2^{-J \frac{5}3 }, \quad \text{for some constant }C>0.
$$
\end{proof}


\section{Proofs of Section~4 (sequentially interactive protocols}
\label{sec:App:SEQINTER}

\subsection{Concentration of the sanitized density estimator}

\begin{proposition}\label{exponineq}
Fix $M,L>0$, $1\le p,q\le \infty$ and $s>0$, and let $\bar{\P}_s^{pq}(L, M)$ be defined as in Theorem~4.1. Then there exist constants $c_1,\,c_2>0$, such that for any $n\ge1$, $\alpha\le 1$ and $J\ge 2$, the estimator $\hat f_J^{(1)}$ in (4.1) satisfies
\begin{align*}
&\sup_{x\in[0,1]}Q_f^{(SI)}\left( \left|\hat f_J^{(1)}(x) - \E_{Q_f^{(SI)}}\left[\hat f_J^{(1)}(x)\right]\right|\geq \left[c_1\frac{J^{a} 2^J}{\alpha \sqrt{n}}  \sqrt{u}\right] \vee \left[ c_2 \frac{J^a 2^J}{n \alpha } u\right] \right) \\
&\quad\quad\leq 4 e^{-u/2},
\end{align*}
for all $u>0$ and all $f\in\bar{\P}_s^{pq}(L, M)$.
\end{proposition}

\begin{proof}
A centered random variable $Y$ is said to be sub-exponential with parameters $(\nu^2,b)$, $\nu^2 >0,\, b>0$, denoted by $SubExp(\nu^2,b)$, if
$$
\E[\exp(t Y)] \leq \exp\left(\frac{\nu^2 t^2}2\right),\text{ for all } |t|<\frac 1b.
$$
The Bernstein inequality states that if $Y$ is $SubExp(\nu^2,b)$, then
$$
\Pr(Y\geq t) \leq \exp\left(-\frac 12 \left(\frac{t^2}{\nu^2} \wedge \frac tb \right) \right), \text{ for all }  t>0,
$$
or, equivalently,
$$
\Pr(Y\geq (\nu \sqrt{u}) \vee (bu)) \leq \exp\left(-\frac u2  \right), \text{ for all }  u>0.
$$
We apply this to the two summands in the decomposition
\begin{align*}
\hat f_J^{(1)}(x) - \E_f\left[\hat f_J^{(1)}(x)\right] &=
\frac 1n \sum_{i=1}^n \sum_{j=-1}^{J-1} \sum_{k=0}^{(1\lor2^j)-1} [ \psi_{jk}(X_i) - \beta_{jk}] \psi_{jk}(x)  \\
&\quad+ \frac 1n \sum_{i=1}^n \sum_{j=-1}^{J-1} \sum_{k=0}^{(1\lor2^j)-1}\sigma_j \frac{\sigma}\alpha W_{ijk} \psi_{jk}(x) \\
&=: T_1(x) + T_2(x).
\end{align*}
Let us start with $T_2$. A Laplace distribution has the following Laplace transform
$$
\E \left[\exp \left( \frac tn \sigma_j \frac{\sigma}\alpha \cdot W_{ijk}\psi_{jk}(x) \right) \right]
 = \frac 1{1 - (\frac tn \sigma_j \frac{\sigma}\alpha \psi_{jk}(x))^2},
 $$
for any $|t| < \frac {n \alpha}{ \sigma \sigma_j \|\psi_{jk}\|_\infty}$, which is further bounded from above by
$$
\exp \left(  2t^2 \sigma_j^2 \left(\frac{ \sigma}{n\alpha}\psi_{jk}(x)\right)^2 \right), \quad  \text{ for any } |t| < \frac {n \alpha}{ 2 \sigma \sigma_j (1\lor 2^{j/2})},
$$
in view of $(1-v)^{-1} \leq e^{2v}$ for all $0\le v\le\frac 12$.
Thus, by independence of the random variables $\{W_{ijk}\}_{i,j,k}$ and since $\sum_k \psi_{jk}(x)^2 \leq 2^j$ for all $x\in[0,1]$ and all $j\ge -1$, we obtain
\begin{align*}
&\E \left[\exp \left(\frac tn \sum_{i=1}^n \sum_{j=-1}^{J-1} \sum_{k=0}^{(1\lor2^j)-1}\sigma_j \frac{\sigma}\alpha W_{ijk} \psi_{jk}(x) \right) \right]\\
&\quad \leq  \left(\prod_{j=-1}^{J-1} \prod_{k=0}^{(1\lor2^j)-1}\exp\left(2t^2 \sigma_j^2 \left(\frac{\sigma}{n\alpha} \psi_{jk}(x)\right)^2\right)\right)^n \\
&\quad\leq \exp\left(2t^2 n \frac{\sigma^2}{n^2\alpha^2} \sum_{j=-1}^{J-1} 2^{j} \sigma_j^2 \right) \\
&\quad=  \exp\left(2t^2 \frac{\sigma^2}{n\alpha^2} \sum_{j=-1}^{J-1} 2^{2j} (1 \vee j)^{2a}) \right)\\
&\quad\leq  \exp\left(\frac12 \left[\frac{2 \sigma J^{a} 2^{J}}{\sqrt{n}\alpha}\right]^2  t^2 \right),
\end{align*}
for any $t$ such that
$$|t| < \min_{-1\le j\le J} \frac{n\alpha}{2 \sigma \sigma_j (1\lor2^{j/2}) } = \frac{n\alpha}{2\sigma J^a 2^J}.
 $$
By the Bernstein inequality we get for all $u>0$,
\begin{align}\label{term2}
\Pr\left( \left|T_2(x)\right|
\geq \left(2\sigma \frac{J^{a}2^{J}}{\alpha \sqrt{n}} \sqrt{u}\right) \vee \left(2\sigma \frac{J^a 2^J}{n\alpha}u\right) \right) \leq 2e^{-u/2}.
\end{align}
For $T_1$, we use the fact that the Haar wavelets generate a multiresolution analysis of $L_2[0,1]$ with Projection operator $P_J$ onto $V_J= \s\{\phi_{Jm} : m=0,\dots, 2^J-1\}$ to write
\begin{align*}
T_1(x) &= \hat{f}^{(1)}_J(x) - [P_J f](x) = \frac{1}{n}\sum_{i=1}^n \kappa_J(x,X_i) - [P_J f](x)\\
&= \frac 1n \sum_{i=1}^n \sum_{k=0}^{2^J-1} [ \phi_{Jk}(X_i) - \alpha_{Jk}] \phi_{Jk}(x).
\end{align*}
We have that
\begin{align*}
\left|\sum_{k} [ \phi_{Jk}(X_i) - \alpha_{Jk}] \phi_{Jk}(x)\right|
&\leq
\sum_k [\phi_{Jk}(X_i) + \alpha_{Jk}]\phi_{Jk}(x) \\
&\leq \sum_k \left(2^{J/2} + \int_0^1 2^{J/2} \phi(2^Jx-k) f(x)dx\right)\phi_{Jk}(x)\\
&\le
2 \cdot2^{J/2} \sum_k 2^{J/2}\phi(2^Jx-k) = 2^{J+1}.
\end{align*}
We write for $|t|\le n 2^{-J-1}$ and $|u| \le \frac{|t|}{n} 2^{J+1} \le 1$, that $e^u \leq 1+u+u^2e$. Thus, we get
\begin{align*}
&\E\left[\exp\left(\frac tn \sum_{i=1}^n \sum_{k} [ \phi_{Jk}(X_i) - \alpha_{Jk}] \phi_{Jk}(x)\right) \right]\\
  &\quad= \left(\E \exp \left(\frac tn \sum_{k} [ \phi_{Jk}(X_1) - \alpha_{Jk}] \phi_{Jk}(x)\right)\right)^n \\
  &\quad\leq \left(  1 + e\frac{t^2}{n^2} \E \left[\left( \sum_{k} [ \phi_{Jk}(X_1) - \alpha_{Jk}] \phi_{Jk}(x)\right)^2\right] \right)^n\\
   &\quad\leq \left( 1 + e\frac{t^2}{n^2} \Var\left( \sum_{k} \phi_{Jk}(X_1) \phi_{Jk}(x)\right)\right)^n.
\end{align*}
The Haar basis is such that for $F_J = 2^J\mathds 1_{[0,2^{-J}]}$, and for all $x,y \in [0,1]$,
$$
\sum_k \phi_{Jk}(x) \phi_{Jk}(y)
= \sum_k 2^J \mathds 1_{[k2^{-J}, (k+1)2^{-J}]^2}(x,y)
\leq F_J(|x-y|),
$$
and it follows that
\begin{align*}
&\Var\left( \sum_{k} \phi_{Jk}(X_1) \phi_{Jk}(x)\right) \leq
\E\left[\left( \sum_{k} \phi_{Jk}(X_1) \phi_{Jk}(x)\right)^2\right] \\
&\quad\le\E\left[ F_J(|X_1-x|)^2 \right]
   \leq  2^{J+1} \|f\|_\infty.
\end{align*}
Thus, for all $t$ such that  $|t| \le \frac n{2^{J+1}}$, we obtain
\begin{align*}
\E\left[\exp\left(\frac tn \sum_{i=1}^n \sum_{k} [ \phi_{Jk}(X_i) - \alpha_{Jk}] \phi_{Jk}(x)\right) \right]
&\leq \left(1 + e\frac{t^2}{n^2} 2^{J+1} \|f\|_\infty \right)^n \\
&\leq \exp\left(\frac{t^2}{2} \cdot e\frac{2^{J+2}}{n} \|f\|_\infty\right)
\end{align*}
Now, by the Bernstein inequality, we get for all $u>0$,
\begin{align}\label{term1}
\Pr\left( |T_1(x)| \geq \left[\|f\|_\infty^{1/2} \frac{2^{(J+2)/2}\sqrt{e}}{ \sqrt{n}} \sqrt{u}\right] \vee \left[ \frac{ 2^J}{n}u\right] \right)
\leq 2e^{-u/2}.
\end{align}
Putting together \eqref{term1} and \eqref{term2}, we get the result.
\end{proof}

\subsection{Proof of Theorem~4.1}

For $f\in\bar\P_s^{pq}(L)$, we have
\begin{align*}
\E_{Q_f^{(SI)}} \tilde D_n &= \E\left[ \E\left[ \tilde D_n \Big| Z^{(1)}\right]\right]
=
\E\left[ \E \left[Z_1^{(2)}| X_1^{(2)}, Z^{(1)}\right]\right]\\
&= \E\left[ \int_0^1 \Pi_\tau[\hat f_J^{(1)} (x)]f(x) dx \right].
\end{align*}
Note that $\Pi_\tau(v) = (\tau \wedge v)\vee(-\tau) = v - (v-\tau)_+ + (-v -\tau)_+$. Thus, the bias can be written as
\begin{align}\label{bias1}
D - \E\left[ \tilde D_n\right] &= D - \E\left[\int_0^1 \hat f_J^{(1)} (x) f(x) dx \right]+ \E \left[\int_0^1 \left(\hat f_J^{(1)} (x) - \tau\right)_+ f(x) dx\right] \notag\\
&\quad - \E \left[\int_0^1 \left( - \hat f_J^{(1)} (x) - \tau\right)_+ f(x) dx\right].
\end{align}
First, compute $\E\left[\int_0^1 \hat f_J^{(1)} (x) f(x) dx\right] = \int_0^1 \sum_{j=-1}^{J-1}\sum_{k=0}^{(1\lor2^J)-1} \beta_{jk} \psi_{jk}(x) f(x)dx = \sum_{j=-1}^{J-1} \|\beta_{j\cdot}\|_2^2$. Thus, by Lemma~\ref{lemma:bias}, we get
\begin{align}\label{bias1_1}
\left|D - \E\left[\int \hat f_J^{(1)} (x) f(x) dx\right]\right| = \sum_{j\ge J} \|\beta_{j\cdot}\|_2^2 \lesssim
\begin{cases}
2^{-2J s'}, \quad&  0 \leq s' < 1,\\
2^{-J \frac 53}, & s' \ge 1.
\end{cases}
\end{align}
Next, we treat the second term on the right-hand-side of (\ref{bias1}), but we skip the third because a similar bound can be obtained by analogous arguments. We write
\begin{align*}
&\E\left[ \int_0^1 (\hat f_J^{(1)} (x) - \tau)_+ f(x) dx\right] \\
&\quad= \E\left[ \int_0^1 \left(\hat f_J^{(1)} (x) - \E\left[ \hat f_J^{(1)} (x)\right]  + \E\left[ \hat f_J^{(1)} (x)\right] - \tau \right)_+ f(x) dx\right]
\end{align*}
and we note that $\tau \geq 2 M$, where $M$ is the uniform bound on $ \|f\|_\infty$. Then, using the fact that the Haar basis generates a multiresolution analysis of $L_2([0,1])$ with projection operator $P_J$ onto $V_J= \s\{\phi_{Jm} : m=0,\dots, 2^J-1\}$,
\begin{eqnarray*}
\tau - \E\left[ \hat f_J^{(1)} (x)\right] &=& \tau - P_J f \geq  \tau-M  \geq \frac \tau2,
\end{eqnarray*}
because $P_J f = \sum_{k=0}^{2^J-1} \alpha_{Jk} \phi_{Jk}(x)\ge0$, $\alpha_{Jk} = \int_0^1 f(x) 2^{J/2} \phi(2^J x - k) dx \leq M 2^{-J/2}$ and,
\begin{equation}\label{eq:ProjBound}
\sup_{x\in[0,1]} P_J f(x) \leq \sup_{x\in[0,1]} M \sum_{k=0}^{2^J-1} \phi(2^J x - k) \leq M,
\end{equation}
as the functions $\phi(2^J \cdot - k) $ have disjoint supports for different values of $k$. For $x\in[0,1]$, let us denote $Y_x = \hat f_J^{(1)} (x) - \E\left[\hat f_J^{(1)} (x)\right]$ and, using the previous consideration, write
\begin{align}
\E\left[ \int_0^1 (\hat f_J^{(1)} (x) - \tau)_+f(x)\,dx\right] &\leq \E \left[\int_0^1\left(Y_x - \frac \tau2\right)_+f(x)\,dx \right] \label{eq:Yx}\\
&= \int_0^1\int _0^\infty \Pr_f\left(Y_x - \frac \tau2 \geq w\right)dw\, f(x)dx\notag\\
&= \int_0^1\int_{\tau/2}^\infty \Pr_f(Y_x \geq u) \, du\, f(x)dx.\notag
\end{align}
We apply Proposition~\ref{exponineq} to get, for some constant $c>0$ not depending on $n$, $f$, $\alpha$ and $u$, and for all $u>0$,
\begin{align}\label{eq:YxConc}
\Pr(Y_x \geq u) \leq 4 \exp\left(- \frac c2 \left[ \left(\frac{n \alpha^2}{J^{2a} 2^{2J}}u^2\right) \wedge \left(\frac{n \alpha}{J^a 2^J} u\right) \right] \right),
\end{align}
which implies
\begin{align*}
&\frac14\int_{\tau/2}^\infty \Pr(Y_x \geq u) du \notag\\
&\quad\le \int_{\tau/2}^{J^a 2^J/\alpha} \exp \left(- \frac c2 \frac{n \alpha^2}{J^{2a} 2^{2J}} u^2\right) du
+ \int_{J^a 2^J/\alpha}^\infty \exp\left(- \frac c2 \frac{n \alpha}{J^a 2^J} u\right) du \nonumber
\end{align*}
Next, we apply Lemma~\ref{lemma:Integrals} with $a_1=0$, $A_1=\frac{c}{2}\frac{n\alpha^2}{J^{2a} 2^{2J}}$, $r_1=2$ and $v_1=\frac\tau2$, and with $a_2=0$, $A_2=\frac{c}{2}\frac{n\alpha}{J^{a} 2^{J}}$, $r_2=1$ and $v_2= J^a 2^J/\alpha$. Note that in the former case,
\begin{align*}
A_1r_1v_1^{r_1} &= \frac c2 \frac{n\alpha^2}{J^{2a}2^{2J}}\frac{\tau^2}{4}
\ge
\frac{c(KM)^2}{2} \frac{n\alpha^2}{J^{2a}2^{2J}}J^{2a+1} 2^{J(1-2(s'\land\frac12))}\\
&=
\frac{c(KM)^2}{2} n\alpha^2 J 2^{-J(1 + 2(s'\land\frac12))}\\
&=
\frac{c(KM)^2}{2} n\alpha^2 \frac{1}{2(s'\land 1)+1}\frac{\log(n\alpha^2)}{\log 2} (n\alpha^2)^{\frac{-1 - 2(s'\land\frac12)}{2(s'\land1)+1}} \\
&\ge
\frac{c(KM)^2}{6\log 2} (n\alpha^2)^{\frac{-1 - 2(s'\land\frac12)}{2(s'\land1)+1}+1} \log(n\alpha^2)\\
&=
\frac{c(KM)^2}{6\log 2} (n\alpha^2)^{\frac{ 2(s'\land1 - s'\land\frac12)}{2(s'\land1)+1}}  \log(n\alpha^2)\\
&\ge
\frac{c(KM)^2}{6\log 2} \log(n\alpha^2) \ge 3,
\end{align*}
for $n\alpha^2\ge \exp(\frac{12\log2}{c(KM)^2})$. In the latter case $A_2r_2v_2^{r_2} = \frac c2 n \ge 3$, if $n\alpha^2\ge6/c$, because $\alpha\le 1$. Therefore, Lemma~\ref{lemma:Integrals} yields
\begin{align*}
\int_{\tau/2}^\infty \exp \left(- \frac c2 \frac{n \alpha^2}{J^{2a} 2^{2J}} u^2\right) du
&\le
(A_1v_1)^{-1} \exp\left(- A_1v_1^{r_1}\right) \\
&=
\frac{2v_1}{A_1r_1v_1^{r_1}} \exp\left(-\frac12 A_1r_1v_1^{r_1}\right)  \\
&\le \frac\tau2 (n\alpha^2)^{-\frac{c(KM)^2}{12\log 2}},
\end{align*}
but for sufficiently large $K$ this will always be smaller than the final rate. Similarly,
\begin{align*}
 \int_{J^a 2^J/\alpha}^\infty \exp\left(- \frac c2 \frac{n \alpha}{J^a 2^J} u\right) du
&\le
2 A_2^{-1} \exp\left(- A_2v_2^{r_2}\right) \\
&=
\frac{2v_2}{A_2r_2v_2^{r_2}} \exp\left(- A_2r_2v_2^{r_2}\right)  \\
&= \frac{J^a 2^J}{\alpha} \frac{4}{cn}e^{-n\frac{c}{2}}\\
&\le J^a 2^J \frac{4}{cn\alpha^2}e^{-n\alpha^2\frac{c}{2}},
\end{align*}
which is much smaller than the final rate. Thus, the only relevant contribution from the bias is the one of \eqref{bias1_1}.

Regarding the variance, we write
$$
\Var[\tilde D_n] = \E \left[ \Var (\tilde D_n |X^{(1)}) \right] + \Var\left[ \E (\tilde D_n |X^{(1)} )\right]
$$
Now,
\begin{align*}
\Var(\tilde D_n |X^{(1)} ) &\leq \E (\tilde D_n ^2 | X^{(1)} ) = \frac 1n \E( (Z_1^{(2)})^2|X^{(1)} ) \\
&= \frac{\tau^2}{n}\left(\frac{e^\alpha+1}{e^\alpha-1}\right)^2 \leq (e+1)^2 4(KM)^2 \frac{J^{2 a +1} 2^{J(1-2(s'\land\frac12))}}{n \alpha^2},
\end{align*}
since $\alpha\le 1$ and $\alpha\le e^\alpha-1$. Moreover,
\begin{align*}
\Var &\left[ \E (\tilde D_n | X^{(1)}) \right] = \Var\left[ \int_0^1 \Pi_\tau\left[\hat f_J^{(1)} (x) \right] f(x) dx \right] \\
&\leq 3 \Var\left[ \int_0^1 (\hat f_J^{(1)} (x) - \tau)_+ f(x) dx\right] + 3 \Var \left[\int_0^1 \hat f_J^{(1)}(x) f(x) dx)\right] \\
&\quad+ 3 \Var\left[ \int_0^1 (-\hat f_J^{(1)} (x) - \tau)_+ f(x) dx\right].
\end{align*}
Again, we only explicitly treat the first two terms, as the first and the third terms are handled analogously. We have
\begin{align*}
\Var&\left[ \int_0^1 (\hat f_J^{(1)} (x) - \tau)_+ f(x) dx\right]
\leq
\E\left[ \left(\int_0^1 (\hat f_J^{(1)} (x) - \tau)_+ f(x) dx \right)^2\right] \\
&\leq \int_0^1 \E\left[ \left(\hat f_J^{(1)} (x) - \tau\right)_+^2\right] f(x) dx \leq \int_0^1 \E\left[\left(Y_x - \frac \tau2\right)_+^2\right] f(x) dx,
\end{align*}
as in \eqref{eq:Yx}. Thus, we may use \eqref{eq:YxConc} and Lemma~\ref{lemma:Integrals} with $a_1=1$ and $A_1$, $r_1$ and $v_1$ as above, satisfying $A_1r_1v_1^{r_1}\ge 3$ and with $a_2=1$ and $A_2$, $r_2$ and $v_2$ as above, satisfying $A_2r_2v_2^{r_2}\ge3$, in order to get
\begin{eqnarray}
\E\left[ \left(Y_x - \frac\tau2\right)_+^2\right]
& = & 2\int_0^\infty t \Pr_f\left( Y_x - \frac\tau2 \geq t\right)dt \nonumber \\
& \leq & 2 \int_{\tau/2}^\infty s \Pr_f (Y_x \geq s) ds \nonumber \\
& \leq & 8 \int_{\tau/2}^{J^a 2^J/\alpha} s \exp\left(- \frac c2 \frac{n \alpha^2}{J^{2a} 2^{2J}} s^2\right) ds \nonumber\\
&  & + 8 \int_{J^a 2^J/\alpha}^\infty s \exp\left(- \frac c2 \frac{n \alpha}{ J^a 2^J} s\right) ds \nonumber \\
& \leq & 8 \frac{3}{A_1r_1} \exp\left( -A_1v_1^{r_1}\right)
+
8 \frac{3v_2}{A_2r_2} \exp\left( -A_2v_2^{r_2}\right)
\nonumber\\
& \leq & 12 \frac 2c \frac{J^{2a} 2^{2J}}{n \alpha^2} \exp\left(- \frac c2 \frac{n \alpha^2}{J^{2a} 2^{2J}} \frac{\tau^2}4 \right)
+
24 \frac 2c \frac{ J^{2a} 2^{2J} }{n \alpha^2} \exp\left( - n \frac c2\right). \label{var1}
\end{eqnarray}
Again, both terms in the last line of the previous display are smaller than the final rate of our estimator, provided that the constant $K$ is large enough.
Finally, we consider the variance of the integrated estimator in (4.1), that is,
\begin{align}
&\Var\left[\int_0^1 \hat f_J^{(1)}(x) f(x) dx \right]
=
\Var\left[\sum_{j=-1}^{J-1} \sum_{k=0}^{(1\lor 2^j)-1}\hat{\beta}_{jk}\beta_{jk} \right]\notag\\
&\quad\leq
\Var\left[\sum_{j=-1}^{J-1}\sum_{k=0}^{(1\lor2^j)-1} \frac 1n \sum_{i=1}^n \psi_{jk}(X_i)\beta_{jk}\right] \notag\\
&\quad\quad+ \Var\left[\frac 1n \sum_{i=1}^n \sum_{j=-1}^{J-1} \sum_{k=0}^{(1\lor2^j)-1} \sigma_j \frac{\sigma}\alpha  W_{ijk} \beta_{jk} \right]\notag\\
&\quad=  \frac 1n \Var\left[ \sum_{j=0}^{J-1}\sum_{k=0}^{2^j-1} \psi_{jk}(X_1) \beta_{jk}\right]
+
\frac 1n \frac{\sigma^2}{\alpha^2} \sum_{j=-1}^{J-1} \sum_{k=0}^{(1\lor2^j)-1} \Var(W_{1jk}) \beta_{jk}^2 \sigma_j^2\notag\\
&\quad\leq  \frac 1n \int_0^1 \left(\sum_{j=0}^{J-1}\sum_{k=0}^{2^j-1} \psi_{jk}(x) \beta_{jk}\right)^2 f(x) dx
+
\frac 2n \frac{\sigma^2}{\alpha^2} \left(1+ \sum_{j=0}^{J-1} j^{2a} 2^j \|\beta_{j \cdot}\|_2^2\right)\notag\\
&\quad\leq \frac 1n M^2 + C\frac {\sigma^2}{n \alpha^2} (1 \vee J^{2a+1} 2^{J(1-2s')}). \label{eq:var2}
\end{align}
where the last inequality follows from \eqref{eq:ProjBound} and Lemma~\ref{lemma:bias}, the latter of which implies
$$
\|\beta_{j \cdot}\|_2^2 \;\lesssim\; 2^{-2 j s'} \mathds 1_{[0,1)}(s') + 2^{-j \cdot \frac 53} \mathds 1_{[1,\infty)}(s'),
$$
and, in turn,
$$
\sum_{j=0}^{J-1} j^{2a} 2^j \|\beta_{j \cdot}\|_2^2 \;\lesssim\; 1 \vee J^{2a+1} 2^{J(1-2s')}.
$$
By putting together \eqref{bias1_1} and \eqref{eq:var2} we get the stated result. \hfill \qed


\subsection{Proof of Theorem~4.2}

Fix a $Q\in\mathcal Q_\alpha^{(SI)}$.
To construct appropriate hypotheses, consider an $S$-regular orthonormal Daubechies wavelet basis
$$
\mathcal W=\left\{\phi_k = \phi(\cdot-r),\psi_{lk} = 2^{l/2}\psi\left(2^l(\cdot) - k\right): k\in\Z, l\in\N_0\right\},
$$
of $L^2(\R)$ with $S>s$ \citep[cf.][Theorem~4.2.10]{Gine16}. This means, in particular, that $\supp \phi\subseteq [0,2S-1]$, $\supp \psi \subseteq [-S,S]$, $\int_\R \psi_{lk}(x)\,dx=0$, $\left\|\sum_{k\in\Z} |\psi(\cdot-k)| \right\|_\infty<\infty$ and $\|\psi_{lk}\|_{1} = 2^{-l/2}\|\psi\|_{1}$. Since $s<1$, it suffices to take $S=1$ and thus, for every fixed $l\ge 1$ and $k=1,\dots, 2^l$, the $\psi_{lk}$ are supported on $[0,1]$. Clearly, $\psi_{lk}\in B_s^{pq}(\R)$, and therefore also $\psi_{lk}\in B_s^{pq}([0,1])$.

Now, for $m\ge 1$, $\delta>0$ and $\nu\in\mathcal V_m:=\{-1,1\}^{2^m}$, define $f_0(x) := 1$ and
$$
f_\nu(x) := 1 + 2^{-m(s+\frac{1}{2})} \delta \sum_{k=1}^{2^{m}} \nu_{k} \psi_{mk}(x), \quad x\in[0,1].
$$
Since
$2^{-m/2}\left\|\sum_{k=1}^{2^m} |\psi_{mk}| \right\|_\infty = \left\|\sum_{k=1}^{2^m} |\psi(\cdot-k)| \right\|_\infty \le \left\|\sum_{k\in\Z} |\psi(\cdot-k)| \right\|_\infty =: D_0 <\infty $, we see that $f_\nu$ is lower bounded by $\frac12$ on $[0,1]$ provided that $2^{-ms}\delta D_0\le \frac12$, which holds (for all $m\ge 0$) if $\delta\le (2D_0)^{-1}$. This shows that $f_\nu$ is a density on $[0,1]$ with $\|f_\nu\|_\infty\le 2$. Moreover, by Proposition~4.3.2 of \citet{Gine16}, we have
\begin{align*}
\|f_\nu\|_{B_{s}^{pq}([0,1])}
&\le
1 + 2^{-m(s+\frac{1}{2})} \delta \left\|\sum_{k=1}^{2^{m}} \nu_{k} \psi_{mk}\right\|_{B_{s}^{pq}([0,1])} \\
&\le
1 + 2^{-m(s+\frac{1}{2})} \delta \left\|\sum_{k=1}^{2^{m}} \nu_{k} \psi_{mk}\right\|_{B_{s}^{pq}(\R)} \\
&\le
1 + c 2^{-m(s+\frac{1}{2})} \delta \left\| \sum_{k=1}^{2^{m}} \nu_{k} \psi_{mk}\right\|_{B_{pq}^{s,W}(\R)}
= 1 + c\delta \le L,
\end{align*}
provided that $\delta>0$ is sufficiently small. Hence, $f_\nu\in\bar{\P}_s^{pq}(L,M)$ for every $\nu\in\mathcal V_m$. Moreover, by construction, $D(f_\nu) = D(f_0) + 2^{-2ms}\delta^2$. If $\Pr_0$ and $\Pr_\nu$ are the probability measures corresponding to $f_0$ and $f_\nu$, respectively, we write $Q_\nu^n := Q\Pr_\nu^n$, $Q_0^n := Q  \Pr_0^n$ and
$
\bar{Q}^n := 2^{-2^m} \sum_{\nu\in\mathcal V_m}Q  \Pr_\nu^n.
$

Set $\Delta := 2^{-2ms}\delta^2/2$. For a measurable function $\hat{D}_n:\mathcal Z\to\R$ and $f \in \bar{\mathcal{P}}_s^{pq}(L,M)$, define
$S_f := \{ z\in\mathcal Z : | \hat{D}_n(z) - D(f)| \ge \Delta\}$, $S_{f,1} := \{ z\in\mathcal Z : \hat{D}_n(z) \ge 1+ \Delta, D(f)\le 1\}$ and $S_{f,2}  := \{ z\in\mathcal Z : \hat{D}_n(z) < 1+ \Delta, D(f)\ge 1+2\Delta\}$, which obey the inclusions $S_{f,j}\subseteq S_f$, for $j=1,2$. Now
\begin{align}
&\sup_{f \in \bar{\mathcal{P}}_s^{pq}(L,M)} Q \Pr_f^{n}( S_f ) 
\ge \sup_{f \in \bar{\mathcal{P}}_s^{pq}(L,M)} \max\left\{ Q \Pr_f^{ n}(S_{f,1}), Q  \Pr_f^{n}(S_{f,2})\right\}\notag\\
&\quad\ge \max\left\{ \max_{\nu\in\mathcal V_m\cup\{0\}}Q \Pr_{f_\nu}^{ n}(S_{f_\nu,1}), \max_{\nu\in\mathcal V_m\cup\{0\}} Q  \Pr_{f_\nu}^{n}(S_{f_\nu,2})\right\}\notag\\
&\quad=
\max\left\{ \ Q \Pr_{0}^{n}\left( \hat{D}_n \ge 1+\Delta\right), 
\max_{\nu\in\mathcal V_m} Q \Pr_{\nu}^{n}\left( \hat{D}_n < 1+\Delta\right)
\right\}\notag\\
&\quad\ge
\frac{1}{2}\left\{
Q \Pr_{0}^{n}\left( \hat{D}_n \ge 1+\Delta\right) 
+
2^{-2^m} \sum_{\nu\in\mathcal V_m} Q \Pr_{\nu}^{ n}\left( \hat{D}_n < 1+\Delta\right) \right\}\notag\\
&\quad\ge
\frac{1}{2} 
\inf_{\text{tests }\phi} \left\{ 
\E_{Q_0^{n}}[\phi]
+
\E_{\bar{Q}^{n}}[1-\phi] \right\} 
=
\frac{1}{2} 
 \left\{
1 - \sup_{\text{tests }\phi}
\E_{Q_0^{n}}[\phi]
-
\E_{\bar{Q}^{n}}[\phi] \right\} \notag\\
&\quad=
\frac12\left(1 - \dtv\left(Q_0^n, \bar{Q}^n\right) \right)\ge \frac12\left(1-\sqrt{D_{KL}\left(Q_0^n, \bar{Q}^n\right)/2}\right), \label{eq:quadTV}
\end{align}
where we have used Pinsker's inequality in the last step.

We abbreviate the regular conditional distributions of $Z_i$ given $Z_1,\dots, Z_{i-1}$ when $X_i$ comes from $\Pr_0$ or $\Pr_\nu$, by $\mathcal L_{Z_i|z_{1:(i-1)}}^{(0)}(dz_i) := \int_{[0,1]} Q_i(dz_i|x_i,z_{1:(i-1)})d\Pr_0(x_i)$ and $\mathcal L_{Z_i|z_{1:(i-1)}}^{(\nu)}(dz_i) := \int_{[0,1]} Q_i(dz_i|x_i,z_{1:(i-1)})d\Pr_\nu(x_i)$, respectively, and we denote the joint distribution of $Z_1,\dots, Z_i$, when $X_1,\dots, X_i$ are i.i.d. from $\Pr_0$, by
$$
\mathcal L_{Z_1,\dots, Z_{i}}^{(0)}(dz_{1:i}) := \mathcal L_{Z_i|z_{1:(i-1)}}^{(0)}(dz_i)\cdots \mathcal L_{Z_2|z_{1}}^{(0)}(dz_2)\mathcal L_{Z_1}^{(0)}(dz_1).
$$
Thus, by the convexity and tensorization property of the KL-divergence, we have
\begin{align}
&D_{KL}\left(Q_0^n, \bar{Q}^n\right)
\le
2^{-2^m}\sum_{\nu\in\mathcal V_m} D_{KL}\left(Q_0^n, Q_\nu^n\right) \notag\\
&\quad=
2^{-2^m}\sum_{\nu\in\mathcal V_m} \sum_{i=1}^n
\int_{\mathcal Z^{i-1}} D_{KL}\left(\mathcal L_{Z_i|z_{1:(i-1)}}^{(0)}, \mathcal L_{Z_i|z_{1:(i-1)}}^{(\nu)}\right)\,d\mathcal L_{Z_1,\dots, Z_{i-1}}^{(0)}.\label{eq:DKLbound1}
\end{align}
Next, for fixed $z_{1:(i-1)}\in\mathcal Z^{i-1}$ (if $i=1$ there is nothing to be fixed here), we bound the KL-divergence by the $\chi^2$-divergence, as in Lemma~2.7 of \citet{Tsybakov09}. Since $Q$ is $\alpha$-sequentially interactive differentially private, Lemma~\ref{lemma:Qdensities} establishes existence of a probability measure $\mu_{z_{1:(i-1)}}$ and a family of $\mu_{z_{1:(i-1)}}$-densities $z_i\mapsto q_i(z_i|x_i,z_{1:(i-1)})$ of $Q_i(\cdot|x_i, z_{1:(i-1)})$, $x_i\in\X$, with
$$
0<q_i(z_i|x_i,z_{1:(i-1)}) \le e^{2\alpha} q_i(z_i|x_i',z_{1:(i-1)}), \quad\forall z_i\in\mathcal Z, \forall x_i,x_i'\in\X.
$$
Abbreviating $q^{(\nu)}_{z_{1:(i-1)}}(z_i) := \int_{[0,1]} q_i(z_i|x_i,z_{1:(i-1)})d\Pr_\nu(x_i)$, we see that
\begin{align}
&D_{KL}\left(\mathcal L_{Z_i|z_{1:(i-1)}}^{(0)}, \mathcal L_{Z_i|z_{1:(i-1)}}^{(\nu)}\right)\notag\\
&\le
\int_{\mathcal Z} \left[\left(\frac{\int_{[0,1]} q_i(z_i|x_i,z_{1:(i-1)})d[\Pr_0 - \Pr_\nu](x_i)}{q^{(\nu)}_{z_{1:(i-1)}}(z_i)} \right)^2 q^{(\nu)}_{z_{1:(i-1)}}(z_i)\right] d\mu_{z_{1:(i-1)}}(z_i) \notag\\
&=
\int_{\mathcal Z} \left[\left(\int_{[0,1]}\left(\frac{ q_i(z_i|x_i,z_{1:(i-1)})}{q^{(0)}_{z_{1:(i-1)}}(z_i)} - c_{\alpha,0} \right)d[\Pr_0 - \Pr_\nu](x_i) \right)^2 \frac{q^{(0)}_{z_{1:(i-1)}}(z_i)}{q^{(\nu)}_{z_{1:(i-1)}}(z_i)}q^{(0)}_{z_{1:(i-1)}}(z_i)\right] d\mu_{z_{1:(i-1)}}(z_i),\label{eq:DKLbound2}
\end{align}
where we choose $c_{\alpha,0} = \frac{1}{2}(e^{2\alpha}+e^{-2\alpha})$. But since
\begin{align*}
e^{-2\alpha} \le \inf_{x_i'\in[0,1]} \frac{ q_i(z_i|x_i,z_{1:(i-1)})}{q_i(z_i|x_i',z_{1:(i-1)})} \le
\frac{ q_i(z_i|x_i,z_{1:(i-1)})}{q^{(0)}_{z_{1:(i-1)}}(z_i)}
\le
\sup_{x_i'\in[0,1]}\frac{ q_i(z_i|x_i,z_{1:(i-1)})}{q_i(z_i|x_i',z_{1:(i-1)})}
\le e^{2\alpha},
\end{align*}
and if we set $c_{\alpha,1} = \frac{1}{2}(e^{2\alpha}-e^{-2\alpha})$, we arrive at
$$
g_i(x_i) := g_i(x_i|z_1,\dots, z_i) := \frac{ q_i(z_i|x_i,z_{1:(i-1)})}{q^{(0)}_{z_{1:(i-1)}}(z_i)} - c_{\alpha,0} \in [-c_{\alpha,1},c_{\alpha,1}].
$$
Next, we consider the average over $\nu$ of the inner squared integral in \eqref{eq:DKLbound2}, i.e.,
\begin{align*}
&2^{-2^m} \sum_{\nu\in\mathcal V_m} \left( \int_{[0,1]} g_i(x_i)[f_0(x_i)-f_\nu(x_i)]\,dx_i\right)^2\\
&\quad=
\delta^22^{-2m(s+\frac{1}{2})} 2^{-2^m} \sum_{\nu\in\mathcal V_m}  \left( \sum_{k=1}^{2^{m}}\nu_{k} \int_{[0,1]} g_i(x_i)\psi_{mk}(x_i)\,dx_i\right)^2\\
&\quad=
\delta^2 2^{-2m(s+\frac{1}{2})} \sum_{k=1}^{2^{m}} \left(\int_{[0,1]} g_i(x_i)\psi_{mk}(x_i)\,dx_i\right)^2\\
&\quad\le
\delta^2 2^{-2m(s+\frac{1}{2})} c_{\alpha,1}^2 \sum_{k=1}^{2^{m}} \|\psi_{mk}\|_{1}^2
\le
\delta^2 2^{-2m(s+\frac{1}{2})} c_{\alpha,1}^2 2^m (2^{-m/2}\|\psi\|_1)^2\\
&\quad= \frac{2^{-m(2s+1)}}{4} (e^{2\alpha}-e^{-2\alpha})^2(\delta\|\psi\|_1)^2.
\end{align*}
Hence, using $\frac{q^{(0)}_{z_{1:(i-1)}}(z_i)}{q^{(\nu)}_{z_{1:(i-1)}}(z_i)} \le \|1/f_\nu\|_\infty\le2$, we see that the average over $\nu$ of \eqref{eq:DKLbound2} is bounded by the same expression multiplied by $2$, and \eqref{eq:DKLbound1} yields
$$
D_{KL}\left(Q_0^n, \bar{Q}^n\right)
\le
n \frac{2^{-m(2s+1)}}{2} (e^{2\alpha}-e^{-2\alpha})^2c_0^2 = \frac12,
$$
where $c_0= \delta\|\psi\|_1$ and $m$ is chosen as $m=\frac{\log(n(e^{2\alpha}-e^{-2\alpha})^2c_0^2)}{2(s+\frac{1}{2})\log2}$.
In view of Markov's inequality and \eqref{eq:quadTV}, this leads to
\begin{align*}
 \sup_{f\in\bar{\P}_s^{pq}(L,M)} \E_{Q\Pr_f^n}\left[\left|\hat{D}_n-D(f)\right|^2 \right]\;&\ge\; 
 \Delta^2 \sup_{f\in\bar{\P}_s^{pq}(L,M)} 
 Q\Pr_f^{n}( |\hat{D}_n - D(f)|^2>\Delta^2 ) \\
 &\ge
 \frac{\Delta^2}{2}\left(1-\sqrt{D_{KL}(Q_0^n,\bar{Q}^n)/2}\right) \\
 &\ge \frac{\Delta^2}{4} 
 =
 \left[ n(e^{2\alpha}-e^{-2\alpha})^2c_0^2\right]^{-\frac{4s}{2s+1}}\frac{\delta^4}{16},
\end{align*}
which finishes the proof. \hfill\qed


\subsection{Auxiliary lemmas for Section~4}

The following lemma is a non-asymptotic version of Lemma~2 in \citet{Butucea08}.

\begin{lemma}\label{lemma:Integrals}
For arbitrary finite constants $A, B, r, s>0$, $a,b\ge0$ and $v>0$, such that $\frac{Arv^r}{a+1}>1$, we have
\begin{eqnarray}
\int_v^\infty u^a e^{-A u^r} du &\le& \frac 1{A r} v^{a + 1-r} e^{-A v^r} \left(1-\frac{a+1}{Arv^r}\right)^{-1}, \quad\text{and} \label{eq1} \\
\int_0^v  u^b e^{B u^s} du &\le& \frac 1{ B s} v^{b+1-s} e^{B v^s} \left(1+\frac{b+1}{Bsv^s}\right)^{-1}.
\label{eq2}
\end{eqnarray}
\end{lemma}
\begin{proof}
To see \eqref{eq1}, simply integrate by parts to get
\begin{align*}
\int_v^\infty u^a e^{-A u^r} du
&= \left[ \frac{u^{a+1}}{a+1} e^{-Au^r}\right]_v^\infty - \int_v^\infty \frac{u^{a+1}}{a+1} e^{-Au^r} (-Aru^{r-1})\,du \\
&= - \frac{v^{a+1}}{a+1} e^{-Av^r} + \frac{Ar}{a+1} \int_v^\infty u^{a+r} e^{-Au^r}\,du\\
&\ge
- \frac{v^{a+1}}{a+1} e^{-Av^r} + \frac{Ar v^r}{a+1} \int_v^\infty u^{a} e^{-Au^r}\,du.
\end{align*}
For $\frac{Ar v^r}{a+1} > 1$, this is equivalent to
\begin{align*}
\int_v^\infty u^a e^{-A u^r} du \le \frac{v^{a+1}}{a+1} e^{-Av^r} \left( \frac{Ar v^r}{a+1}-1\right)^{-1},
\end{align*}
which implies the desired result. \eqref{eq2} follows analogously and without any further restrictions on the constants.
\end{proof}

\begin{lemma}\label{lemma:Qdensities}
Let $\alpha\in(0,\infty)$, $(\X, \mathcal F)$ and $(\mathcal Z, \mathcal G)$ be measurable spaces and $Q$ a Markov kernel from $\X$ to $\mathcal Z$. If $Q(A|x) \le e^\alpha Q(A|x')$, for all $A\in \mathcal G$ and all $x,x'\in\X$, then there exists a probability measure $\mu$ and a family of $\mu$-densities $(q_x)_{x\in\X}$, such that for every $x\in\X$, $dQ(\cdot|x) = q_x\,d\mu$ and $e^{-\alpha}\le q_x(z)\le e^{\alpha}$, for \textbf{all} $z\in\mathcal Z$.
\end{lemma}

\begin{proof}
Let $x_0\in\X$ and $\mu := Q(\cdot|x_0)$. For a fixed $x\in\X$, we have $Q(\cdot|x) \ll \mu$, and we write $\tilde{q}_x$ for a corresponding density. Since $\int_A \tilde{q}_x\,d\mu = Q(A|x) \le e^\alpha Q(A|x_0) = \int_A e^\alpha\,d\mu$ and $Q(A|x) \ge e^{-\alpha} Q(A|x_0) = \int_A e^{-\alpha}\,d\mu$, for all $A\in\mathcal G$, we have $e^{-\alpha}\le \tilde{q}_x \le e^\alpha$, $\mu$-almost surely. Let $N_x\in\mathcal G$ be the corresponding $\mu$-null set. Then define $q_x(z) = \tilde{q}_x(z)$, if $z\in N_x^c$, and set $q_x(z) = 1$, otherwise. Thus, $q_x$ is still a $\mu$-density of $Q(\cdot|x)$ with $e^{-\alpha}\le q_x(z) \le e^\alpha$.
\end{proof}


\section{Proofs of Section~6 on adaptive estimation}

\subsection{Proof of Theorem~6.1 (non-interactive protocol)}

Let $\Pr_n$ denote the empirical measure with respect to $X_1,...,X_n$. 
We decompose our estimator into the non-private version and new terms $A_J$ and $B_J$ due to privacy as follows:
\begin{align*}
\hat D_J &- 2 \left( \frac 1n \sum_{i=1}^n f(X_i) -D  \right) - D \\
&=  U_n(H_J) - 2 (\mathbb{P}_n - \mathbb{P})(f- f_J) - \|f- f_J\|_2^2 + A_J + B_J,
\end{align*}
where $f_J(x) = \sum_{j=-1}^{J-1} \sum_{k=0}^{(1\lor 2^j)-1} 
\beta_{jk} \psi_{jk}(x)$ is the projection of $f$ at the resolution level $J$, $U_n(H_J)$ is the $U$-statistic with kernel $H_J$ given by
\begin{align*}
H_J(x,y) &= \sum_{j=-1}^{J-1} \sum_{k=0}^{(1\lor 2^j)-1} (\psi_{jk}(x) - \beta_{jk})(\psi_{jk}(y) - \beta_{jk}),\\ 
U_n(H_J)&=\frac{1}{n(n-1)}\sum_{\substack{i,h=1\\ i\not=h}}^nH_J(X_i,X_h)
\end{align*}
and, finally,
\begin{align*}
A_J &= \frac 2{n(n-1)} \sum_{i \ne h} \sum_{j=-1}^{J-1} \sum_{k=0}^{(1\lor 2^j)-1} \sigma_j \frac \sigma \alpha W_{ijk} \psi_{jk}(X_h), \\
B_J &= \frac 1{n(n-1)} \sum_{i \ne h} \sum_{j=-1}^{J-1} \sum_{k=0}^{(1\lor 2^j)-1} \sigma_j^2 \frac{\sigma^2}{\alpha^2} W_{ijk} W_{hjk}.
\end{align*}
We decompose $pen^{(NI)}(J)$ into the sum of $ pen(J)$, $pen_A(J)$ and $pen_B(J)$, for all $J$ in $\mathcal{J}$, that we specify below. Let us denote
\begin{align*}
    V_J &= U_n(H_J) - 2 (\mathbb{P}_n - \mathbb{P})(f- f_J) - \|f- f_J\|_2^2  - pen(J) \\
    & \quad + A_J - pen_A(J)\\
    &\quad + B_J - pen_B(J).
\end{align*}
By the definition of $\hat D^{(NI)}$ we get
$$
\hat D^{(NI)} - 2 \left( \frac 1n \sum_{i=1}^n f(X_i) -D  \right) - D = \sup_{J \in \mathcal{J}} V_J.
$$
Using
$$
\left| \sup_{J \in \mathcal{J}} V_J\right| = \sup_{J \in \mathcal{J}} (V_J)_+ \vee \inf_{J \in \mathcal{J}} (V_J)_-,
$$
we have that
$$
\E \left[(\sup_{J \in \mathcal{J}} V_J)^2 \right] \leq \sum_{J \in \mathcal{J}} \E[(V_J)_+^2] + \inf_{J \in \mathcal{J}} \E[(V_J)_-^2].
$$
We start with bounding the first term on the right-hand side in the former inequality. 
Following \cite{Laurent2005}, we similarly use the 
concentration inequality for $U$-statistics in \cite{HouRey2003} and get 
\begin{align*}
    \Pr &\left( |U_n(H_J)| \geq \frac C{n-1} \left[\sqrt{2(J+1) M 2^J n(n-1)} \sqrt{t} + 8(J+1) Mn t + 2(J+1) 2^{J+2} t^2 \right]
    \right)\\
    & \leq 5.6 \exp(-t),
\end{align*}
which, combined with the deviation of the empirical process part, provides the penalty
\begin{align*}
pen(J) &= \frac {\kappa_0} n \left( \sqrt{M (J+1) 2^J \log(2^J+1) } + M(J+1) \log(2^J +1) \right. \\
& \left. + \frac{(J+1)2^J \log^2(2^J +1)}{n} \right).
\end{align*}
with $\kappa_0$ as specified in \cite{Laurent2005}.
We obtain 
\begin{align*}
    \sum_{J \in \mathcal{J}}\E &\left[ \left(\Big[U_n(H_J) - 2 (\mathbb{P}_n - \mathbb{P})(f- f_J) - \|f- f_J\|_2^2 - pen(J)\Big]_+ \right)^2\right] \lesssim \frac{1}{n}
    \end{align*}
    and
    \begin{align*}
    \E &\left[ \left( \Big[U_n(H_J) - 2 (\mathbb{P}_n - \mathbb{P})(f- f_J) - \|f- f_J\|_2^2 - pen(J)\Big]_- \right)^2\right]\\
    &\leq C_0 \left( \|f-f_J\|_2^4 + pen^2(J)\right)
\end{align*}
for some absolute constant $C_0>0$.

Let us deal now with the additional terms that occur due to privacy in our setup.

{\bf First additional term:}
\medskip
\noindent
Recall that
\begin{align*}
A_J&=\frac{1}{n} \sum_{i=1}^n  \sum_{j=-1}^{J-1} \sum_{k=0}^{(1\vee 2^j)-1} W_{ij k} \sigma_j \frac{\sigma}{\alpha} \Xi_{ijk}, \quad \text{with } \Xi_{ijk}= \frac{1}{n-1} \sum_{h\not=i} \psi_{j k}(X_h).
\end{align*}
The aim is to specify some potentially random and ideally small real numbers $y_J$, $J\in\mathcal{J}$, such that
$$
\sum_{J\in \mathcal{J}} \E [(A_J-y_J)_+^2]  \lesssim \frac{1}{n \alpha^2}.
 $$
 For any $y=y((X_i))>0$, 
\begin{align*}
(\Delta_A) & :=\E([A_J-y]_+^2) =\E\E\big[[A_J-y]_+^2\big\arrowvert(X_i)\big]\\
&= \E {\int_0^\infty t\Pr([A_J-y]_+ > t\arrowvert(X_h)) d t} . 
\end{align*}
Due to the independence of $(X_i)$ and $(W_{ijk})$,  we use Bernstein's inequality for sums of independent Laplace random variables (cf. \cite{BouLugMas13}) to get
\begin{flalign} \label{eq:Bernstein}
\Pr\big(A_J>\eta\big \arrowvert  (X_i)\big) 
\le 
\exp\Bigg( 
-\frac{\eta^2}{2} \frac{1}{
\underbrace{\frac{2}{n} \frac1n \sum_{i=1}^n \sum_{j,k} \sigma_j^2\frac{\sigma^2}{\alpha^2} \Xi_{ijk}^2}_{=:C_1} 
+ \eta \underbrace{\frac{\max_{i,j,k} \sigma_j \sigma|\Xi_{ijk}|}{n\alpha}}_{\substack{=:C_2\\ }  }} \Bigg)
\end{flalign}
where we drop the dependence on $J$ in $C_1 = C_{1,J}$ and $C_2 = C_{2,J}$.
Here and in what follows, the double sum over $j,k$ is abbreviated by $\sum_{j,k}$, dropping in particular its dependence on $J$ if this is clear from the context. 

Hence, using Lemma \ref{lemma: bound},
{\allowdisplaybreaks
\begin{align*}
 {(\Delta_A)} &\leq   
\E \int_{y}^\infty t \Pr (A_J>t\mid (X_i)) d t\\
&\leq  \E \int_{y} ^\infty t \exp \Big( -\frac{t^2}{2} \frac{1}{C_1+t C_2}\Big) d t
\\*
&\leq  \E \bigg\{ 
\underbrace{2C_1 \exp \Big( -\frac{y^2}{4C_1}\Big)}_{{\text{(I)}}}
+ \underbrace{4 C_2\Big[ y \vee \frac{C_1}{C_2}\Big] \exp\Big(-\frac{y \vee \frac{C_1}{C_2}}{4 C_2} \Big) }_{{\text{(II)}}} + \\
&\quad\quad
 + \underbrace{16 C_2^2 \exp\Big(-\frac{y \vee \frac{C_1}{C_2}}{4 C_2} \Big) }_{{\text{(III)}}}
\bigg\}.
\end{align*}
} 

Now define 
\begin{align}\label{eq: y_J}
y_J := 8 \sqrt{C_1}\,  \log[2^{4J+1}]
\end{align}
and insert it into (I), (II) and (III).
Then 
\begin{align*}
\text{(I)}
&= 2C_1 \exp\Big( -\frac{y_J^2}{4 C_1}\Big) 
\le 2 C_1 \exp\Big(-4\log[2^{4J+1}]\Big)
 = 2C_1 \frac{1}{2^{(4J+1)4}}.
\end{align*}
Concerning (II), note that if $y_J>\frac{C_1}{C_2}$, then
\begin{align*}
\frac{[y_J\vee\frac{C_1}{C_2}]}{4C_2} 
&= 2 \frac{\sqrt{C_1}}{C_2} \log[2^{4J+1}]
\ge 2 \log \big[2^{4J+1}\big]
\end{align*}
because $\sqrt{C_1}\geq C_2$. 
If $\displaystyle y_J\le\frac{C_1}{C_2}$, then 
\begin{align*}
&8 \sqrt{C_1}\, \log\big[2^{4J+1}\Big]\leq \frac{C_1}{C_2}\
 \Longleftrightarrow \ \frac{\sqrt{C_1}}{4C_2} \ge 2 \log[ 2^{4J+1}].
 \end{align*}
 Consequently,
 \begin{flalign}
\label{eq:5-st25-*}
 \quad \frac{\big[ y_J\vee \frac{C_1}{C_2}\big]}{4 C_2} &\ge 2 \log \Big[ 2^{4J+1} \Big].&&
\end{flalign}
 Finally, $C_2\sqrt{C_1}\leq C_1$ implies 
 $$
 4C_2 \Big[y_J\vee \frac{C_1}{C_2}\Big] \leq 8 C_1 \log[2^{4J+1}].
 $$
 Summarizing,  
 \begin{align*}
\text{(II)} 
&= 4C_2 [y_J\vee \frac{C_1}{C_2}]\exp\Big( -\frac{[y_J\vee \frac{C_1}{C_2}]^2}{4C_2}\Big)\\
&\le8 C_1 \log \Big[ 2^{4J+1} \Big]\exp\Big( -2 \log\big(2^{4J+1}\big) \Big).
\end{align*}
Again by \eqref{eq:5-st25-*} and the inequality $C_2^2\leq C_1$,
\begin{align*}
\text{(III)} &= 16 C_2^2 \exp \Big( -\frac{[y_J\vee \frac{C_1}{C_2}]}{4C_2}\Big)&&
\\
&\leq 16 C_1 \exp\Big( -2 \log (2^{4J+1}) \Big).
\end{align*}
This gives,
\begin{flalign*}
\quad  (\Delta_A)
&\le  \E \int_{y_J}^\infty t \Pr (X>t\mid (X_i))d t && \\
&\le  \E \bigg[ 
2 C_1\cdot \frac{1}{2^{(4J+1)4}} 
+  \frac{32 C_1 \log [2^{4J+1}] }{2^{2(4J+1)}} 
+ \frac{16 C_1 }{2^{2(4J+1)}} 
\bigg]
\end{flalign*}
Now we are in the position to evaluate the expected value. Because of
\begin{align*}
\E \Big| \frac{1}{n-1} \sum_{h\not=i} \psi_{jk}(X_h)\big) \Big|^2 = \Big(1-\frac{1}{n-1}\Big)\beta_{jk}^2+\frac{1}{n-1}\E\big[\psi_{jk}(X_1)^2\big],
\end{align*}
we get
\begin{align*}
\E C_1 &\leq  \frac{2 \sigma^2}{n\alpha^2} \sum_{j,k} 2^{j} (1+j)^{2a}\bigg\{\beta_{jk}^2+\frac{\E\big[\psi_{jk}(X_1)^2\big]}{n-1}\bigg\} \\
& \lesssim \frac{1}{n \alpha^2} \Big(J^{2a+1}2^{J}+\frac{1}{n-1}2^{2J}\Big) .
\end{align*}
Therefore, with $y_J$  in \eqref{eq: y_J},  $\E [(A_J-y_J)_+^2]  \lesssim \frac{1}{n \alpha^2} \frac{1}{2^J}$. As a consequence,
$$
\sum_{J\in \mathcal{J}} \E [(A_J-y_J)_+^2]  \lesssim \frac{1}{n \alpha^2}.
 $$
Finally, with $\bar y_J: = 8 \sqrt{\E C_1}\,  \log[2^{4J+1}]$ and 
$$
pen_A(J) = 8  \, \left(\frac{J^{2a+1} 2^J}{n \alpha^2} + \frac{2^{2J}}{n^2 \alpha^2} \right)^{1/2}\,  \log[2^{4J+1}],
$$ 
such that $\bar y_J \leq pen_A(J)$, we arrive at
\begin{align}
\sum_{J\in \mathcal{J}} \E [(A_J-pen_A(J))_+^2] 
&\leq 2  \sum_{J\in \mathcal{J}} \E [(A_J-y_J)_+^2]  + 2 \sum_{J\in \mathcal{J}} \E [(\bar y_J-y_J)^2] \nonumber \\
&\lesssim \frac{1}{n \alpha^2} + \sum_{J\in \mathcal{J}} \log^2[2^{4J+1}] (\Var(C_{1}))^{1/2}.\label{boundC1}
\end{align}
In order to bound $\Var (C_1)$, observe that
\begin{align*}
    C_{1} & = 
    \frac{2}{n} \frac1n \sum_{i=1}^n \sum_{j,k} \sigma_j^2\frac{\sigma^2}{\alpha^2} \Xi_{ijk}^2\\
    &=   \frac{2}{n} \frac1n \sum_{i=1}^n \sum_{j,k} \sigma_j^2\frac{\sigma^2}{\alpha^2} \frac 1{(n-1)^2} \sum_{h,h' \ne i} \psi_{jk}(X_h)\psi_{jk}(X_{h'}) \\
    & = \frac{2 \sigma^2}{n^2(n-1)\alpha^2} \sum_{h=1}^n \sum_{j,k}\sigma_j^2 \psi^2_{jk}(X_h)
    +   \frac{2 \sigma^2(n-2)}{n^2(n-1)^2\alpha^2} \sum_{h\ne h'}^n \sum_{j,k}\sigma_j^2 \psi_{jk}(X_h)\psi_{jk}(X_{h'}) \\
    & = : C_{1,1}+C_{1,2}.
\end{align*}
Note that $\Var(C_{1}) \leq 2 \Var(C_{1,1}) +2 \Var(C_{1,2}) $. Therefore, 
\begin{align*}
    \Var(C_{1,1}) & \asymp \frac{1}{n^5 \alpha^4} \Var \left( \sum_{j,k} \sigma_j^2 \psi_{j,k}^2(X_1)
    \right)\\
    &\lesssim \frac{1}{n^5 \alpha^4} \E  \left(
     \sum_{j,j',k,k'}\sigma_j^2 \sigma_{j'}^2 \psi_{j,k}^2(X_1) \psi_{j',k'}^2(X_1)
    \right)^2 \\
    &\lesssim \frac{1}{n^5 \alpha^4} 
    \sum_{j\geq j',k}\sigma_j^2 \sigma_{j'}^2 \sum_{k'} \int \psi_{j,k}^2 \psi_{j',k'}^2 f\\
    &\lesssim \frac{J^{4a}}{n^5 \alpha^4} 
    \sum_{j\geq j',k}2^{2(j+j')} \sum_{k'} \int I_{j,k} I_{j',k'} \cdot M,
\end{align*}
where we denote by $I_{j,k} = [\frac{k}{2^j}, \frac{k+1}{2^j}]$. We use repeatedly that for any, $j,\,k$ and $j'$ such that $j \geq j'$, there are only $2^{j-j'}$ values of $k'$ such that $I_{j,k} I_{j',k'} \ne 0$. For these values of $k'$ the length of the interval $I_{j,k} I_{j',k'}$ is $2^{-j}$. Therefore 
$\sum_{k'} \int I_{j,k} I_{j',k'} \leq 2^{j-j'}\cdot 2^{-j}$.
We obtain 
$$
\Var(C_{1,1}) \lesssim  \frac{J^{4a}}{n^5 \alpha^4} \sum_{j\geq j',k} 2^{2j +j'}
 \lesssim  \frac{J^{4a}}{n^5 \alpha^4} \sum_{j} 2^{4j} \asymp \frac{J^{4a} \cdot 2^{4J}}{n^5 \alpha^4}.
$$
We plug this in \eqref{boundC1} to get
\begin{align*}
    &\sum_{J\in \mathcal{J}} \log^2[2^{4J+1}] (\Var(C_{1,1}))^{1/2}\\
    &\lesssim \sum_{J\in \mathcal{J}} (4J+1)^2  \frac{J^{2a} \cdot 2^{2J}}{n^{5/2}\alpha^2}
    \lesssim \frac 1n \cdot \frac{(J_{max})^{2+2a} 2^{2J_{max}}}{n^{3/2} \alpha^2} \lesssim \frac 1n,
\end{align*}
where we use that $2^{J_{max}} \asymp (n \alpha^2)^{2/3} \log^{-\kappa /3}(n \alpha^2)$ for some $\kappa>4(a+1) >0$.

Similarly, for the term appearing in $C_{1,2} - \E(C_{1,2})$,
\begin{align*}
   \E\biggl( \sum_{h\ne h'}^n& \sum_{j,k}\sigma_j^2 (\psi_{jk}(X_h)\psi_{jk}(X_{h'}) - \beta_{jk}^2)
    \biggr)^2\\
    & = n(n-1) \E \left(\sum_{j,k}\sigma_j^2 (\psi_{jk}(X_1)\psi_{jk}(X_{2}) - \beta_{jk}^2)\right)^2 \\
    &\ \ \  + n(n-1)(n-2)\E  \left(\sum_{j,k}\sigma_j^2 (\psi_{jk}(X_1)\psi_{jk}(X_{2}) - \beta_{jk}^2)\right) \\
    & \ \ \ \ \ \ \cdot \left(\sum_{j',k'}\sigma_j^2 (\psi_{j'k'}(X_1)\psi_{j'k'}(X_{3}) - \beta_{j'k'}^2)\right) = : T_1 + T_2.
\end{align*}
We bound from above $\Var(C_{1,2}) \lesssim \frac{1}{n^6\alpha^4}( T_1 + |T_2|)$, with \begin{allowdisplaybreaks}
\begin{align*}
    \frac{1}{n^6\alpha^4}T_1 & \lesssim \frac{1}{n^6 \alpha^4} \Var\left(\sum_{h\ne h'}^n \sum_{j,k}\sigma_j^2 \psi_{jk}(X_h)\psi_{jk}(X_{h'})\right)\\
    & \lesssim \frac{1}{n^4 \alpha^4} \Var\left(\sum_{j,k}\sigma_j^2 \psi_{jk}(X_1)\psi_{jk}(X_{2}) \right)\\
    & \lesssim \frac{1}{n^4 \alpha^4} \E \left[
    \left( \sum_{j,k}\sigma_j^2 \psi_{jk}(X_1)\psi_{jk}(X_{2})\right)^2
    \right]\\
    & \lesssim \frac{1}{n^4 \alpha^4} \sum_{j,j',k,k'} \sigma_j^2 \sigma_{j'}^2 \E^2[\psi_{jk}(X_1) \psi_{j'k'}(X_1)]\\
    &\lesssim \frac{1}{n^4 \alpha^4} \sum_{j,j',k} \sigma_j^2 \sigma_{j'}^2 2^{j+j'}\sum_{k'}\bigg(\int I_{jk} I_{j'k'}\bigg)^2 M^2\\
    &\lesssim  \frac{J^{4a}}{n^4 \alpha^4} \sum_{j \geq j', k} 2^{2(j+j')} \cdot 2^{j-j'} 2^{-2j} M^2 \\
    & \lesssim  \frac{J^{4a}}{n^4 \alpha^4} \sum_{j\geq j'} 2^j \cdot 2^{j+j'}\lesssim
     \frac{J^{4a}}{n^4 \alpha^4} 2^{3J}.
\end{align*}
Therefore, for all $J$ in $\mathcal{J}$, we get
\begin{equation}\label{boundT1}
    \frac 1{n^3\alpha^2} \sqrt{T_1} \lesssim \frac 1n \cdot \frac{J^{2a} \cdot 2^{3J/2}}{n \alpha^2} .
\end{equation}
Now,
\begin{align*}
    \frac{1}{n^6\alpha^4}|T_2| & \lesssim 
    \frac{J^{4a}}{n^3\alpha^4} \left|\sum_{j,k,j',k'} 2^{j+j'} \left( \E[\psi_{jk}(X_1)\psi_{j'k'}(X_1)\psi_{jk}(X_2)\psi_{j'k'}(X_3)]-\beta_{jk}^2\beta_{j'k'}^2 
    \right)\right|\\
    & = \frac{J^{4a}}{n^3\alpha^4} \left| \sum_{j,k,j',k'} 2^{j+j'}  \int \psi_{jk}\psi_{j'k'}f \cdot  \beta_{jk}\beta_{j'k'} - \bigg(\sum_j 2^j \|\beta_{j \cdot }\|_2^2\bigg)^2
    \right|\\
    & \lesssim \frac{J^{4a}}{n^3\alpha^4} \cdot \left(\sum_{j \geq j',k} 2^{j+j'} \sum_{k'} 2^{\frac 12 (j+j')}\int I_{jk} I_{j'k'}M \cdot  |\beta_{jk}\beta_{j'k'}|
    + 2^{2J} M^2 \right)\\
    & \lesssim \frac{J^{4a}}{n^3\alpha^4} \cdot \left(\sum_{j \geq j',k} 2^{\frac 32(j+j')} |\beta_{jk}|\bigg(\sum_{k'} \bigg[\int I_{jk} I_{j'k'}\bigg]^2 \bigg)^{1/2} \|\beta_{j'\cdot}\|_2
    + 2^{2J} M^2 \right)\\
    & \lesssim \frac{J^{4a}}{n^3\alpha^4} \cdot \left(\sum_{j \geq j',k} 2^{\frac 32(j+j')} |\beta_{jk}| 2^{-\frac 12 (j+j')} 
    + 2^{2J} M^2 \right)\\
    &\lesssim \frac{J^{4a}}{n^3\alpha^4} \cdot \left(\sum_{j \geq j'} 2^{j+j'} 2^{j/2} \|\beta_{j\cdot}\|_2
    + 2^{2J} M^2 \right) \lesssim \frac{ 2^{3J}}{n^3\alpha^4}.
\end{align*}
Thus
\begin{equation}
    \label{boundT2}
    \frac 1{n^3\alpha^2} \sqrt{T_2} \lesssim \frac 1n \cdot \frac{2^{3J/2}}{\sqrt{n} \alpha^2} .
\end{equation}
Combining \eqref{boundT1} and \eqref{boundT2}, we get that for any $\kappa \geq 4(a +1)$:
\begin{align*}
    &\sum_{J\in \mathcal{J}} \log^2[2^{4J+1}] (\Var(C_{1,2}))^{1/2}\\
    &\lesssim \frac{1}{n}\sum_{J\in \mathcal{J}} (4J+1)^2  \frac{2^{3J/2}}{\sqrt{n}\alpha^2}
    \lesssim \frac 1n \cdot \frac{(J_{max})^{2+2a} 2^{3J_{max}/2}}{\sqrt{n} \alpha^2} \lesssim \frac 1n.
\end{align*}
\end{allowdisplaybreaks}

Indeed, remember that $2^{3J_{max}}/(n\alpha^2)^2 \lesssim 1/\log^{\kappa }(n \alpha^2)$. 
Plugging the bounds for $C_{1,1}$ and $C_{1,2}$ into \eqref{boundC1}, we get that
$$
\sum_{J\in \mathcal{J}} \E [(A_J-pen_A(J))_+^2] \lesssim \frac 1{n\alpha^2}.
$$
We conclude for the negative part that
$$
\E [(A_J-pen_A(J))_-^2] \leq  \E A_J^2+  pen_A^2(J) \leq 2 pen_A^2(J).
$$


\bigskip
{\bf Second additional term:}
\noindent
Recall now
\begin{align*}
B_J&=\frac{1}{n} \sum_{i=1}^n  \sum_{j=-1}^{J-1} \sum_{k=0}^{(1\vee 2^j)-1} W_{ij k} \sigma_j^2 \frac{\sigma^2}{\alpha^2} \Big[ \frac{1}{n-1} \sum_{h\not=i} W_{hj k} \Big],
\end{align*}
The aim is to specify some ideally small real numbers $y_J$, $J\in\mathcal{J}$, such that
$$
\sum_{J\in \mathcal{J}} \E [(B_J-y_J)_+^2]  \lesssim_{\log} \frac{1}{n^2\alpha^4}.
 $$
 For any $y>0$, 
 \begin{align*}
(\Delta_B):=\E([B_J-y]_+^2) &= \int_0^\infty t\Pr([B_J-y]_+ > t) d t  
\le {\int_y^\infty u\Pr(X>u) d u}.
\end{align*}
Let $(W_{ij k}')$ be an independent copy of $(W_{ij k})$ and 
$$
B_J'=\frac1n \sum_{i=1}^n \sum_{j,k} W_{ij k} \sigma_j^2 \frac{\sigma^2}{\alpha^2}\Big[ \frac{1}{n-1} \sum_{h\not=i} W_{hj k}'\Big].
$$
Then, with the notation $\zeta_{ijk} = \frac{1}{n-1} \sum_{h\not=i} W_{hj k}'$,
\begin{flalign*}
\Pr\big(B_J'>\eta\big \arrowvert  (W_{hjk}')_{hjk}\big) 
\le 
\exp\Bigg( 
-\frac{\eta^2}{2} \frac{1}{
\underbrace{\frac{2}{n} \frac1n \sum_{i=1}^n \sum_{j,k} \sigma_j^4 \frac{\sigma^4}{\alpha^4}\zeta_{ijk}^2}_{=:C_1} 
+ \eta \underbrace{\frac{\max_{i,j,k} \sigma_j^2\sigma^2|\zeta_{ijk}|}{n\alpha^2}}_{\substack{=:C_2\\ }  }} \Bigg).
\end{flalign*}
Here and in what follows, the double sum over $j,k$ is abbreviated by $\sum_{j,k}$, dropping in particular its dependence on $J$ if this is clear from the context. Note that $C_1$ and $C_2$ are random and equally depend on $J$.

By the decoupling inequality of \cite{delaPena95}, 
there exists some universal constant $C>0$ with
{\allowdisplaybreaks
\begin{align*}
{(\Delta_B)} &\le C\int_y^\infty u \Pr \Big(B_J'>\frac{u}{C} \Big) d u, \qquad\qquad \frac{u}{C}=t
\\*
&= C^3 \int_{y/C}^\infty t \, \Pr(B_J' > t) d t
\\
&= C^3 \int_{y/C}^\infty t \,\E\Pr(B_J'>t\mid W')d t
\\
&= C^3 \E \int_{y/C}^\infty t \Pr (B_J'>t\mid W') d t
\\
&\leq C^3 \E \int_{y/C=:y'} ^\infty t \exp \Big( -\frac{t^2}{2} \frac{1}{C_1+tC_2}\Big) d t
\\*
&\leq C^3 \E \bigg\{ 
\underbrace{2C_1 \exp \Big( -\frac{y'^2}{4C_1}\Big)}_{{\text{(I)}}}
+ \underbrace{4 C_2\Big[ y' \vee \frac{C_1}{C_2}\Big] \exp\Big(-\frac{y' \vee \frac{C_1}{C_2}}{4 C_2} \Big) }_{{\text{(II)}}} +\\
&\hspace{2cm}
 + \underbrace{16 C_2^2 \exp\Big(-\frac{y' \vee \frac{C_1}{C_2}}{4 C_2} \Big) }_{{\text{(III)}}}
\bigg\},
\end{align*}
}where we have also used Lemma \ref{lemma: bound}. 

As we have used Fubini's theorem in the third equation, $y$ is not allowed to depend on $(W')$. Inspired by the choice for the previously treated  first term, define
\begin{align*}
y' := 8 \sqrt{\E C_1}\,  \log[2^{4J+1}].
\end{align*}
In order to bound (I), (II) and (III), we need concentration of $C_1$ around $\E C_1$ (cf. Lemma~\ref{lemma:C1tail}).

We evaluate the expectation. Because of
\begin{align*}
\E \Big| \frac{1}{n-1} \sum_{h\not=i} W_{h j k}' \Big|^2 = \frac{2}{n-1},
\end{align*}
we get
\begin{align*}
\E C_1 &= \frac{2 \sigma^4}{n\alpha^4} \sum_{j,k} \sigma_j^2 \frac{2}{n-1}
 \asymp \frac{1}{n^2 \alpha^4} J^{4a}2^{3J}.
\end{align*}

Now, we are in the position to continue with bounding (I), (II) and (III). 

\medskip
\underline{Upper bound of $\E$(I):}
Let us denote by $V = C_1/ \E C_1$. Then
\begin{align*}
\E\text{(I)}
&= 2 \E C_1 \cdot \E \bigg[V \exp\Big( -\frac{16}{V} \cdot  \log^2[2^{4J+1}]\Big)\bigg] .
\end{align*}
On the set $\{ V\leq n^{-1/3} \}$, we get an easy upper bound 
$$
2  \E C_1 \cdot n^{-1/3} \exp(-16 n^{1/3}  \log^2[2^{4J+1}]) \leq 2  \E C_1 \cdot \frac 1{2^{4J}}, 
$$
while on the set $\{V>n^{-1/3} \}$ we bound the exponential term by 1 and write
\begin{align*}
& 2 \E C_1 \cdot \sum_{k \in \N}\E \bigg[V \exp\Big( -\frac{16}{V} \cdot  \log^2[2^{4J+1}]\Big)\cdot \mathds{1} (\frac{1}{n^{1/3}}\vee k^5 < V \leq (k+1)^5 )\bigg] \\
&\lesssim \E C_1 \cdot \sum_{k\geq 1} (k+1)^5 \bigg\{ 2^{J+2} \exp\left( -2^J k^5  \right) + \frac{1}{k^{10} n^3 2^J} \bigg\}\\
&\lesssim \E C_1 \cdot  \bigg\{\sum_{k \geq 1 } \frac{((k+1)2^J)^5}{2^{5J}} 2^{J} \exp\left( -2^J k \right) + \frac{1}{n^3 2^J} \bigg\}.
\end{align*}
Finally, \begin{align*}
    \E(I)& \lesssim \E C_1 \cdot \bigg\{ \frac{1}{2^{4J}}   + \frac{1}{n^3 2^J} \bigg\}.
\end{align*}

\medskip
\noindent
{\underline{Upper bound of $\E$(II):}} 
If $y' > C_1/C_2$, then $\sqrt{C_1}\geq C_2$ implies 
$$
\frac{8 \sqrt{\E C_1}\log(2^{4J+1})}{4 C_2}\geq 2 \frac{\sqrt{\E C_1}}{\sqrt{C_1}}\log(2^{4J+1}).
$$
If $y'\leq C_1/C_2$, we conclude that
$$
\frac{\sqrt{C_1}}{4 C_2}\geq 2 \log(2^{4J+1})\frac{\sqrt{\E C_1}}{\sqrt{C_1}}.
$$
Put the two cases together to get
 \begin{flalign}
\label{eq:5-st25-**}
\quad \frac{\big[ y'\vee \frac{C_1}{C_2}\big]}{4 C_2} &\ge 2 \bigg(\frac{\sqrt{\E C_1}}{\sqrt{C_1}} + \frac{\E C_1}{C_1}\bigg) \log \Big[ 2^{4J+1} \Big]&&
\end{flalign}
 Finally, $C_2\sqrt{C_1}\leq C_1$ implies 
 $$
 4C_2 \Big[y'\vee \frac{C_1}{C_2}\Big] \leq 8 (C_1+\E C_1) \log[2^{4J+1}].
 $$
 Summarizing,  
 \begin{align*}
\text{(II)} 
&= 4C_2 [y'\vee \frac{C_1}{C_2}]\exp\Big( -\frac{[y'\vee \frac{C_1}{C_2}]^2}{4C_2}\Big)\\
&\le 8 (C_1 +\E C_1)\log \Big[ 2^{4J+1} \Big]\exp\bigg( -2 \bigg(\frac{\sqrt{\E C_1}}{\sqrt{C_1}} + \frac{\E C_1}{C_1}\bigg) \log\big(2^{4J+1}\big) \bigg).
\end{align*}
Therefore, using again the notation $V=C_1/\E C_1$ and proceeding as before, we get:
\begin{align*}
\E\text{(II)}&\lesssim \E C_1\log \Big[ 2^{4J+1}\Big]\E \left[(V+1) \cdot \exp \left( -2 \big( \frac{1}{\sqrt{V}} + \frac{1}{V}\big) \log [2^{4J+1}] \right) \right]\\
&\lesssim \E C_1 \cdot \log \Big[ 2^{4J+1}\Big] \cdot \bigg\{ \frac{1}{2^{4J}}+\frac{1}{n^3 2^J}\bigg\}.
\end{align*}

\medskip
\noindent
{\underline{Upper bound of $\E$(III):}} 
Again by \eqref{eq:5-st25-**} and the inequality $C_2^2\leq C_1$,
\begin{align*}
\text{(III)} &= 16 C_2^2 \exp \Big( -\frac{[y'\vee \frac{C_1}{C_2}]}{4C_2}\Big)&&
\\
&\leq 16 C_1 \exp\Big( -2 \bigg(\frac{\sqrt{\E C_1}}{\sqrt{C_1}} + \frac{\E C_1}{C_1}\bigg) \log (2^{4J+1}) \Big).
\end{align*}
The same arguments as used for $\E$(II) reveal
$$
\E\text{(III)}\lesssim \E C_1 
\cdot \log \Big[ 2^{4J+1}\Big] \cdot \bigg\{ \frac{1}{2^{4J}}+\frac{1}{n^3 2^J}\bigg\}.
$$
Finally, inserting $y=Cy'=C8 \sqrt{\E C_1}\,  \log[2^{4J+1}]$ into $(\Delta_B)$, where $C$ denotes the universal constant of the decoupling inequality of \cite{delaPena95}, 
 we obtain
\begin{flalign*}
\quad  (\Delta_B)
&\le C^3 \E \int_{y'}^\infty t \Pr (X'>t\mid W')d t && \\
&\lesssim \E C_1 \log \Big[ 2^{4J+1}\Big]\bigg\{ \frac{1}{2^{4J}}+\frac{1}{n^3 2^J}\bigg\}.
\end{flalign*}
Therefore, with $y_J:=C8 \sqrt{\E C_1}\,  \log[2^{4J+1}]$ we get the following upper bound 
\begin{align*}
\E [(B_J-y_J)_+^2] & \lesssim \frac{1}{n^2 \alpha^4} \sum_{J \in \mathcal{J}} J^{4a} 2^{3J} \cdot \log[2^{4J+1}]\bigg\{ \frac{1}{2^{4J}}+\frac{1}{n^3 2^J}\bigg\}\\
& \lesssim \frac{1}{n^2 \alpha^4}  J_{max}^{4a+1} \bigg(1 + \frac{2^{2 J_{max}}}{n^3}\bigg).
\end{align*} 
As a consequence, remembering also that $2^{3 J_{\max}} \lesssim n^2\alpha^4$,
$$
\sum_{J\in \mathcal{J}} \E [(B_J-y_J)_+^2]  \lesssim_{\log} \frac{1}{n^2 \alpha^4}.
 $$
 We conclude to the same inequality if we replace $y_J$ with its bound from above $pen_B(J):= 256 \sigma^2 (J+1)^{2a+1} 2^{3J/2}/(n \alpha^2)$ and for the negative part that
$$
\E [(B_J-pen_B(J))_-^2] \leq  \E B_J^2+  pen_B^2(J) \leq 2 pen_B^2(J).
$$

In order to conclude the proof of the theorem we see that
\begin{align*}
\E[(\sup_J V_J)^2] &\lesssim_{log} \frac 1{n\alpha^2} + \inf_{J \in \mathcal{J}}\left(\|f-f_J\|_2^4 +pen^2(J)+pen_A^2(J)+ pen_B^2(J)\right)\\
&\lesssim_{log} \frac 1{n\alpha^2} + \inf_{J \in \mathcal{J}}\left(\|f-f_J\|_2^4 +\frac{2^J}{n \alpha^2} + \frac{2^{3J}}{n^2\alpha^4}\right)
\end{align*}
and the infimum is attained at the minimax rate, up to some logarithmic factors. \hfill \qed


\subsection{Proof of Theorem~6.2 (sequentially interactive protocol)}

Similarly to the non-interactive case, we write
\begin{align*}
    \E \left[(\hat D_n^{(SI)} - D - \widehat{pen}(\hat J)^2\right] 
    &= \E\Big[\sup_{J \leq J_{\max}} (\hat D_J - D - \widehat{pen}(J) )^2\Big] \\
    & \leq \sum_{J \leq J_{\max}} \E[(V_J)_+^2] + \inf_{J \leq J_{\max}} \E [(V_J)_-^2],
\end{align*}
where $V_J:= \hat D_J - D - \widehat{pen}(J)$.

We start with bounding $\sum_{J \leq J_{\max}}\E[(V_J)_+^2]$. To this aim, note first that
\begin{align*}
    \frac{1}{2}\E[(V_J)_+^2] & \leq \E [(\widehat{pen}(J) - pen(J))^2] + \E [(\hat D_J - D - pen(J))_+^2] := T_1+T_2,
\end{align*}
say, with $pen^2(J) = \frac {1}{n\alpha^2} \sum_{j=-1}^{J-1} \sigma_j^2 \|\beta_{j \cdot}\|_2^2$. First
\begin{align*}
    T_1 :&= \E \left[(\widehat{pen}(J) - pen(J))^2\right]\\
    &\leq 
    \E \left|\widehat{pen}^2(J) - pen^2(J)\right| \\
    &\leq \frac 1{n \alpha^2} \sum_{j=-1}^{J-1} \sigma_j^2\E \left|\left(\frac{1}{\arrowvert \mathcal{N}_j\arrowvert}\sum_{i\in\mathcal{N}_j} Z_i^{(2,j)}\right)_+ - \|\beta_{j \cdot }\|_2^2\right|  \\
    & \leq  \frac 1{n \alpha^2} \sum_{j=-1}^{J-1} \sigma_j^2 \left\{\E \left(\E \left|\frac{1}{\arrowvert \mathcal{N}_j\arrowvert}\sum_{i\in\mathcal{N}_j} Z_i^{(2,j)} - \int \left(\sum_k \hat \beta_{jk} \psi_{jk} \right)_\tau f \right| |Z^{(1)}\right) \right.\\
    &\ \ \   + \left. \E\left|\int \left(\sum_k \hat \beta_{jk} \psi_{jk} \right)_\tau f -  \|\beta_{j \cdot }\|_2^2\right| \right\}\\
    & \leq \frac 1{n \alpha^2} \sum_{j=-1}^{J-1} \sigma_j^2 \Bigg\{\E \sqrt{\Var \left(\frac{1}{\arrowvert \mathcal{N}_j\arrowvert}\sum_{i\in\mathcal{N}_j} Z_i^{(2,j)}  \Bigg|Z^{(1)}\right) }\\
    &\ \ \   +  \E\left|\int \Pi_\tau \left(\sum_k \hat \beta_{jk} \psi_{jk} \right) f - \sum_k \hat \beta_{jk} \beta_{jk}\right| + \left|\sum_k \hat \beta_{jk} \beta_{jk}-  \|\beta_{j \cdot }\|_2^2\right| \Bigg\}\\
    &= T_{1,1}+T_{1,2}+T_{1,3}, \text{ say}.
\end{align*}
We need to show that the sum $\sum_{J\leq J_{\max}} T_1(J)$ stays bounded by a sequence not larger than the minimax rate, up to some logarithmic factors. 

The first term in the sum above is bounded from above as follows
$$
T_{1,1}\leq \frac 1{n\alpha^2} \sum_{j=-1}^{J-1} j^{2a} 2^{j} \frac{\tau}{\sqrt{\lfloor n/(2 J_{\max})\rfloor \alpha^2 }} \lesssim \frac{J^{2a} 2^J\sqrt{J_{\max}}}{ (n \alpha^2)^{3/2}}.
$$
Summing this up over $J=1,\dots, J_{\max}$, we see that this gives a rate of order $(n\alpha^2)^{-1}$ up to log-factors.
By an argument analogous to \eqref{eq:var2}, the third term in the sum can be bounded from above using that $\|\beta_{j \cdot}\|^2_2 \lesssim 2^{-2j s'} I(s'<1) + 2^{-J \frac 53} I(s' \geq 1)$, (see Lemma~\ref{bias}), by
\begin{align*}
T_{1,3}  &\leq \frac 1{n \alpha^2} \sum_{j=-1}^{J-1} \sigma_j^2 \Var ^{1/2}  \left(\sum_k \hat \beta_{jk} \beta_{jk}\right) 
  \lesssim \frac 1{n \alpha^2} \sum_{j=-1}^{J-1} \sigma_j^2 \frac{\sigma_j}{ \sqrt{n \alpha^2}} \|\beta_{j \cdot }\|_2\\
  &\lesssim \frac{ 1 \vee( J^{3a} 2^{3J/2 - Js'})}{(n \alpha^2)^{3/2}} \mathds{1}_{0<s'<1}
  +\frac{ 1 \vee( J^{3a} 2^{2J/3})}{(n \alpha^2)^{3/2}} \mathds{1}_{s'\geq 1}.
\end{align*}
Summing this up over $J=1,\dots, J_{\max}$, we obtain the bound 
$$
\frac 1{n \alpha^2} \cdot \left( \frac{J_{\max}^{3a}2^{J_{\max}(3/2-s')}}{(n\alpha^2)^{1/2}}
+ \frac{J_{\max}^{3a}2^{2 J_{\max} /3}}{(n\alpha^2)^{1/2}}
\right) .
$$
Plugging in $2^{J_{\max}} \asymp \sqrt{n\alpha^2}/\log^{B/2}(n\alpha^2)$, we see that in the regime where $s'\in(0,\frac12)$ this is bounded from above by $(n\alpha^2)^{-\frac{4s'}{2s'+1}}$, up to log-factors, while, for $s'>\frac12$, we end up with the parametric rate, again, up to log-factors.
For the second term $T_{1,2}$ in the sum, first recall that $\Pi_\tau(v) - v = -(v-\tau)_+ + (-v-\tau)_+$ and thus
\begin{align*}
    T_{1,2}& ={\frac{1}{n\alpha^2}\sum_{j=-1}^{J-1}\sigma_j^2}\,\E \left| \int_0^1 \Pi_\tau\left(\sum_k \hat \beta_{jk} \psi_{jk} \right) f - \sum_k \hat \beta_{jk} \beta_{jk}
    \right|\\
    &\lesssim {\frac{1}{n\alpha^2}\sum_{j=-1}^{J-1}\sigma_j^2}\,\E \int_0^1 \left(\sum_k \hat \beta_{jk} \psi_{jk}-\tau \right)_+ f \\
    &\ \ \ \ \ \ 
     + {\frac{1}{n\alpha^2}\sum_{j=-1}^{J-1}\sigma_j^2}\,\E \int_0^1 \left(-\sum_k \hat \beta_{jk} \psi_{jk}-\tau \right)_+ f.
\end{align*}
{In order to bound the two expected values within the sums  above, we may replicate the arguments of the proof of Theorem~4.1 with $\hat f^{(1)}_J (x)$ replaced by 
$$ Y_j(x):= \hat f^{(1)}_{j+1} (x) - \hat f^{(1)}_j (x)=\sum_k \hat \beta_{jk} \psi_{jk}(x)
$$ and $\tau = \log^\kappa (n \alpha^2)$ in \eqref{eq:Yx} and further on. Note that  the truncation value $\tau$ is free of $s'$. Indeed, analogously to Proposition~\ref{exponineq},  there exist constants $c_1,\, c_2>0$ such that
$$
|Y_j(x) - \E Y_j(x)| \leq \left[ c_1 \frac{(j+1)^a 2^j}{\sqrt{n\alpha^2}} \sqrt{u} \right] \vee \left[c_2 \frac{(j+1)^a 2^j}{n\alpha} u \right]
$$
with probability larger than $1-4\exp(-u/2)$ for all $u>0$, such that the arguments for bounding  \eqref{eq:Yx} by means of Lemma~\ref{lemma:Integrals} with $a_1=0$, $A_1=\frac{c_1}{2}\frac{n\alpha^2}{(j+1)^{2a}2^{2j}}$, $r_1=2$, $v_1=\frac{\tau}{2}$ and $a_2=0$, $A_2=\frac{c_2}{2}\frac{n\alpha}{(j+1)^a2^j}$, $r_2=1$ and $v_2=(j+1)^a2^j/\alpha$ are applicable. 
We get
\begin{align*}
&    \E \int_0^1 \left(\pm\sum_k \hat \beta_{jk} \psi_{jk}-\tau \right)_+ f 
\leq \E \int_0^1 \left(\pm(Y_j - \E Y_j)-\frac{\tau}2 \right)_+ f\\
& \lesssim \frac{4}{c_1} \frac{2 (j+1)^{2a} 2^{2j}}{\tau n \alpha^2} \exp \left(
-\frac{c_1}{8} \frac{\tau^2 n\alpha^2}{(j+1)^{2a} 2^{2j}}
\right)
+ \frac{4}{c_2} \frac{(j+1)^a 2^j}{n \alpha^2} \exp\left(- \frac{c_2}{2} n\alpha^2 \right).
\end{align*}
This gives 
\begin{align*}
    T_{1,2} & \lesssim \frac{1}{(n \alpha^2)^2} \cdot \left\{ \frac{8}{c_1} \sum_j \frac{1}{\tau}(j+1)^{4a} 2^{3j} \exp (- \frac{c_1}2  \frac{n\alpha^2 (\tau/2)^2}{ J_{max}^{2a} 2^{2 J_{max}}})
    \right.\\
    & \left. + \frac{4}{c_2} \sum_j (j+1)^{3a} 2^{2j} \exp(-\frac{c_2}2 n\alpha^2)
     \right\}.
\end{align*}
By summing up over $J$ we get the upper bound (up to constants), we get
$$
\frac{J_{max}^{4a} 2^{3J_{max}}}{(n \alpha^2)^2 \log^\kappa (n \alpha^2)} \exp(- c \log^{2\kappa +2B - 2a}(n \alpha^2)) \lesssim \frac 1{n \alpha^2},
$$
for some constant $c>0$ and the last inequality holds for $\kappa(a,B)>0$ chosen large enough. 
}

Next, denote by $D_J = \sum_{j=-1}^{J-1} \|\beta_{j \cdot}\|_2^2$ and let us decompose 
\begin{align*}
T_2&:=\E [(\hat D_J - D - pen(J))_+^2] \\
    &\leq
    \E \left[\left(\hat D_J - \sum_{j=-1}^{J-1} \int_0^1 \Pi_\tau  \left(\sum_k \hat \beta_{jk} \psi_{jk} \right) f \right)_+^2 \right] \\
    & + \E\left[ \left( \sum_{j=-1}^{J-1} \int_0^1 \Pi_\tau  \left(\sum_k \hat \beta_{jk} \psi_{jk} \right) f  - \sum_{j=-1}^{J-1} \sum_k \hat \beta_{jk} \beta_{jk}\right)_+^2 \right] \\
    & + \E \left[\left(  \sum_{j=-1}^{J-1} \sum_k \hat \beta_{jk} \beta_{jk} - D_J - (D- D_J )- pen(J)\right)_+^2\right]\\
    & =:T_{2,1}+ T_{2,2} + T_{2,3}, \text{ say}.
\end{align*}
For the first term we write
\begin{align*}
    T_{2,1}& = \E\E \left[\left(\sum_j (\bar{Z}_{\cdot}^{(2,j)} - \E(Z_1^{(2,j)}|Z^{(1)})) \right)^2 | Z^{(1)}\right] \\
    & \lesssim J \cdot \E \left( \frac{\tau^2}{|\mathcal{N}_j|\alpha^2}\right) \lesssim\log^{2 \kappa}(n\alpha^2) \frac{ J}{(n/ J_{max})\alpha^2} \leq \log^{2 \kappa}(n\alpha^2) \frac{ J_{max}^2}{n \alpha^2}.
\end{align*}
Next, we write $\sum_k \hat \beta_{jk} \beta_{jk} = \int_0^1 Y_j f$ and get that
\begin{align*}
    T_{2,2} &\leq J \cdot \sum_j \E \left[\left(\int_0^1 [\Pi_\tau(Y_j) -Y_j ] f   \right)_+^2\right] \\
    &\leq J \cdot \sum_j \E \left[\left(\int_0^1 -[Y_j - \tau ]_+ f + [-Y_j-f]_+ f   \right)^2\right] \\
    & \leq J \cdot \sum_j 2 \cdot \max \E \left[ \left(\int_0^1 (\pm (Y_j - \E Y_j) - \frac {\tau}2)_+^2 f\right)^2 \right]
\end{align*}
and we follow the lines of proof for the non-adaptive case with $\hat f_J^{(1)}$ replaced by $Y_j$ in Section~B.2 and also analogously to the bound for the term $T_{1,2}$ here above to get, for $\kappa(a,B)>0$ large enough,
$$
T_{2,2} \lesssim_{log} \frac{1 }{n \alpha^2}.
$$
Finally,
\begin{align}
    T_{2,3}& \lesssim \int_0^{\infty} u \cdot \Pr \left(\sum_{j=1}^{J-1} \sum_k (\hat \beta_{jk}- \beta_{jk}) \beta_{jk} \geq u + D-D_J + pen_J\right) du \nonumber\\
    & \lesssim \int_{D-D_J + pen_J}^\infty 
    u \cdot \Pr \left(\sum_{j=1}^{J-1} \sum_k (\hat \beta_{jk}- \beta_{jk}) \beta_{jk} \geq u \right) du.\label{T_23}
\end{align}

Here we use directly the concentration from Lemma~\ref{lemma: bound}. We decompose
\begin{align*}
    \sum_{j=1}^{J-1} \sum_k (\hat \beta_{jk}- \beta_{jk}) \beta_{jk} &= \Pr_{n/2} \sum_{j=1}^{J-1} \sum_k (\psi_{jk}(\cdot) - \beta_{jk}) \beta_{jk} \\
&+ \frac 2n \sum_{i=1}^{n/2}\sum_{j=1}^{J-1} \sum_k \sigma_j \frac \sigma \alpha W_{ijk}  \beta_{jk} 
\end{align*}
Analogously to the non-interactive setup, we use the Bernstein inequality in \eqref{eq:Bernstein} with random $\Xi_{ijk}$ replaced by deterministic $\beta_{jk}$. 
The empirical process in the first term of the previous display possesses uniformly in $J$ an exponential tail bound which is smaller than the exponential tail bound for the second part and consequently covered by doubling the constant of the penalty. Indeed, the first term has variance of a smaller order uniformly in $J$ than the second term, and also an upper bound which is uniformly in $J$ smaller than $\sigma/(n\alpha) \max_{jk} \sigma_j|\beta_{jk}|$. We write
$$
\Pr \left(\sum_{j=1}^{J-1} \sum_k (\hat \beta_{jk}- \beta_{jk}) \beta_{jk} \geq u \right) \leq
\exp\left( -\frac{u^2}{2} \frac 1{\frac {\sigma^2}{n \alpha^2} \sum_j \sigma_j^2\|\beta_{j \cdot }\|_2^2  + u \frac \sigma{n \alpha } \max_{jk} \sigma_j |\beta_{jk}| }\right).
$$
Remark that $\max_{jk} \sigma_j |\beta_{jk}| \leq J^a $. We apply Lemma~\ref{lemma: bound} in order to bound the integral in (\ref{T_23}) with $a_1 = pen(J)^2$, $a_2 = J^a/(n\alpha)$ and $y=D-D_J+pen(J)$ to get the upper bound
\begin{align*}
T_{2,3} &\leq 2a_1 \exp\Big(-\frac{y^2}{4a_1}\Big) + 4a_2 \Big[ y \vee \frac{a_1}{a_2}\Big] \exp\Big( -\frac{\big[y \vee \frac{a_1}{a_2}\big]}{4a_2} \Big)\\
&+ 16 a_2^2 \exp\Big(  -\frac{\big[y \vee \frac{a_1}{a_2}\big]}{4a_2}\Big)
\\
& \lesssim pen(J)^2 + a_2^2 \lesssim pen(J)^2.
\end{align*}
By an argument analogous to \eqref{eq:var2}, $pen(J)^2 = (n\alpha^2)^{-1} \sum_{j=-1}^{J-1} \sigma_j^2 \|\beta_{j \cdot}\|_2^2$ can be bounded from above  using that $\|\beta_{j \cdot}\|^2_2 \lesssim 2^{-2j s'} I(s'<1) + 2^{-J \frac 53} I(s' \geq 1)$, (see Lemma~\ref{bias}). After summing up $pen(J)^2$ over $J$ we get the bound 
$$
J_{max} \frac{J_{max}^{2a} \cdot (1 \vee 2^{J_{max}(1-2s')})}{n \alpha^2} \lesssim \frac{J_{max}^{2a+1}}{n \alpha^2} + J_{max}^{2a+1} (n \alpha^2)^{-\frac 12 - s'} \cdot I(s'< \frac 12).
$$
We see that, in case $s'<1/2$, this bound is smaller than $(n \alpha^2)^{-2s'/(4s'+1)}$ and that the bound is up to logarithmic terms smaller than $\mathfrak r_n^{(SI)}(\alpha, a, s')$.

\bigskip

We finish the proof by studying 
\begin{align*}
\inf_{J \leq J_{\max}} \E [(V_J)_-^2] 
& \leq 4 \inf_{J \leq J_{\max}} \big\{ \E [(\hat D_J - D_J)^2]+ (D-D_J)^2 + pen(J)^2 \\
& + \E [(\widehat{pen}(J) - pen(J))^2]  \big\}\\
& \lesssim_{log} \inf_{J \leq J_{\max}} (D-D_J)^2 + pen(J)^2 + T_1,
\end{align*}
where $T_1$ has been defined and bounded from above here above. We conclude that the risk of the adaptive estimator $\hat D^{(SI)}$ is up to some logarithmic factor bounded from above by $\mathfrak r_n^{(SI)}(\alpha,a,s')$. \hfill\qed


\subsection{Auxilliary lemmas}

\begin{lemma}\label{lemma: bound} For any constants $y,a_1,a_2>0$, 
\begin{align*}
\int_{y} ^\infty & t \exp \Big( -\frac{t^2}{2} \frac{1}{a_1+ta_2}\Big) d t\\
& \leq 2a_1 \exp\Big(-\frac{y^2}{4a_1}\Big) + 4a_2 \Big[ y\vee \frac{a_1}{a_2}\Big] \exp\Big( -\frac{\big[y\vee \frac{a_1}{a_2}\big]}{4a_2} \Big)+\\
&\quad\quad
+ 16 a_2^2 \exp\Big(  -\frac{\big[y\vee \frac{a_1}{a_2}\big]}{4a_2}\Big).
\end{align*}
\end{lemma}
\begin{proof}
Because of
$$
\frac{t^2}{2} \, \frac{1}{a_1+a_2t}\geq 
\begin{cases}
\frac{t^2}{4a_1} & \text{ if } t\leq a_1/a_2\\
\frac{t}{4a_2} & \text{ if } t> a_1/a_2,
\end{cases}
$$
the left-hand side is upper bounded by
\begin{align*}
\int_y^\infty t\exp \Big( -\frac{t^2}{2} \, \frac{1}{a_1+a_2t}\Big) d t
&\le \underbrace{\int_y^{(a_1/a_2) \vee y} t \exp\Big( -\frac{t^2}{4a_1}\Big) d t}_{=:\text{A1}} +\\
&\quad\quad
+ \underbrace{\int_{y\vee (a_1/a_2)} ^\infty t \exp\Big(-\frac{t}{4a_2}\Big)  d t}_{=:\text{A2}}.
\end{align*} 
Both integrals can be evaluated explicitly:
\begin{align*}
\text{A1} &=-2a_1\int_y^{(a_1/a_2) \vee y} -\frac{t}{2a_1} \exp\Big( -\frac{t^2}{4a_1}\Big) d t\\
&= -2a_1 \exp\Big(-\frac{t^2}{4a_1}\Big) \,\Big|_y^{(a_1/a_2)\vee y} \\
&= 2a_1 \exp \Big( -\frac{y^2}{4a_1} \Big)
- 2a_1 \exp\Big( -\frac{\big[ \frac{a_1}{a_2} \vee y\big]^2}{4a_1}\Big) 
\end{align*}
and
\begin{align*}
\text{A2}&= -t 4a_2 \exp\Big(-\frac{t}{4a_2}\Big) \,\Big|_{y\vee (a_1/a_2)} ^\infty
+ \int_{y\vee (a_1/a_2)} 4a_2 \exp\Big( -\frac{t}{4a_2}\Big) d t \\
&= 4a_2 \Big[ y\vee\frac{a_1}{a_2}\Big] \exp \Big(-\frac{\big[y\vee \frac{a_1}{a_2}\big]}{4a_2} \Big)
+ 16 a_2^2 \exp \Big(-\frac{\big[y\vee \frac{a_1}{a_2}\big]}{a_2} \Big).
\end{align*}
Dropping the negative summand in the expression for A1 reveals the bound.
\end{proof}

\begin{lemma}[Tail bound of $C_1$]
\label{lemma:C1tail}
For any {$\gamma>n^{-1/3}$}, $n >1$, 
\begin{align*}
\Pr\big(&C_1>\gamma\E C_1\big)
\lesssim 2^{J+2}\exp\Big(-n\min\Big\{\frac{2^J\gamma}{n},\frac{2^{J/2}\sqrt{\gamma}}{\sqrt{n}}\Big\}\Big) + \frac{1}{\gamma^2 n^3 2^J}.
\end{align*}
\end{lemma}

\begin{proof} We decompose $\zeta_{ijk} = \frac{n}{n-1} \frac 1n \sum_{h=1}^n W'_{hjk} - \frac{1}{n-1} W'_{ijk}$ and use that $(a+b)^2 \leq 2 a^2 + 2 b^2$ together with $n/(n-1) \leq 2$. 
Due to the inequality
\begin{align*}
\Pr(C_1>\eta)&\leq \Pr\bigg(\frac{16}{n}\sum_{j,k}\sigma_j^4\frac{\sigma^4}{\alpha^4}\Big[\frac{1}{n}\sum_{i=1}^nW_{ijk}\Big]^2>\eta/2\bigg)\\
& + \Pr\bigg(\frac{4}{n(n-1)^2}\frac{1}{n}\sum_{i=1}^n\sum_{j,k}\sigma_j^4 \frac{\sigma^4}{\alpha^4} W_{ijk}^2>\eta/2\bigg),
\end{align*}
it remains to bound both expressions on the right-hand side. By the union bound and the Bernstein exponential inequality for Laplace i.i.d. random variables derived, e.g., in \cite{BouLugMas13}, we get
\begin{align*}
\Pr\bigg(\frac{1}{n}&\sum_{j,k}\sigma_j^4 \frac{\sigma^4}{\alpha^4}\Big[\frac{1}{n}\sum_{i=1}^nW_{ijk}\Big]^2>\frac{\eta}{32}\bigg)\\
&\leq  2^{J+1}\Pr\Big(\Big\arrowvert \frac{1}{n} \sum_{i=1}^n W_{i 1 1}\Big\arrowvert>\underbrace{\frac{\sqrt{n} \alpha^2}{8\sigma^2} \frac{\sqrt{\eta}}{2^{J}J^{2a}}}_{=:\eta_{n,J}}\Big)\\
&\leq 2^{J+2}\exp\bigg(-n\frac{\eta_{n,J}^2}{2}\frac{1}{2+\eta_{n,J}}\bigg)\\
&\leq  2^{J+2}\exp\Big(-n\min\Big\{\frac{\eta_{n,J}^2}{8},\frac{\eta_{n,J}}{4}\Big\}\Big).
\end{align*}
Concerning the second expression, Chebychev's inequality reveals
\begin{align*}
\Pr\bigg(&\frac{4}{n(n-1)^2}\frac{1}{n}\sum_{i=1}^n\sum_{j,k}\sigma_j^4 \frac{\sigma^4}{\alpha^4} W_{ijk}^2>\eta/2\bigg)\\
&\leq \Pr\bigg(\frac{4}{n(n-1)^2}\frac{1}{n}\sum_{i=1}^n\sum_{j,k}\sigma_j^4 \frac{\sigma^4}{\alpha^4} (W_{ijk}^2-2)>\eta/4\bigg) + \\
&\quad\quad +\mathds{1}\Big\{\frac{8 \sigma^4}{n(n-1)^2\alpha^4}\sum_{j,k} \sigma_j^{4}>\eta/4 \Big\}\\
&\leq \frac{16^2 \sigma^8}{\eta^2n^3 (n-1)^4\alpha^8}2^{5J}J^{8a}\E(W_{ijk}^2 - 2)^2+ \mathds{1}\Big\{\frac{8 \sigma^4}{n(n-1)^2 \alpha^4}2^{3J+1}J^{4a}>\eta/4\Big\}.
\end{align*}
With $\eta=\gamma\E C_1 = \gamma \frac{4 \sigma^4}{n (n-1) \alpha^4 } \sum_{j,k} \sigma_j^2$, the indicator function above is equal to $0$ for any $\gamma > c n^{-1} J_{max}^{2a}2^{J_{max}}$   for some constant $c>0$.  Moreover, we obtain $ \frac{2^{J/2}}{\sqrt{n}} \sqrt{\gamma} \lesssim \eta_{n,J}$.
Therefore, 
\begin{align*}
\Pr(C_1>\gamma \E C_1)
&\lesssim 2^{J+2}\exp\Big(-n\min\Big\{\frac{2^J\gamma}{ n},\frac{2^{J/2}\sqrt{\gamma}}{\sqrt{n}}\Big\}\Big) + \frac{2^{5J}J^{8a}}{\gamma^2 (\E C_1)^2 n^7 \alpha^8}\\
&\lesssim 2^{J+2}\exp\Big(-n\min\Big\{\frac{2^J\gamma}{ n},\frac{2^{J/2}\sqrt{\gamma}}{\sqrt{n}}\Big\}\Big) + \frac{1}{\gamma^2  n^3 2^J}.
\end{align*}
\end{proof}

\end{document}